\documentclass[ps]{imsart}
\RequirePackage[colorlinks,citecolor=blue,urlcolor=blue]{hyperref}
\RequirePackage[OT1]{fontenc}
\usepackage{amssymb,amsmath,amsthm}
\usepackage[english]{babel}
\usepackage[round]{natbib}
\usepackage{booktabs}
\usepackage{caption}
\usepackage{enumitem}
\usepackage{graphicx}
\usepackage[utf8]{inputenc}
\usepackage{makecell}
\usepackage{mathtools}
\usepackage{pdflscape}
\usepackage{tikz}
\usepackage{url}

\usepackage{cleveref}

\usepackage{stevemath}

\usetikzlibrary{shapes,arrows,positioning,arrows.meta}

\doi{10.1214/154957804100000000}
\pubyear{0000}
\volume{0}
\firstpage{0}
\lastpage{0}
\arxiv{1808.03204}

\tikzstyle{block} = [rectangle, draw,
    text width=16em, text centered, rounded corners, minimum height=4em]
\tikzstyle{block2} = [rectangle, draw,
    text width=8em, text centered, rounded corners, minimum height=2.5em]
\tikzstyle{block3} = [rectangle,
    text width=8em, text centered, minimum height=4em]
\tikzstyle{line} = [draw, -implies, double distance=3pt]
\tikzstyle{dashedline} = [draw, dashed]
\tikzstyle{doubleimplies} = [draw, implies-implies, double distance=3pt]
\tikzstyle{arrow} = [draw, -{Latex[length=3mm]}]
\tikzstyle{doublearrow} = [draw, {Latex[length=3mm]}-{Latex[length=3mm]}]
\tikzstyle{noblock} = [rectangle, text width=5em, text centered,
    minimum height=4em]
\tikzstyle{widenoblock} = [rectangle, text width=9em, text centered,
    minimum height=2em]
\tikzstyle{wideblock} = [rectangle, draw,
    text width=14em, text centered, rounded corners, minimum height=4em]

\newcommand{\psistarprime}{{\psi^\star}'}
\newcommand{\slope}{\mathfrak{s}}
\newcommand{\decay}{D}
\newcommand{\gammamax}{\gamma_{\max}}

\newcommand{\subpsiclass}{\mathbb{S}^{l_0}_{\psi}}

\startlocaldefs
\numberwithin{equation}{section}

\theoremstyle{definition}

\theoremstyle{plain}

\crefname{thm}{Theorem}{Theorems}
\crefname{lem}{Lemma}{Lemmas}
\crefname{prop}{Proposition}{Propositions}
\crefname{cor}{Corollary}{Corollaries}
\crefname{assn}{Assumption}{Assumptions}
\crefname{defn}{Definition}{Definitions}
\crefname{rmk}{Remark}{Remarks}
\crefname{exm}{Example}{Examples}
\crefname{fct}{Fact}{Facts}
\endlocaldefs

\begin{document}

\begin{frontmatter}
\title{Time-uniform Chernoff bounds via nonnegative supermartingales}
\runtitle{Time-uniform Chernoff bounds}

\begin{aug}
\author{\fnms{Steven R.} \snm{Howard}\thanksref{onr1}}
\address{Department of Statistics \\
University of California, Berkeley \\
Berkeley, California 94720, USA \\
stevehoward@berkeley.edu
}
\author{\fnms{Aaditya} \snm{Ramdas}}
\address{Departments of Statistics and Machine Learning\\
Carnegie Mellon University \\
Pittsburgh, PA 15213, USA \\
aramdas@stat.cmu.edu
}
\author{\fnms{Jon} \snm{McAuliffe}}
\address{Department of Statistics \\
University of California, Berkeley \\
Berkeley, California 94720, USA \\
jonmcauliffe@berkeley.edu
}
\author{\fnms{Jasjeet} \snm{Sekhon}\thanksref{onr2}}
\address{Departments of Statistics and Political Science \\
University of California, Berkeley \\
Berkeley, California 94720, USA \\
sekhon@berkeley.edu
}
\end{aug}

\runauthor{S. R. Howard et al.}
\thankstext{onr1}{Supported by Office of Naval Research grant N00014-17-1-2176.}
\thankstext{onr2}{Supported by Office of Naval Research grants N00014-15-1-2367, N00014-17-1-2176.}

\begin{abstract}
  We develop a class of exponential bounds for the probability that a
  martingale sequence crosses a time-dependent linear threshold. Our key insight
  is that it is both natural and fruitful to formulate exponential concentration
  inequalities in this way. We illustrate this point by presenting a single
  assumption and theorem that together unify and strengthen many tail
  bounds for martingales, including classical inequalities (1960-80) by
  Bernstein, Bennett, Hoeffding, and Freedman; contemporary inequalities
  (1980-2000) by Shorack and Wellner, Pinelis, Blackwell, van de Geer, and de la
  Peña; and several modern inequalities (post-2000) by Khan, Tropp, Bercu and
  Touati, Delyon, and others. In each of these cases, we give the strongest and
  most general statements to date, quantifying the time-uniform concentration of
  scalar, matrix, and Banach-space-valued martingales, under a variety of
  nonparametric assumptions in discrete and continuous time. In doing so, we
  bridge the gap between existing line-crossing inequalities, the sequential
  probability ratio test, the Cram\'er-Chernoff method, self-normalized
  processes, and other parts of the literature.
\end{abstract}



\tableofcontents
\end{frontmatter}

\section{Introduction}

Concentration inequalities play an important role in probability and statistics,
giving non-asymptotic tail probability bounds for random variables or suprema of
random processes. In this paper, we consider a method to bound the probability
that a martingale ever crosses a time-dependent linear threshold. We were
motivated by the fact that such bounds are the key ingredient in many sequential
inference procedures. We argue, however, that this formulation is materially
better for the development of exponential concentration inequalities, even in
some non-sequential settings. We give a master assumption and theorem which
handle all of these cases, in discrete and continuous time, for scalar-valued,
matrix-valued, and smooth Banach-space-valued martingales.  By unifying and
organizing dozens of results, we illustrate how these results relate to one
another and highlight the specific ingredients contributed by each author. Our
improvements to existing results come in the form of weakened assumptions,
extension of fixed-time or finite-horizon bounds to infinite-horizon uniform
bounds, and improved exponents.

Our main results are presented in full generality in the following section. To
motivate these results, we first contrast a small handful of well-known,
concrete results from the exponential concentration literature; see
\cref{sec:historical} for a more detailed overview of the literature we draw
upon. Throughout the paper, most of our results are presented for filtered
probability spaces, and we use $\E_t$ to denote expectation conditional on the
underlying filtration $\Fcal_t$ at time $t$. For any discrete-time process
$(Y_t)_{t \in \N}$, we write $\Delta Y_t \defineas Y_t - Y_{t-1}$ for the
increments. Finally, we write $\Hcal^d$ for the space of $d \times d$ Hermitian
matrices. The relation $A \preceq B$ denotes the semidefinite order on
$\Hcal^d$, while $\gamma_{\max}: \Hcal^d \to \R$ denotes the maximum eigenvalue
map.

\begin{example}\label{th:intro_example}
  Unless indicated otherwise, let $(S_t)_{t=0}^\infty$ be a real-valued
  martingale with respect to a filtration $(\Fcal_t)_{t=0}^\infty$, with
  $S_0 = 0$.
  \begin{enumerate}
  \item[(a)] Three of the earliest and most well-known results for exponential
    concentration are attributed to Bernstein, Bennett, and Hoeffding. Assume
    the increments $(\Delta S_t)$ are independent, and let
    $v_t \defineas \sum_{i=1}^t \E (\Delta S_i)^2$. We present Bernstein's
    inequality \citep{bernstein_theory_1927} in a widely used form (e.g.,
    \citealp{boucheron_concentration_2013}, Corollary 2.11): if, for some fixed
    $m \in \N$ and $c > 0$, the increments satisfy the moment condition
    $\sum_{i=1}^m \E (\Delta S_t)^k \leq \frac{k!}{2} c^{k-2} v_m$ for all
    integers $k \geq 3$, then for any $x > 0$, we have
    \begin{align}
      \P\eparen{S_m \geq x} \leq \expebrace{-\frac{x^2}{2(v_m + cx)}}.
      \label{eq:bernstein_intro}
    \end{align}
    Bernstein's moment condition is easily seen to be satisfied if the
    increments are bounded. \Citet[eq.\ 8b]{bennett_probability_1962} improved
    Bernstein's result for bounded increments: if $\Delta S_t \leq 1$ for all
    $t$, then for any $x > 0$ and $m \in \N$, we have
    \begin{align}
      \P\eparen{S_m \geq x} \leq \pfrac{v_m}{x + v_m}^{x + v_m} e^x.
      \label{eq:bennett_intro}
    \end{align}
    Finally, \citet[eq.\ 2.3]{hoeffding_probability_1963} gave a simplified
    result for increments bounded from above and below: if
    $\abs{\Delta S_t} \leq 1$ for all $t$, then for any $x > 0$ and
    $m \in \N$, we have
    \begin{align}
      \P\eparen{S_m \geq x} \leq e^{-x^2 / 2m}.
      \label{eq:hoeffding_intro}
    \end{align}
  \item[(b)] \citet[Theorem 1]{blackwell_large_1997}: if
    $\abs{\Delta S_t} \leq 1$ for all $t$, then for any $a, b > 0$, we have
    \begin{align}
      \P(\exists t \in \N: S_t \geq a + bt) \leq e^{-2ab}.
      \label{eq:blackwell_intro}
    \end{align}
    Relative to Hoeffding's inequality, Blackwell removes the assumption of
    independent increments, although this possibility was noted by Hoeffding
    himself \citep[p. 18]{hoeffding_probability_1963}. More importantly,
    Blackwell replaces the event $\brace{S_m \geq x}$ for fixed time $m$ with
    the time-uniform event $\brace{\exists t \in \N: S_t \geq a + bt}$. To see
    that Blackwell's result recovers and strengthens that of Hoeffding, set
    $a=x/2$, $b=x/2m$ and note that Blackwell's uniform bound recovers
    Hoeffding's bound at time $t=m$, so that Blackwell obtains the same
    probability bound for a larger event.
  \item[(c)] \citet[Theorem 1.6]{freedman_tail_1975}: if
    $\abs{\Delta S_t} \leq 1$ for all $t$, then writing
    $V_t \defineas \sum_{i=1}^t \Var\condparen{\Delta S_i}{\Fcal_{i-1}}$, then
    for any $x, m > 0$, we have
    \begin{align}
      \P\eparen{\exists t \in \N: V_t \leq m \text{ and } S_t \geq x}
        \leq \pfrac{m}{x+m}^{x+m} e^x.
      \label{eq:freedman_intro}
    \end{align}
    Similar to Bernstein's and Bennett's inequalities, but unlike those of
    Hoeffding and Blackwell, Freedman's inequality measures time in terms of a
    predictable quantity, the accumulated conditional variance $V_t$, rather
    than simply the number of observations $t$. Freedman's inequality bounds the
    deviations of $(S_t)$ uniformly over time, but only up to the finite time
    horizon implied by $V_t \leq m$.
  \item[(d)] \citet[Theorem 6.2, eq.\ 6.4]{de_la_pena_general_1999}: if the
    increments are conditionally symmetric, that is,
    $\Delta S_t \sim -\Delta S_t \mid \Fcal_{t-1}$ for all $t$, then letting
    $V_t = \sum_{i=1}^t \Delta S_i^2$, for any $\alpha \geq 0$ and
    $\beta, x, m > 0$ we have
    \begin{align}
      \P\eparen{\exists t \in \N: V_t \geq m
        \text{ and } \frac{S_t}{\alpha + \beta V_t} \geq x}
      \leq \expebrace{-x^2\eparen{\frac{\beta^2}{2m} + \alpha \beta}}.
      \label{eq:pena_intro}
    \end{align}
    A remarkable feature of this result is that we measure time via the adapted
    quantity $V_t$. Unlike Freedman's inequality, which uses the true
    conditional variance to measure time, de la Pe\~na's inequality relies only
    on empirical quantities. In further contrast to Freedman's inequality, de la
    Pe\~na's bound holds uniformly over $V_t \geq m$ rather than $V_t \leq m$,
    and we bound the deviations of the self-normalized process
    $S_t / (\alpha + \beta V_t)$.
  \item[(e)] \citet[Theorem 6.2]{tropp_user-friendly_2012}: departing from the
    above results for real-valued martingales, here we begin with a martingale
    $(Y_t)_{t \in \N}$ taking values in $\Hcal^d$. Assume that the increments
    $\Delta Y_t$ are independent and, for some fixed $c > 0$ and
    $\Hcal^d$-valued sequence $(W_t)_{t \in \N}$, the moments of the increments
    satisfy
    $\E\condparen{\Delta S_t^k}{\Fcal_{t-1}} \preceq \frac{k!}{2} c^{k-2} \Delta
    W_t$ for all $t$ and all $k \geq 2$. Then, writing
    $S_t = \gamma_{\max}(Y_t)$ and $V_t = \gamma_{\max}(W_t)$, for any
    $x > 0$ and $t \geq 1$, we have
    \begin{align}
      \P\eparen{S_t \geq x} \leq d \cdot \expebrace{-\frac{x^2}{2(V_t + cx)}}.
      \label{eq:tropp_intro}
    \end{align}
    This elegant result extends Bernstein's inequality to the matrix
    setting. Note the prefactor of $d$ that appears when we bound the deviations
    of the maximum eigenvalue of a $d \times d$ matrix-valued process.
  \item[(f)] Finally, we recall a textbook result for Brownian motion (e.g.,
    \citealp{durrett_probability:_2017}, Exercise 7.5.2): if
    $(S_t)_{t \in (0, \infty)}$ is a standard Brownian motion, then for any
    $a, b > 0$, we have
    \begin{align}
      \P(\exists t \in (0, \infty): S_t \geq a + bt) = e^{-2ab}.
    \end{align}
    The result closely resembles Blackwell's inequality for discrete-time
    martingales with bounded increments, but here we have an equality.
  \end{enumerate}
\end{example}

Clearly, these results have much in common with each other and with myriad other
results from the exponential concentration literature. Examining the proofs, we
find many shared ingredients which are now well known: the notions of
sub-Gaussian and sub-exponential random variables, the Cram\'er-Chernoff method,
the large-deviations supermartingale, and so on. Nonetheless, there are enough
differences among the results and their proofs to leave us wondering whether
these results are merely similar in appearance, or whether they are all special
cases of some underlying, general argument.

In this paper, we provide a framework that formally unifies the above results
along with many others. Our framework consists of two pieces. First, we
crystallize the notion of a \emph{sub-$\psi$ process}
(\cref{th:canonical_assumption}), a sufficient condition general enough to
encompass a broad set of results not previously treated together, yet specific
enough to derive a useful set of equivalent concentration inequalities. This
definition provides a convenient categorization of exponential concentration
results into sub-Bernoulli, sub-Gaussian, sub-Poisson, sub-exponential, and
sub-gamma bounds. Second, we give a generalization of the Cram\'er-Chernoff
argument, \cref{th:uniform_chernoff}. This result yields strengthened versions
of many existing inequalities and illustrates equivalences among different forms
of exponential bounds. For example, \cref{th:uniform_chernoff} strengthens both
``Freedman-style'' inequalities such as \eqref{eq:freedman_intro} and ``de la
Pe\~na-style'' inequalities such as \eqref{eq:pena_intro} to hold uniformly over
all time, and in these strengthened forms, the two styles of inequalities are
shown to be equivalent, as depicted in \cref{fig:equivalence}. We remark that
the seminal works from which these examples are drawn, like others referenced
below, include many other important contributions, and our claims about
\cref{th:uniform_chernoff} refer only to the particular inequalities cited from
each work.

\begin{figure}[h!]
  \centering
  \begin{tikzpicture}[node distance = 7.5cm, auto]
    \node[wideblock] (thm1b)
      {This paper's \Cref{th:uniform_chernoff}(b)};
    \node[wideblock, right of=thm1b] (thm1c)
      {This paper's \Cref{th:uniform_chernoff}(c,d)};
    \node[wideblock, below of=thm1b, node distance=2.5cm] (freedman)
      {Freedman-style inequalities,\\ such as \eqref{eq:freedman_intro}};
    \node[wideblock, below of=thm1c, node distance=2.5cm] (dlP)
      {de la Pe\~na-style inequalities,\\ such as \eqref{eq:pena_intro}};
    \path [line] (thm1b) -- node[noblock, left] {implies}
      (freedman);
    \path [line] (thm1c) -- node[noblock, right] {implies} (dlP);
    \path [doubleimplies] (thm1b) --
      node[widenoblock, above] {imply each other} (thm1c);
    \path [dashedline] (freedman) --
      node[widenoblock, above] {do not imply each other} (dlP);
  \end{tikzpicture}
  \caption{This paper's \cref{th:uniform_chernoff} implies both Freedman-style
    inequalities such as \eqref{eq:freedman_intro} and de la Pe\~na-style
    inequalities such as \eqref{eq:pena_intro}. Refer also to
    \cref{fig:three_bounds,fig:ratio_constant} for visualizations of these
    implications. \label{fig:equivalence}}
\end{figure}
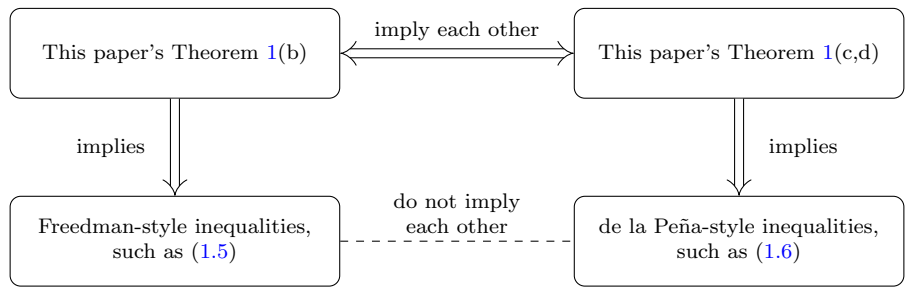

Once the framework is in place, the proof of the main result follows using tools
from classical large-deviation theory \citep{dembo_large_2010}. We construct a
nonnegative supermartingale as in \citet{freedman_tail_1975}, and we obtain a
bound on its entire trajectory using Ville's maximal inequality
\citep{ville_etude_1939}. We invoke Tropp's ideas \citep{tropp_freedmans_2011}
to extend the results to the matrix setting. The equivalences that follow from
optimizing linear bounds are obtained using convex analysis
\citep{rockafellar_convex_1970}.  By drawing together various proof ingredients
from different sources, we elucidate previously unrecognized or understated connections. For
example, we demonstrate how self-normalized matrix inequalities follow easily upon
combining ideas from the literature on self-normalized processes with those from
matrix concentration.

\subsection{Paper organization}

\Cref{sec:main_results} lays out our framework for exponential line-crossing
inequalities. Specifically, we formally state \cref{th:canonical_assumption} and
\cref{th:uniform_chernoff} that together describe a general formulation of the
Cram\'er-Chernoff method.  After stating \Cref{th:uniform_chernoff},
we give a quick overview of existing results which can be recovered in our
framework and the improvements thus obtained. A short proof of our master
theorem comes next, and following some remarks, we provide three 
illustrative examples.

\Cref{sec:sufficient_conditions,sec:special_cases} are devoted to a catalog of
important results from the literature which fit
into our framework, often yielding results which are stronger than those
originally published. In \cref{sec:sufficient_conditions}, we consider the
maximum-eigenvalue process of a matrix-valued martingale and enumerate useful
sufficient conditions for such a process to be sub-$\psi$, collecting and in
some cases generalizing a variety of ingenious results from the
literature. \Cref{sec:special_cases} examines various instantiations of our
master theorem, obtaining corollaries by combining one of the sufficient
conditions from \cref{sec:sufficient_conditions} with one of the four equivalent
conclusions of \cref{th:uniform_chernoff}. These illustrate how our framework
recovers and strengthens existing exponential concentration results. We discuss
sharpness, another geometrical insight, and future work in
\cref{sec:discussion}. Proofs of most results are in \cref{sec:proofs}.

\subsection{Historical context}\label{sec:historical}

To aid the reader, we give here some historical context for the existing results
discussed below. This is not intended to be a comprehensive history of the
literature on exponential concentration, and we focus on the specific results
discussed in \cref{sec:special_cases}, giving pointers to further references as
appropriate.

The Cram\'er-Chernoff method takes its name from the works of
\citet{cramer_sur_1938} and \citet{chernoff_measure_1952}. Both of these authors
were concerned with a precise characterization of the asymptotic decay of tail
probabilities beyond the regime in which the central limit theorem applies;
Cram\'er provided the first proof of such a ``large deviation principle'', while
Chernoff gave a more general formulation and placed more emphasis on the
non-asymptotic upper bound which is our focus. These results spawned a vast
literature on large deviation principles, with the goal of giving sharp upper
and lower bounds on the limiting exponential decay of certain probabilities
under a sequence of measures; see \citet{dembo_large_2010} for an excellent
presentation of this literature. Our focus, on non-asymptotic upper bounds for
nonparametric classes of distributions, is rather different, though such upper
bounds often make an appearance in proofs of large deviation principles.

Bernstein was perhaps the earliest proponent of the sort of exponential tail
bounds that are the focus of this paper, having proposed his famous inequality
in 1911, according to \citet{prokhorov_bernstein_1995}; see also
\citet{craig_tchebychef_1933}, \citet[ch.\ 10, ex.\ 12-14, pp.\
204-205]{uspensky_introduction_1937} and \citet{bernstein_theory_1927}, though
the last source appears rather inaccessible. The modern theory of exponential
concentration began to take shape in the 1960's, as (using the terminology of
this paper, from \cref{sec:sufficient_conditions})
\citet{bennett_probability_1962} improved Bernstein's sub-gamma inequality to
sub-Bernoulli and sub-Poisson ones for random variables bounded from
above. \citet{hoeffding_probability_1963} gave alternative sub-Bernoulli and
sub-Gaussian bounds for random variables bounded from both above and below. For
further references on this line of work, see
\citet{boucheron_concentration_2013}, whose treatment of the Cram\'er-Chernoff
method has been invaluable in formulating our own framework, as well as
\citet{mcdiarmid_concentration_1998}.

\citet[p.\ 936]{godwin_generalizations_1955} reports that Bernstein generalized
his inequality to dependent random variables. \citet[pp.\
17-18]{hoeffding_probability_1963} considered the generalization of his
sub-Bernoulli and sub-Gaussian bounds to martingales and the possibility of
finite-horizon uniform inequalities based on Doob's maximal inequality; the
martingale generalization was later explored by
\citet{azuma_weighted_1967}. \citet{freedman_tail_1975} extended Bennett's
sub-Poisson bound to martingales, giving a uniform bound subject to a maximum
value of the predictable quadratic variation of the martingale. This
``Freedman-style'' bound has been generalized to other settings in many
subsequent works \citep{de_la_pena_general_1999, khan_$l_p$-version_2009,
  tropp_freedmans_2011,
  fan_exponential_2015}. \Citet{chen_statistical_2012,chen_new_2012} has
considered the extension of Chernoff-style bounds to hold uniformly over time
for scalar-valued martingales in a manner similar to our line-crossing
inequalities, including a condition similar to our sub-$\psi$ definition; our
formulation further encompasses matrix-valued processes and self-normalized
inequalities.

The extension of these methods to matrix-valued processes, via control of the
matrix moment-generating function, originated with
\citet{ahlswede_strong_2002}. The method was refined by
\citet{christofides_expansion_2007}, \citet{oliveira_spectrum_2010,
  oliveira_sums_2010} and then by \citet{tropp_freedmans_2011,
  tropp_user-friendly_2012}, whose influential treatment synthesized and
improved upon past work, generalizing many scalar exponential inequalities to
operator-norm inequalities for matrix martingales. We have incorporated Tropp's
formulation into our framework, and we focus on his theorem statements for our
matrix bound statements. See \citet{tropp_introduction_2015} for a recent
exposition and further references.

There is a long history of investigation of the concentration of Student's
$t$-statistic under non-normal sampling. \citet{efron_students_1969} gives many
references to early work. He also shows, by making use of Hoeffding's
sub-Gaussian bound, that the equivalent self-normalized statistic
$\eparen{\sum_i X_i} / \sqrt{\sum_i X_i^2}$ satisfies a 1-sub-Gaussian tail
bound whenever the $X_i$ satisfy a symmetry condition, a result he attributes to
Bahadur and Eaton \citep[p.\ 1284]{efron_students_1969}. Starting with
\citet{logan_limit_1973}, there has been a great deal of work on limiting
distributions and large deviation principles for self-normalized statistics; see
\citet{shao_self-normalized_1997} and references therein.  In terms of
exponential tail bounds, \citet{de_la_pena_general_1999} explored general
conditions for bounding the deviations of a martingale, introduced new
decoupling techniques (cf. \citealp{de_la_pena_decoupling_1999}), and showed
that any martingale with conditionally symmetric increments satisfies a
self-normalized sub-Gaussian bound with no integrability condition. This work
laid the foundation for the type of self-normalized exponential inequalities
which we explore in this paper. These methods were extended by
\citet{de_la_pena_moment_2000, de_la_pena_self-normalized_2004}, which
introduced a general supermartingale ``canonical assumption'' that is a key
precursor of our sub-$\psi$ condition, and initiated a flurry of subsequent
activity on self-normalized exponential inequalities
(cf. \citealp{de_la_pena_pseudo-maximization_2007, de_la_pena_theory_2009}). We
note in particular inequality (3.9) of \citet{de_la_pena_self-normalized_2001},
which gives an infinite-horizon boundary-crossing inequality based on a mixture
extension of their canonical assumption, as well as the multivariate
inequalities (3.24) (for a $t$-statistic) and (3.29) (for general mixture
boundaries) given by
\citet{de_la_pena_theory_2009}. \Citet{bercu_exponential_2008} gave a
self-normalized sub-Gaussian bound without symmetry by incorporating the
conditional quadratic variation, requiring only finite second moments, and some
ingenious further extensions have been given by \citet{delyon_exponential_2009},
\citet{fan_exponential_2015}, and \citet{bercu_concentration_2015}, many of
which we include in our collection of sufficient conditions for a process to be
sub-$\psi$ (\cref{sec:sub_psi_conds}). See
\citet{de_la_pena_self-normalized_2009} and \citet{bercu_concentration_2015} for
further references.

Ville's maximal inequality for nonnegative supermartingales, the technical
underpinning of \cref{th:uniform_chernoff}, originates with
\citet[p. 101]{ville_etude_1939}. It is commonly attributed to Doob, though Doob
acknowledged Ville's priority extensively in his works, e.g.,
\citet[pp. 458-460]{doob_regularity_1940}. \citet{mazliak_splendors_2009}
contains further historical discussion and sources.

\section{Main results}\label{sec:main_results}

Let $(S_t)_{t \in \Tcal \union \brace{0}}$ be a real-valued process adapted to
an underlying filtration $(\Fcal_t)_{t \in \Tcal \cup \brace{0}}$, where either
$\Tcal = \N$ for discrete-time processes or $\Tcal = (0, \infty)$ for
continuous-time processes. In continuous time, we assume $(\Fcal_t)$ satisfies
the ``usual hypotheses'', namely, that it is right-continuous and complete, and
we assume $(S_t)$ is c\`adl\`ag; see, e.g., \citet{protter_stochastic_2005}.  In
a statistical setting, we may think of $(S_t)$ as a summary statistic
accumulating over time, for example a cumulative sum of observations, whose
deviations from zero we would like to bound under some null hypothesis. In this
setting, a bound on the deviations of $(S_t)$ holding uniformly over time can be
used to construct an appropriate sequential hypothesis test, a special case of
which is Wald's sequential probability ratio test discussed in
\cref{sec:expo_family}. We first explain our key condition on $(S_t)$, the
sub-$\psi$ condition. We then state, prove, and interpret our master theorem,
followed by some more detailed examples of its application.

\subsection{The sub-$\psi$ condition}

Our key condition on $(S_t)$ is stated in terms of two additional objects. The
first object is a real-valued, nondecreasing process
$(V_t)_{t \in \Tcal \union \brace{0}}$, also adapted to $(\Fcal_t)$ (and
c\`adl\`ag in the continuous-time case). It is an ``accumulated variance'' process
which serves as a measure of \emph{intrinsic time}, an appropriate quantity to
control the deviations of $S_t$ from zero \citep{blackwell_amount_1973}. The
second object is a function $\psi: \R_{\geq 0} \to \R$, reminiscent of a
cumulant-generating function, which quantifies the relationship between $S_t$
and $V_t$. The simplest case is when $S_t$ is a cumulative sum of i.i.d.,
real-valued, mean-zero random variables with distribution $F$, in which case we
take $V_t = t$ and let $\psi(\lambda) = \log \int e^{\lambda x} \d F(x)$ be the
CGF of $F$. Our key condition requires that $S_t$ is unlikely to grow too
quickly relative to intrinsic time $V_t$; it generalizes developments from
\citet{freedman_tail_1975, de_la_pena_self-normalized_2004,
  tropp_freedmans_2011, chen_new_2012}, and others.

\begin{definition}[Sub-$\psi$ process]\label{th:canonical_assumption}
  Let $(S_t)_{t \in \Tcal \union \brace{0}}$ and
  $(V_t)_{t \in \Tcal \union \brace{0}}$ be two real-valued processes adapted to
  an underlying filtration $(\Fcal_t)_{t \in \Tcal \union \brace{0}}$ with
  $S_0 = V_0 = 0$ a.s.\ and $V_t \geq 0$ a.s.\ for all $t \in \Tcal$. For a
  function $\psi: [0, \lambda_{\max}) \to \R$ and a scalar
  $l_0 \in [1, \infty)$, we say $(S_t)$ is \emph{$l_0$-sub-$\psi$ with variance
    process $(V_t)$} if, for each $\lambda \in [0, \lambda_{\max})$, there
  exists a supermartingale $(L_t(\lambda))_{t \in \Tcal \union \brace{0}}$ with
  respect to $(\Fcal_t)$ such that $L_0(\lambda) \leq l_0$ a.s.\ and
  \begin{align}
    \expebrace{\lambda S_t - \psi(\lambda) V_t } \leq L_t(\lambda)
    \text{ a.s.\ for all } t \in \Tcal.
  \end{align}
  For given $\psi$ and $l_0$, we write $\subpsiclass$ for the class of pairs of
  $l_0$-sub-$\psi$ processes $(S_t, V_t)$:
  \begin{align}
    \subpsiclass \defineas \ebrace{
      (S_t, V_t):
      (S_t) \text{ is $l_0$-sub-$\psi$ with variance process } (V_t)}.
  \end{align}
\end{definition}

We often say simply that a process is sub-$\psi$, omitting $l_0$ from our
terminology for simplicity. All examples considered in this paper fit into three
cases for the value of $l_0$: $l_0 = 1$, when deriving one-sided bounds on
scalar martingales; $l_0 = 2$, when deriving bounds on the norm of certain
Banach-space-valued martingales; or $l_0 = d$, when deriving bounds on the
maximum-eigenvalue process of a $d \times d$ matrix-valued martingale. Also,
though we often speak of a process $(S_t)$ being
sub-$\psi$, the sub-$\psi$ condition formally applies to the pair $(S_t,V_t)$
and not to the process $(S_t)$ alone, so that meaningful statements are always
made in the context of a specific intrinsic time process $(V_t)$.

\Cref{th:canonical_assumption} may at first defy intuition. We
can motivate it from several angles:
\begin{itemize}
\item Suppose $S_t$ is a scalar-valued martingale whose deviations we wish to
  bound uniformly over time. We might like to apply Ville's maximal inequality
  (see \cref{sec:proof_uniform_chernoff}), but must first transform $S_t$ into a
  \emph{nonnegative} supermartingale. It is natural to consider the exponential
  transform $e^{\lambda S_t}$ for some $\lambda > 0$, which immediately yields a
  submartingale. Our task, then, is to find some appropriate $\psi$ and $(V_t)$
  which ``pull down'' the submartingale so that the process
  $\expebrace{\lambda S_t - \psi(\lambda) V_t}$ is a
  supermartingale. Intuitively, the exponential process
  $\expebrace{\lambda S_t - \psi(\lambda) V_t}$ measures how quickly $S_t$ has
  grown relative to intrinsic time $V_t$, and the free parameter $\lambda$
  determines the relative emphasis placed on the tails of the distribution of
  $S_t$, i.e., on the higher moments. Larger values of $\lambda$ exaggerate
  larger movements in $S_t$, and $\psi$ captures how much we must
  correspondingly exaggerate $V_t$.
\item Consider again the simple case in which $S_t$ is a cumulative sum of
  i.i.d. draws from a distribution $F$ over the reals with mean zero and CGF
  $\psi(\lambda) < \infty$ for $\lambda \in [0, \lambda_{\max})$. Then, setting
  $V_t = t$, we may take $L_t(\lambda)$ equal to the exponential process
  $\expebrace{\lambda S_t - \psi(\lambda) t}$, which is a martingale in this
  case, so that the defining inequality of \cref{th:canonical_assumption} is an
  equality. The exponential process may be interpreted as the likelihood ratio
  in an exponential family containing $F$ with sufficient statistic
  $S_t$. See \cref{ex:coin_flipping} for a more detailed exposition of this
  setting and \cref{sec:expo_family} for more on the connection with exponential
  families.
\item Alternatively, we may begin from the martingale method for concentration
  inequalities (\citealp{hoeffding_probability_1963};
  \citealp{azuma_weighted_1967}; \citealp{mcdiarmid_concentration_1998};
  \citealp{raginsky_concentration_2012}, section 2.2), itself based on the
  classical Cram\'er-Chernoff method (\citealp{cramer_sur_1938};
  \citealp{chernoff_measure_1952}; \citealp{boucheron_concentration_2013},
  section 2.2). The martingale method starts from an assumption such as
  $\E\condparen{e^{\lambda (X_t - \E\condparen{X_t}{\Fcal_{t-1}})}}{\Fcal_{t-1}}
  \leq e^{\psi(\lambda) \sigma_t^2}$ for all $t \geq 1$ and
  $\lambda \in [0, \lambda_{\max})$. When $\psi(\lambda) = \lambda^2 / 2$ and
  $\lambda_{\max} = \infty$ (and the condition holds for $\lambda < 0$ as well),
  this is the definition of a conditionally sub-Gaussian random variable with
  variance parameter $\sigma_t^2$. When
  $\psi(\lambda) = \lambda^2 / (2 (1 - c \lambda))$ and $\lambda_{\max} = 1/c$,
  we have the definition of a random variable which is conditionally sub-gamma
  on the right tail with variance parameter $\sigma_t^2$ and scale parameter $c$
  \citep{boucheron_concentration_2013}. Writing
  $S_t \defineas \sum_{i=1}^t (X_i - \E_{i-1} X_i)$ and
  $V_t \defineas \sum_{i=1}^t \sigma_i^2$, the process
  $\expebrace{\lambda S_t - \psi(\lambda) V_t}$ is then a supermartingale for
  each $\lambda \in \R$. For example, if $\Delta S_t \in [a_t, b_t]$ for all
  $t$, then $(S_t)$ is 1-sub-$\psi$ with $\psi(\lambda) = \lambda^2 / 2$ on
  $\lambda \in [0, \infty)$, and $V_t = \sum_{i=1}^t \pfrac{b-a}{2}^2$; this
  fact underlies \cref{th:intro_example}(a,b). Or, if $S_t \leq 1$ for all $t$,
  then $(S_t)$ is 1-sub-$\psi$ with $\psi(\lambda) = e^\lambda - \lambda - 1$ on
  $\lambda \in [0, \infty)$, a fact which leads to \cref{th:intro_example}(c).
\item Unlike the martingale method assumption, \cref{th:canonical_assumption}
  allows $(V_t)$ to be adapted rather than predictable, which leads to a variety
  of self-normalized inequalities \citep{de_la_pena_general_1999,
    de_la_pena_self-normalized_2004, de_la_pena_self-normalized_2009,
    bercu_concentration_2015, fan_exponential_2015}, for example yielding bounds
  on the deviation of a martingale in terms of its quadratic variation. In this
  context, \cref{th:canonical_assumption} is closely related to the ``canonical
  assumption'' of \citet[eq.\ 1.6]{de_la_pena_self-normalized_2004}, which
  requires that $\expebrace{\lambda S_t - \Phi(\lambda V_t)}$ is a
  supermartingale for certain nonnegative, strictly convex functions $\Phi$. We
  have found it more useful to separate the second term into
  $\psi(\lambda) V_t$, though both formulations yield interesting results. For
  example, if $\Delta S_t \sim -\Delta S_t \mid \Fcal_{t-1}$, then $(S_t)$ is
  1-sub-$\psi$ with $\psi(\lambda) = \lambda^2/2$ over
  $\lambda \in [0, \infty)$, and $V_t = \sum_{i=1}^t \Delta S_t^2$, from which
  we may obtain \cref{th:intro_example}(d).
\item Also in contrast to \cite{de_la_pena_self-normalized_2004}, we allow the
  exponential process to be merely upper bounded by a supermartingale, rather
  than being a supermartingale itself; this permits us to handle bounds on the
  maximum eigenvalue process of a matrix-valued martingale, using techniques from
  \citet{tropp_freedmans_2011}. For example, under the conditions of
  \cref{th:intro_example}(e), the maximum eigenvalue process $(S_t)$ is
  $d$-sub-$\psi$ with $\psi(\lambda) = \lambda^2 / [2(1 - c \lambda)]$ on
  $\lambda \in [0, 1/c)$. In this case, the exponential process
  $\expebrace{\lambda S_t - \psi(\lambda) V_t}$ is not a supermartingale, but is
  upper bounded by the trace-exponential supermartingale
  $\tr \expebrace{\lambda Y_t - \psi(\lambda) W_t}$. The initial value of this
  trace-exponential process is $l_0 = d$, which leads to the pre-factor of $d$
  in the bound \eqref{eq:tropp_intro}.
\end{itemize}

\Cref{sec:sufficient_conditions} collects a variety of sufficient conditions
from the literature for a process to be sub-$\psi$, including all of the
examples given above. These conditions illustrate the broad applicability of
\cref{th:canonical_assumption} in nonparametric settings, i.e., those which
restrict the distribution of $(S_t)$ to some infinite-dimensional class, for
example all processes with bounded increments, or with increments having finite
variance. Even in such nonparametric cases, $\psi$ is still a CGF of some
distribution in all of our examples, though this is not required for the most
basic conclusion of \cref{th:uniform_chernoff}. Indeed, the full force of
\cref{th:uniform_chernoff} comes into effect only when $\psi$ satisfies certain
properties which hold for CGFs of zero-mean, non-constant random variables
\citep[Theorem 2.3]{jorgensen_theory_1997}:

\begin{definition}\label{th:cgf_like}
  A real-valued function $\psi$ with domain $[0, \lambda_{\max})$ is called
  \emph{CGF-like} if it is strictly convex and twice continuously differentiable
  with $\psi(0) = \psi'(0_+) = 0$ and
  $\sup_{\lambda \in [0, \lambda_{\max})} \psi(\lambda) = \infty$. For such a
  function we define
  \begin{align}
    \bar{b} = \bar{b}(\psi) \defineas
    \sup_{\lambda \in [0, \lambda_{\max})} \psi'(\lambda) \in (0, \infty].
  \end{align}
\end{definition}

In many typical cases we have $\lambda_{\max} = \infty$ and
$\bar{b}=\infty$. With \cref{th:canonical_assumption,th:cgf_like} in place, we
are ready to set up and state our main result in the following section.

\subsection{The master theorem}

To state our main theorem on general exponential line-crossing inequalities, we
will make use of the following transforms of $\psi$:
\begin{align}
&\text{The Legendre-Fenchel transform: }
\psi^\star(u)
  \defineas \sup_{\lambda \in [0, \lambda_{\max})} [\lambda u - \psi(\lambda)],
  \quad \text{for } u \geq 0. \\
&\text{The ``decay'' transform: }
\decay(u) \defineas \sup\ebrace{
  \lambda \in (0, \lambda_{\max}) : \frac{\psi(\lambda)}{\lambda} \leq u
}, \quad \text{for } u \geq 0. \\
&\text{The ``slope'' transform: }
\slope(u) \defineas \frac{\psi(\psistarprime(u))}{\psistarprime(u)},
  \quad \text{for } u \in (0, \bar{b}).
\end{align}
In the definition of $\decay(u)$, we take the supremum of the empty set to equal
zero instead of the usual $-\infty$. For $u > 0$, this case can arise in
general, but not when $\psi$ is CGF-like. Note that $\decay(u)$ can also be
infinite. We call $\decay(u)$ the ``decay'' transform because it determines the
rate of exponential decay of the upcrossing probability bound in
\cref{th:uniform_chernoff}(a) below. We call $\slope(u)$ the ``slope'' transform
because it gives the slope of the linear boundary in
\cref{th:uniform_chernoff}(b); this is defined only when $\psi$ is
CGF-like. Defining $\slope(0) = 0$ and $\slope(\bar{b}) = \bar{b}$ when
$\bar{b} < \infty$, we find that $\slope(u)$ is continuous, strictly increasing,
and $0 \leq \slope(u) < u$ on $u \in [0, \bar{b})$ (see
\cref{th:slope_facts}). We do not know of other references for the slope transform, or other situations where it arises naturally. \Cref{tab:psi_transforms} gives examples of these
transforms for some common $\psi$ functions.

Our main theorem has four parts, each of which facilitates comparisons with a
particular related literature, as we discuss in \cref{sec:special_cases}. Recall
\cref{th:canonical_assumption} of the class $\subpsiclass$ of $l_0$-sub-$\psi$
processes, and the underlying filtration $(\Fcal_t)$ to which processes $(S_t)$
and $(V_t)$ are adapted.

\begin{theorem}\label{th:uniform_chernoff}
  Fix any $\psi: \R_{\geq 0} \to \R$ and $l_0 \in [1, \infty)$.
\begin{enumerate}[label=(\alph*)]
\item For any $a, b > 0$, we have
  \begin{align}
    \sup_{(S_t,V_t) \in \subpsiclass}
    \P\condparen{\exists t \in \Tcal: S_t \geq a + b V_t}{\Fcal_0}
      \leq l_0 \expebrace{-a \decay(b)}.
  \end{align}
\end{enumerate}
Additionally, whenever $\psi$ is CGF-like, the following three statements are
equivalent to statement (a).
\begin{enumerate}
\item[(b)] For any $m > 0$ and $x \in (0, m\bar{b})$, we have
  \begin{multline}
    \sup_{(S_t,V_t) \in \subpsiclass}
    \P\condparen{
      \exists t \in \Tcal :
      S_t \geq x + \slope\pfrac{x}{m} \cdot \eparen{V_t - m}
    }{\Fcal_0} \\ \leq l_0 \expebrace{-m \psi^\star\pfrac{x}{m}}.
  \end{multline}
\item[(c)] For any $m > 0$ and $x \in (0, \bar{b})$, we have
  \begin{multline}
    \sup_{(S_t,V_t) \in \subpsiclass}
    \P\condparen{
      \exists t \in \Tcal :
      \frac{S_t}{V_t} \geq x - \pfrac{x - \slope(x)}{V_t} \cdot (V_t - m)
    }{\Fcal_0} \\ \leq l_0 \expebrace{-m \psi^\star(x)}.
  \end{multline}
\item[(d)] For any $m \geq 0$, $x > 0$ and $b > 0$, we have (below we take
  $m\bar{b} = \infty$ whenever $\bar{b} = \infty$)
  \begin{multline}
    \sup_{(S_t,V_t) \in \subpsiclass}
    \P\condparen{
      \exists t \in \Tcal: V_t \geq m \text{ and } S_t \geq x + b (V_t - m)
    }{\Fcal_0}
    \\ \leq \begin{cases}
      l_0 \expebrace{-(x - (b \bmin \bar{b}) m) \decay(b)},
        & x > m \bar{b} \text{ or } \slope\eparen{\tfrac{x}{m}} > b \\
      l_0 \expebrace{-m \psi^\star\eparen{\tfrac{x}{m}}},
        & x \leq m \bar{b} \text{ and } \slope\eparen{\tfrac{x}{m}} \leq b.
    \end{cases} \label{eq:late_crossing_bound}
  \end{multline}
\end{enumerate}
\end{theorem}

We give a straightforward proof in \cref{sec:proof_uniform_chernoff} that uses
only Ville's maximal inequality for nonnegative supermartingales
\citep{ville_etude_1939} and elementary convex
analysis. \Cref{th:uniform_chernoff} can be seen to unify and strengthen many
known exponential bounds, showing that we lose nothing in going from a
fixed-time to a uniform bound. This includes classical inequalities by Hoeffding
(\cref{th:chernoff_cases}a), Bennett and Freedman (\cref{th:chernoff_cases}b),
and Bernstein (\cref{th:chernoff_cases}c), along with their matrix extensions
due to Tropp and Mackey et al.  (\cref{th:chernoff_cases}a-c); discrete-time
scalar line-crossing inequalities due to Blackwell
(\cref{th:gamma_line_crossing,th:bernoulli_line_origin}) and Khan
(\cref{sec:line_crossing}); self-normalized bounds due to de la Pe\~na
(\cref{th:variance_normalized,th:ratio_subg}), Delyon
(\cref{th:delyon_uniform}), Bercu and Touati (\cref{th:delyon_uniform}), and Fan
(\cref{th:fan_uniform}); bounds for martingales in smooth Banach spaces due to
Pinelis (\cref{th:pinelis}); continuous-time bounds due to Shorack and Wellner
(\cref{th:continuous_bounds}) and van de Geer (\cref{th:continuous_bounds}); and
Wald's sequential probability ratio test
(\cref{th:sprt_summary}). Visualizations of how the bounds of
\cref{th:uniform_chernoff} relate to Freedman's and de la Pe\~na's inequalities
are provided in \cref{fig:three_bounds,fig:ratio_constant}. For convenience,
\cref{tab:improvements} lists the existing results we recover and our
corresponding corollaries, along with ways in which our analysis strengthens
conclusions.

For the remainder of the paper after \cref{sec:proof_uniform_chernoff}, we will
assume $\Fcal_0$ is the trivial $\sigma$-field and omit from our notation the
conditioning on $\Fcal_0$ in the results of \cref{th:uniform_chernoff} and its
corollaries.

\begin{table}
  \renewcommand{\arraystretch}{1.5}
  \centering
  \begin{tabular}{rlccccc}
    \toprule
    Existing result & Our result & [A] & [B] & [C] & [D] & [E] \\
    \midrule
    \citet{bernstein_theory_1927} & \cref{th:chernoff_cases}(c)
      & & \checkmark & \checkmark & \checkmark & \\
    \citet[eq.\ 8b]{bennett_probability_1962} & \cref{th:chernoff_cases}(b)
      & \checkmark & \checkmark & \checkmark & \checkmark & \\
    \citet[Theorem 2]{hoeffding_probability_1963}
      & \cref{th:chernoff_cases}(a)
      & \checkmark & \checkmark & & \checkmark & \\
    \citet[Theorem 1.6]{freedman_tail_1975} & \cref{th:chernoff_cases}(b)
      & & \checkmark & \checkmark & \checkmark & \\
    \citet[App.\ B, Ineq.\ 1]{shorack_empirical_1986}
      & \cref{th:continuous_bounds}(b)
      & & \checkmark & & & \\
    \citet[Theorems 3.4, 3.5]{pinelis_optimum_1994} & \cref{th:pinelis}
      & & \checkmark & & & \\
    \citet[Lemma 2.2]{van_de_geer_exponential_1995}
      & \cref{th:continuous_bounds}(c)
      & & \checkmark & & \checkmark \\
    \citet[Theorem 1]{blackwell_large_1997}
     & \cref{th:gamma_line_crossing}(a)
      & \checkmark & & \checkmark & \checkmark & \\
    \citet[Theorem 2]{blackwell_large_1997} & \cref{th:bernoulli_line_origin}
      &  & &  & \checkmark & \\
    \citet[Theorem 2]{blackwell_large_1997}
      & \cref{th:gamma_line_crossing}(b)
      & \checkmark & & \checkmark & \checkmark & \\
    \citet[Theorems 6.1, 1.2B (eq.\ 1.5)]{de_la_pena_general_1999}
      & \cref{th:variance_normalized}
      & & \checkmark & \checkmark & \checkmark & \\
    \citet[Theorem 6.2]{de_la_pena_general_1999} & \cref{th:ratio_subg}
      & &  & \checkmark & \checkmark & \checkmark \\
    \citet[Theorem 2.1]{bercu_exponential_2008} & \cref{th:delyon_uniform}
      & & \checkmark & & \checkmark & \checkmark \\
    \citet[Theorem 4]{delyon_exponential_2009} & \cref{th:delyon_uniform}
      & & \checkmark & & \checkmark & \\
    \citet[Theorem 4.2]{khan_$l_p$-version_2009}
      & \cref{th:uniform_chernoff}(b)
      & & \checkmark & \checkmark & \checkmark & \\
    \citet[Theorem 4.3]{khan_$l_p$-version_2009} & \cref{th:uniform_chernoff}(d)
      & & & \checkmark & \checkmark & \checkmark \\
    \citet[Theorem 1.2]{tropp_freedmans_2011} & \cref{th:chernoff_cases}(b)
      & & \checkmark & & & \\
    \citet[Theorem 1.3]{tropp_user-friendly_2012}
      & \cref{th:chernoff_cases}(a)
      & & \checkmark & & & \checkmark \\
    \citet[Theorem 1.4]{tropp_user-friendly_2012}
      & \cref{th:chernoff_cases}(c)
      & & \checkmark & & & \\
    \citet[Corollary 4.2]{mackey_matrix_2014} & \cref{th:chernoff_cases}(a)
      & \checkmark & \checkmark & & & \\
    \bottomrule
  \end{tabular}
  \captionsetup{singlelinecheck=off}
  \caption[Some existing results which are strengthened by
  \cref{th:uniform_chernoff}]{Some existing results which are strengthened by
    \cref{th:uniform_chernoff}, as detailed in
    \cref{sec:special_cases}. \label{tab:improvements} For clarity, we enumerate
    the different ways in which we strengthen or generalize existing results
    with the following mnemonics:
\begin{enumerate}
\item[{[A]}] Assumptions: we recover the result under weaker conditions on the
  distributional or dependence structure of the process.
\item[{[B]}] Boundary: we strengthen the result by replacing a fixed-time bound
  or a finite-horizon constant uniform boundary with an infinite-horizon linear
  uniform boundary which is everywhere at least as strong (i.e., low) as the
  fixed-time or finite-horizon bound.
\item[{[C]}] Continuous time: we extend a discrete-time result to include
  continuous time.
\item[{[D]}] Dimension: we extend a result for scalar process to one for
  $\Hcal^d$-valued processes, recovering the scalar result at $d = 1$.
\item[{[E]}] Exponent: we improve the exponent in the result's probability
  bound.
\end{enumerate}
}
\end{table}

\subsection{Proof of \cref{th:uniform_chernoff}}
\label{sec:proof_uniform_chernoff}

Throughout the proof, we write $\P_0(\cdot)$ for the conditional probability
$\P\condparen{\cdot}{\Fcal_0}$. Ville's maximal inequality for nonnegative
supermartingales (\citealp{ville_etude_1939}) is the foundation of all uniform
bounds in this paper. It is an infinite-horizon uniform extension of Markov's
inequality:
\begin{lemma}[Ville's inequality]\label{th:ville}
  If $(L_t)_{t \in \Tcal \union \brace{0}}$ is a nonnegative supermartingale
  with respect to the filtration $(\Fcal_t)_{t \in \Tcal \union \brace{0}}$,
  then for any $a > 0$, we have
  \begin{align}
    \P_0\eparen{\exists t \in \Tcal: L_t \geq a} \leq \frac{L_0}{a}.
  \end{align}
\end{lemma}
For completeness, we give an elementary proof of \cref{th:ville} in
\cref{sec:proof_ville}. Applying Ville's inequality to
\cref{th:canonical_assumption} gives, for any $(S_t,V_t) \in \subpsiclass$,
$\lambda \in (0, \lambda_{\max})$, and $z \in \R$,
\begin{multline}
\P_0\eparen{
  \exists t \in \Tcal: \expebrace{\lambda S_t - \psi(\lambda) V_t} \geq e^z
} \\
\leq \P_0\eparen{\exists t \in \Tcal: L_t \geq e^z}
\leq L_0 e^{-z} \leq l_0 e^{-z}.
\label{eq:simple_linear_bound}
\end{multline}
To derive \cref{th:uniform_chernoff}(a) from \eqref{eq:simple_linear_bound}, fix
$a, b > 0$ and choose $\lambda \in [0, \lambda_{\max})$ such that
$\psi(\lambda) \leq b \lambda$, supposing for the moment that some such value of
$\lambda$ exists. Then
\begin{align}
\P_0\eparen{\exists t \in \Tcal: S_t \geq a + b V_t}
  &= \P_0\eparen{
      \exists t \in \Tcal: \expebrace{\lambda S_t - b \lambda V_t}
      \geq e^{a \lambda}
    } \\
  &\leq \P_0\eparen{
      \exists t \in \Tcal: \expebrace{\lambda S_t - \psi(\lambda) V_t}
      \geq e^{a \lambda}
    } \\
  &\leq l_0 e^{-a \lambda},
\end{align}
applying \eqref{eq:simple_linear_bound} in the last step. This bound holds for
all choices of $\lambda$ in the set
$\brace{\lambda \in [0, \lambda_{\max}): \psi(\lambda) / \lambda \leq b}$, so to
minimize the final bound, we take the supremum over this set, recovering the
stated bound $l_0 e^{-a \decay(b)}$ by the definition of $\decay(b)$. If no
value $\lambda \in [0, \lambda_{\max})$ satisfies
$\psi(\lambda) \leq b \lambda$, then $\decay(b) = 0$ by definition, so that the
bound holds trivially. This shows that \cref{th:canonical_assumption} implies
\cref{th:uniform_chernoff}(a).

To complete the proof we will show that the four parts of
\cref{th:uniform_chernoff} are equivalent whenever $\psi$ is CGF-like. We
repeatedly use the well-known fact about the Legendre-Fenchel transform that
${\psi'}^{-1}(u) = \psistarprime(u)$ for $0 < u < \bar{b}$, which follows by
differentiating the identity
$\psi^\star(u) = u {\psi'}^{-1}(u) - \psi({\psi'}^{-1}(u))$.  We also require
some simple facts about $\psi(\lambda) / \lambda$:
\begin{lemma}\label{th:slope_facts}
Suppose $\psi$ is CGF-like with domain $[0, \lambda_{\max})$.
\begin{enumerate}[label=(\roman*)]
\item $\psi(\lambda) / \lambda < \psi'(\lambda)$ for all
  $\lambda \in (0, \lambda_{\max})$.
\item $\lambda \mapsto \psi(\lambda) / \lambda$ is continuous and strictly
  increasing on $\lambda > 0$.
\item
  $\inf_{\lambda \in (0, \lambda_{\max})} \psi(\lambda) / \lambda =
  \lim_{\lambda \downarrow 0} \psi(\lambda) / \lambda = 0$.
\item
  $\sup_{\lambda \in (0, \lambda_{\max})} \psi(\lambda) / \lambda =
  \lim_{\lambda \uparrow \lambda_{\max}} \psi(\lambda) / \lambda = \bar{b}$.
\item $\psi(\decay(b)) / \decay(b) = b$ for any $b \in (0, \bar{b})$, that is,
  $\decay(b)$ is the inverse of $\psi(\lambda) / \lambda$.
\item $\slope(u)$ is continuous, strictly increasing, and $0 < \slope(u) < u$
  for all $u \in (0, \bar{b})$.
\end{enumerate}
\end{lemma}
\begin{proof}[Proof of \cref{th:slope_facts}]
  To see (i), write
  $\psi(\lambda) = \int_0^\lambda \psi'(t) \d t < \lambda \psi'(\lambda)$, where
  the inequality follows since $\psi$ is strictly convex so that $\psi'$ is
  strictly increasing. For (ii), the function is continuous because $\psi$ is
  continuous, and differentiating reveals it to be strictly increasing by part
  (i). L'Hôpital's rule implies (iii) along with the assumptions
  $\psi(\lambda) = \psi'(\lambda) = 0$ at $\lambda = 0$, and implies (iv) along
  with the CGF-like assumption $\sup_\lambda \psi(\lambda) = \infty$, which
  means $\psi(\lambda) \uparrow \infty$ as $\lambda \uparrow \lambda_{\max}$
  since $\psi$ is convex. Part (v) follows from the definition of
  $\decay(\cdot)$ and parts (ii), (iii) and (iv). To obtain (vi), note that
  $\slope$ is the composition of $\lambda \mapsto \psi(\lambda) / \lambda$ with
  $\psistarprime$. Both of these are continuous and strictly increasing, the
  former by part (ii) and the latter since $\psistarprime = {\psi'}^{-1}$ and
  $\psi'$ is continuous and strictly increasing by the CGF-like assumption. As
  $u \downarrow 0$, we have $\psistarprime(u) = {\psi'}^{-1}(u) \downarrow 0$,
  so $\slope(u) \downarrow 0$ since $\psi(0) = \psi'(0_+) = 0$. Likewise, if
  $\bar{b} < \infty$, then as $u \uparrow \bar{b}$,
  $\psistarprime(u) \uparrow \lambda_{\max}$ and $\slope(u) \uparrow
  \bar{b}$. Hence $\slope(u)$ is continuous as defined. Next, note that
  $\psi(u) > 0$ for $u > 0$ since $\psi$ is strictly convex with
  $\psi(0) = \psi'(0_+) = 0$, and $\psistarprime(u) = {\psi'}^{-1}(u) > 0$ since
  $\psi'(\lambda)$ increases from zero at $\lambda = 0$ to $\bar{b}$ as
  $\lambda \uparrow \lambda_{\max}$. Hence $\slope(u) > 0$ for $u > 0$. Finally,
  use part (i) to write
  $\slope(u) = \psi(\psistarprime(u)) / \psistarprime(u) <
  \psi'(\psistarprime(u)) = u$, using the fact that
  $\psistarprime(u) = {\psi'}^{-1}(u)$ for $u \in (0, \bar{b})$.
\end{proof}

\Cref{th:slope_facts} allows us to prove the equivalences among the parts of
\cref{th:uniform_chernoff} as follows.

\begin{itemize}
\item $(a) \implies (b)$: Fix $m > 0$ and $x \in (0, m\bar{b})$. Any line with
  slope $b \in (0, x/m)$ and intercept $x - bm$ passes through the point
  $(m, x)$ in the $(V_t, S_t)$ plane, and part $(a)$ yields
  \begin{align}
    \P_0\eparen{\exists t \in \Tcal: S_t \geq x + b (V_t - m)}
      &\leq l_0 \expebrace{-(x - bm) \decay(b)} \\
      &= l_0 \expebrace{
           -m \eparen{\frac{x}{m} \cdot \decay(b) - \psi(\decay(b))}
         }
  \end{align}
  using \cref{th:slope_facts}(v) in the second step. Now we choose the slope $b$
  to minimize the probability bound. The unconstrained optimizer $b_\star$
  satisfies ${\psi'(\decay(b_\star)) = x / m}$, and a solution is guaranteed to
  exist by our restriction on $x$. This solution is given by
  $\decay(b_\star) = {\psi'}^{-1}(x/m) = \psistarprime(x/m)$. Hence
  $b_\star = \slope(x/m)$ using \cref{th:slope_facts}(v) and the definition of
  $\slope(\cdot)$. \Cref{th:slope_facts}(vi) shows $0 < b_\star < x/m$,
  verifying that $b_\star$ is feasible. Identify the Legendre-Fenchel
  transformation
  $\psi^\star(x/m) = (x/m) \decay(b_\star) - \psi(\decay(b_\star))$ to complete
  the proof of part $(b)$.
\item $(b) \implies (c)$: Fix $m > 0$ and $x \in (0, \bar{b})$ and observe that
  \begin{multline}
    \P_0\eparen{
      \exists t \in \Tcal:
      \frac{S_t}{V_t} \geq x - \pfrac{x - \slope(x)}{V_t} \cdot (V_t - m)
    }
    \\= \P_0\eparen{\exists t \in \Tcal:
         S_t \geq m x + \slope(x) \cdot (V_t - m)}.
  \end{multline}
  Now applying part $(b)$ with values $m$ and $m x$ yields part $(c)$.
\item $(c) \implies (a)$: Fix $a, b > 0$. Suppose first that $b < \bar{b}$, and
  set $x = \psi'(\decay(b))$ and $m = a / (x - \slope(x))$. Recalling
  $\psistarprime = {\psi'}^{-1}$ we see that
  $\slope(x) = \psi(\decay(b)) / \decay(b) = b$ by
  \cref{th:slope_facts}(v). Also, \cref{th:slope_facts}(vi) shows that $m >
  0$. Now apply part $(c)$ to obtain
  \begin{align}
    \P_0\eparen{\exists t \in \Tcal: S_t \geq a + b V_t}
      &\leq l_0 \expebrace{-a \cdot \frac{\psi^\star(x)}{x - \slope(x)}} \\
      &= l_0 \expebrace{
        -a \cdot \frac{\psi^\star(x) \cdot \psistarprime(x)}
                      {x \psistarprime(x) - \psi(\psistarprime(x))}
      }.
  \end{align}
  Recognizing the Legendre-Fenchel transform in the denominator of the final
  exponent, we see that the probability bound equals
  $l_0 \expebrace{-a \psistarprime(x)}$. Again using
  $\psistarprime(x) = {\psi'}^{-1}(x) = \decay(b)$
  yields part $(a)$.

  If instead $b \geq \bar{b}$, then the above argument yields
  \begin{multline}
    \P_0\eparen{\exists t \in \Tcal: S_t \geq a + b V_t}
      \leq \inf_{b' < \bar{b}}
        \P_0\eparen{\exists t \in \Tcal: S_t \geq a + b' V_t}
      \\ \leq l_0 \expebrace{-a \sup_{b' < \bar{b}} D(b')}.
  \end{multline}
  But $\sup_{b' < \bar{b}} D(b') = \lambda_{\max} = D(b)$ from the definition of
  $D(\cdot)$.

\item $(a) \implies (d)$: Fix $m \geq 0$ and $x, b > 0$. Observe that
  $\brace{\exists t \in \Tcal: V_t \geq m,\, S_t \geq x + b(V_t - m)} \subseteq
  \brace{\exists t \in \Tcal: S_t \geq x' + b' (V_t - m)}$ for any
  $0 < x' \leq x$ and $0 < b' \leq b$, so part $(a)$ yields
  \begin{align}
    \P_0\eparen{\exists t \in \Tcal: V_t \geq m,\, S_t \geq x + b(V_t - m)}
      \leq l_0 \expebrace{-(x' - b' m) D(b')}
  \end{align}
  for any $(x', b')$ in the feasible set
  $\brace{x' \in (0, x], b' \in (0, b]: x' > m b'}$. If $x > m \bar{b}$, then
  $(x, b \bmin \bar{b})$ is feasible; note that
  $D(b \bmin \bar{b}) = D(b)$ by the definition of $D(\cdot)$. If
  $x \leq m \bar{b}$ and $b < s(x/m)$, then by \cref{th:slope_facts}(vi) and the
  definition $\slope(\bar{b}) \defineas \bar{b}$, we have $b < x/m$, so
  $(x, b)$ is feasible and $b \leq \bar{b}$. Combining these two cases,
  we have
  \begin{align}
    \P_0\eparen{\exists t \in \Tcal: V_t \geq m,\, S_t \geq x + b(V_t - m)}
      \leq l_0 \expebrace{-(x - (b \bmin \bar{b}) m) D(b)}
  \end{align}
  whenever $x > m \bar{b}$ or $b < s(x/m)$, proving the first case in
  \eqref{eq:late_crossing_bound}. On the other hand, if $x \leq m \bar{b}$ and
  $s(x/m) \leq b$, then $(x', s(x'/m))$ is feasible for any $x' < x$, by
  \cref{th:slope_facts}(vi). This yields
  \begin{align}
    \P_0\eparen{\exists t \in \Tcal: V_t \geq m,\, S_t \geq x + b(V_t - m)}
      \leq l_0 \expebrace{-m \psi^\star\pfrac{x'}{m}}
  \end{align}
  as in part $(b)$. We minimize the probability bound over $x' < x$, noting that
  $\sup_{x' < x} \psi^\star(x'/m) = \psi^\star(x/m)$ since $\psi^\star$ is
  increasing (as $\psi$ is CGF-like) and closed \citep[Theorem
  12.2]{rockafellar_convex_1970}. This proves the second case in
  \eqref{eq:late_crossing_bound}.
\item $(d) \implies (a)$: set $m = 0$ and $x = a$ to recover part $(a)$. \qed
\end{itemize}

It is worth noting here that, unlike the proofs of \citet{freedman_tail_1975},
\citet{khan_$l_p$-version_2009}, \citet{tropp_freedmans_2011}, and
\citet{fan_exponential_2015}, we do not explicitly construct a stopping time in
our proof. While an optional stopping argument is hidden within the proof of
Ville's inequality, the underlying stopping time here is different from that in
the aforementioned citations.

\subsection{Interpreting the theorem}

\begin{figure}[h!]
\includegraphics{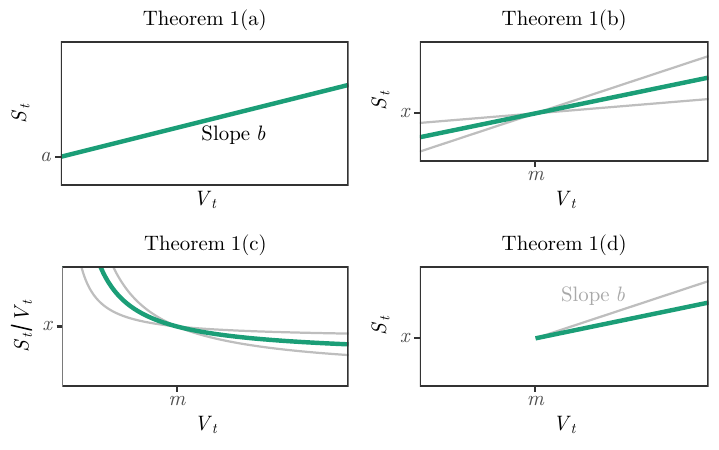}
\caption{Illustration of the equivalent statements of
  \cref{th:uniform_chernoff}, as described in the
  text. \label{fig:theorem_1_illustration}}
\end{figure}

It is instructive to think of the parts of \cref{th:uniform_chernoff} as
statements about the process $(V_t, S_t)$ or $(V_t, S_t / V_t)$ in $\R^2$. Many
of our results are better understood via this geometric intuition. Specifically,
\Cref{fig:theorem_1_illustration} illustrates the following points:
\begin{itemize}
\item \Cref{th:uniform_chernoff}(a) takes a given line $a + b V_t$ and bounds
  its $S_t$-upcrossing probability.
\item \Cref{th:uniform_chernoff}(b) takes a point $(m, x)$ in the
  $(V_t, S_t)$-plane and, out of the infinitely many lines passing through it,
  chooses the one which yields the tightest upper bound on the corresponding
  $S_t$-upcrossing probability.
\item \Cref{th:uniform_chernoff}(c) is like part (b), but instead of looking at
  $S_t$, we look at $S_t / V_t$, fix a point $(m, x)$ in the
  $(V_t, S_t / V_t)$-plane, and choose from among the infinitely many curves
  $b + a/V_t$ passing through it to minimize the probability bound.
\item The intuition for \cref{th:uniform_chernoff}(d) is as follows. If we want
  to bound the upcrossing probability of the line $(x - bm) + b V_t$ on
  $\brace{V_t \geq m}$, we can clearly obtain a conservative bound from
  \cref{th:uniform_chernoff}(a) with $a = x - bm$. This yields the first case in
  \eqref{eq:late_crossing_bound}. However, we can also apply
  \cref{th:uniform_chernoff}(b) with the values $m, x$, obtaining a bound on the
  upcrossing probability for a line which passes through the point $(m,x)$ in
  the $(V_t,S_t)$-plane, and this line yields the minimum possible probability
  bound among all lines passing through $(m,x)$. If the slope of this line,
  $\slope(x/m)$, is less than $b$, then this optimal probability bound is
  conservative for the upcrossing probability over the original line
  $x + b(V_t - m)$ on $\brace{V_t \geq m}$. This gives the second case in
  \eqref{eq:late_crossing_bound}, which is guaranteed to be at least as small as
  the bound in the first case when $\slope(x/m) \leq b$.
\end{itemize}

We make some additional remarks below:
\begin{itemize}
\item We extend bounds for discrete-time scalar-valued processes to include both
  discrete-time matrix-valued processes and continuous-time scalar-valued
  processes, but we do not handle continuous-time matrix-valued processes, as
  this seems to require further technical developments beyond the scope of this
  paper (see \citet{bacry_concentration_2018} for one approach to exponential
  bounds in this case). We write [C or D] when discussing extensions to existing
  results to emphasize this fact (see \cref{tab:improvements}).
\item Most of this paper is concerned with right-tail bounds, hence the
  restriction to $\lambda \geq 0$ in \cref{th:canonical_assumption}. It is
  understood that identical techniques yield left-tail bounds upon verifying
  that \cref{th:canonical_assumption} holds for $(-S_t)$.
\item The purpose of excluding $\psi$ being CGF-like from
  \cref{th:canonical_assumption} is to separate the truth of statement (a),
  which follows solely from \cref{th:canonical_assumption}, from its equivalence to (b), (c),
  and (d), which follows from $\psi$ being CGF-like.
\end{itemize}

\subsection{Three simple examples}\label{sec:thm_examples}

We illustrate some simple instantiations of our theorem with three examples: a
sum of coin flips, a discrete-time concentration inequality for random matrices,
and a continuous-time scalar Brownian motion. These examples make use of several
results from \cref{sec:sufficient_conditions} describing conditions under which
a process is sub-$\psi$; such results may be taken for granted on a first
reading.

\begin{example}[Coin flipping]\label{ex:coin_flipping}
  Suppose $X_i \asiid \Ber(p)$, and let $S_t = \sum_{i=1}^t (X_i - p)$ denote
  the centered sum. The CGF of each increment of $S_t$, scaled by $1/[p(1-p)]$,
  is
  $\psi_B(\lambda) \defineas [p(1-p)]^{-1} \log \E \expebrace{\lambda(X_i - p)}
  = [p(1-p)]^{-1} \log(p e^{(1-p)\lambda} + (1-p) e^{-p\lambda})$, so that
  $\lambda_{\max} = \infty$ and $\bar{b} = 1/p$. One may directly check the
  martingale property to confirm that
  ${L_t(\lambda) \defineas \expebrace{\lambda S_t - \psi_B(\lambda) p (1-p) t}}$
  is a martingale for any $\lambda$, so that $(S_t)$ is 1-sub-$\psi_B$ with
  $V_t = p (1-p) t$. Then, for any $t_0 \in \N$ and $x \in (0, (1-p) t_0)$,
  setting $m = p (1-p) t_0$ in \cref{th:uniform_chernoff}(b) yields
\begin{align}
&\P\eparen{\exists t \in \N:
    S_t \geq x + p(1-p) \slope_B\pfrac{x}{p (1-p) t_0} \cdot (t - t_0)} \\
  &\qquad \leq \expebrace{-t_0 \KL{p + \frac{x}{t_0}}{p}} \\
  &\qquad = \ebracket{
       \pfrac{p}{p + x/t_0}^{p + x/t_0} \pfrac{1-p}{1-p - x/t_0}^{1-p - x/t_0}
     }^{t_0}.
\end{align}
Here KL denotes the Bernoulli Kullback-Leibler divergence,
\begin{align}
  \KL{q}{p} = q \log\pfrac{q}{p} + (1 - q) \log\pfrac{1-q}{1-p}.
\end{align}
It takes some algebra to obtain this KL as the Legendre-Fenchel transform of
$\psi_B$; in \cref{tab:psi_transforms} we summarize all such transforms used in
this paper. The final expression is Equation (2.1) of
\citet{hoeffding_probability_1963}, but here we have a bound not just for the
deviation of $S_m$ above its expectation at the fixed time $m$, but for the
upper deviations of $S_t$ for all $t \in \N$, simultaneously. We can use this to
sequentially test a hypothesis about $p$, or to construct a sequence of
confidence intervals for $p$ possessing a coverage guarantee holding uniformly
over unbounded time.

The slope transform $\slope_B(u)$ for $\psi_B$, given in
\cref{tab:psi_transforms}, is unwieldy. To derive a more analytically convenient
bound, we use the fact that $p (1-p) \psi_B(\lambda) \leq \lambda^2 / 8$ for all
$\lambda \geq 0$; see the proof of \cref{th:psi_relations}, part 2. Hence
$\expebrace{\lambda S_t - \lambda^2 t / 8} \leq L_t(\lambda)$ with $L_t$ defined
as above, so $(S_t)$ is also 1-sub-$\psi$ with $\psi(\lambda) = \lambda^2 / 8$
and $V_t = t$. Now \cref{th:uniform_chernoff}(b) yields
\begin{align}
\P\eparen{\exists t \in \N: S_t \geq x + \frac{x}{2m} \cdot (t - m)}
  \leq \expebrace{-\frac{2x^2}{m}}.
  \label{eq:example_subg}
\end{align}
This is equivalent to Blackwell's line-crossing inequality
\eqref{eq:blackwell_intro}, and in the form \eqref{eq:example_subg} it is clear
that it recovers Hoeffding's inequality at the fixed time $t = m$. Instead of
using $p (1-p) \psi_B(\lambda) \leq \lambda^2 / 8$, we might alternatively use
$\psi_B(\lambda) \leq (1-2p)^{-2} (e^{(1-2p)\lambda} - (1-2p)\lambda - 1)$; see
the proof of \cref{th:psi_relations}, part 3. This will yield a uniform
extension of Bennett's inequality \eqref{eq:bennett_intro} which improves upon
Hoeffding's inequality substantially for values of $p$ near zero and one. We
will see other examples of such ``sub-Poisson'' bounds below.
\end{example}

\begin{example}[Covariance estimation for a spiked random vector ensemble]
  The estimation of a covariance matrix via an i.i.d.\ sample is a common
  application of exponential matrix concentration, starting with
  \citet{rudelson_random_1999}. See also \citet{vershynin_introduction_2012},
  \citet{gittens_tail_2011}, \citet{tropp_introduction_2015}, and
  \citet{koltchinskii_concentration_2017} for more recent treatments; this
  particular example is drawn from \citet{wainwright_high-dimensional_2017}. Let
  $d \geq 2$ and consider $\R^d$-valued, mean-zero observations
  $X_i = \sqrt{d} \xi_i e_{U_i}$, where $\xi_i \asiid \Rademacher$,
  $(e_k)_{k = 1}^d$ are the standard basis vectors and
  $U_i \asiid \Unif\ebrace{1, \dots, d}$. What can we say about the
  concentration of the sample covariance matrix
  $\widehat{\Sigma}_t \defineas t^{-1} \sum_{i=1}^t X_i X_i^T$ around the true
  covariance $I_d$, the $d \times d$ identity matrix? Let $\gamma_{\max}(A)$
  denote the maximum eigenvalue of a matrix $A$. We have
  $\gamma_{\max}(X_i X_i^T - I_d) = d - 1$ always, and
  $\E (X_i X_i^T - I_d)^2 = \pfrac{(d-1)^2}{d} I_d$. Hence
  \cref{th:discrete_facts}(c) shows that
  $S_t = t \gamma_{\max}(\widehat{\Sigma}_t - I_d)$ is $d$-sub-$\psi$ with
  variance process $V_t = \frac{(d - 1)^2 t}{d}$, where
\begin{align}
\psi(\lambda) = \frac{e^{(d-1) \lambda} - (d-1) \lambda - 1}{(d-1)^2}
  \leq \frac{\lambda^2}{2(1 - (d - 1) \lambda / 3)}.
  \label{eq:spiked_ensemble_psi}
\end{align}
Here the inequality holds for all $\lambda \in [0, 3/(d-1))$ as demonstrated in
the proof of \cref{th:psi_relations}, part 5. Applying
\cref{th:uniform_chernoff}(c) with $\psi$ equal to the final expression in
\eqref{eq:spiked_ensemble_psi}, we obtain, after some algebra, for any
$x, m > 0$,
\begin{multline}
\P\eparen{
  \exists t \in \N: \gamma_{\max}\eparen{\widehat{\Sigma}_t - I_d}
  \geq x \pfrac{1 + \frac{m}{t}\sqrt{1 + 2x/3(d-1)}}{1 + \sqrt{1 + 2x/3(d-1)}}
} \\ \leq d \expebrace{-\frac{mx^2}{2(d-1)\ebracket{(d-1) / d + x / 3}}}.
\label{eq:spiked_ensemble_bound}
\end{multline}
At the fixed time $t = m$, this implies
\begin{align}
\gamma_{\max}\eparen{\widehat{\Sigma}_m - I_d}
  \leq \sqrt{\frac{2(d-1)^2 \log(d/\alpha)}{dm}}
       + \frac{2(d-1) \log(d/\alpha)}{3m}
\end{align}
with probability at least $1-\alpha$, a known fixed-sample result
\citep{wainwright_high-dimensional_2017}. However, as above,
\eqref{eq:spiked_ensemble_bound} gives a bound on the upper deviations of
$\widehat{\Sigma}_t$ for all $t \in \N$ simultaneously. Such a bound enables,
for example, sequential hypothesis tests concerning the true covariance matrix.
\end{example}

\begin{example}[Line-crossing for Brownian motion]\label{ex:bm}
  Let $(S_t)_{t \in [0, \infty)}$ denote standard Brownian motion. It is a
  standard fact that the process $\expebrace{\lambda S_t - \lambda^2 t / 2}$ is
  a martingale, so that $(S_t)$ is 1-sub-$\psi$ with
  $\psi(\lambda) = \lambda^2 / 2$ and $V_t = t$. In this case,
  \cref{th:uniform_chernoff} says that, for any $a,b > 0$,
\begin{align}
\P\eparen{\exists t \in (0, \infty): S_t \geq a + bt} \leq e^{-2ab},
\end{align}
a well-known line-crossing bound for Brownian motion, which in fact holds with
equality \citep[Exercise 7.5.2]{durrett_probability:_2017}.
\end{example}

\section{Sufficient conditions for sub-$\psi$ processes}
\label{sec:sufficient_conditions}

Much of the power of \cref{th:canonical_assumption} comes from the array of
sufficient conditions for it which have been discovered under diverse,
nonparametric conditions. In this section, we define some standard $\psi$
functions and collect a broad set of conditions from the literature for a
process $(S_t)$ to be sub-$\psi$ with one of these functions, summarized in
\cref{tab:sufficient_conditions_1,tab:sufficient_conditions_2}. In other words,
we collect here some families of process pairs $(S_t,V_t)$ which are contained
within $\subpsiclass$ for standard choices of $\psi$. All discrete-time results
in this paper use $S_t = \gamma_{\max}(Y_t)$ where $(Y_t)_{t \in \N}$ is a
martingale taking values in $\Hcal^d$, with the exception of \cref{sec:banach},
which deals with martingales in abstract Banach spaces. Typically, setting
$d = 1$ recovers the corresponding known scalar result exactly. We note also
that our results for Hermitian matrices extend directly to rectangular matrices
using Hermitian dilations \citep{tropp_user-friendly_2012}, as we illustrate in
\cref{th:rectangular_series}.

\subsection{Five useful $\psi$ functions}\label{sec:psi_functions}

We define five particular $\psi$ functions corresponding to five sub-$\psi$
cases: the sub-Gaussian case in Hoeffding's inequality, the ``sub-gamma'' case
corresponding to Bernstein's inequality, the sub-Poisson case from Bennett's and
Freedman's inequalities, and the sub-exponential and sub-Bernoulli cases which
are used in several other existing bounds. The $\psi$ functions and
corresponding transforms for these five cases are summarized in
\cref{tab:psi_transforms}, while \cref{fig:psi_relations} summarizes
relationships among these cases, with \cref{th:psi_relations} containing the
formal statements. Recall
$\bar{b} = \sup_{\lambda \in [0, \lambda_{\max})} \psi'(\lambda)$ from
\cref{th:cgf_like}, and note that we take $1/0 = \infty$ by convention in the
expressions for $\lambda_{\max}$ and $\bar{b}$ below.

\begin{enumerate}
\item We say $(S_t)$ is \emph{sub-Bernoulli} with range parameters $g, h > 0$
  when it is sub-$\psi_{B,g,h}$ for some suitable variance process $(V_t)$,
  where
  \begin{align}
    \psi_{B,g,h}(\lambda) \defineas
      \frac{1}{gh} \log\pfrac{g e^{h\lambda} + h e^{-g \lambda}}{g + h}
    \quad \text{for } 0 \leq \lambda < \infty = \lambda_{\max},
  \end{align}
  which is the scaled CGF of a mean-zero random variable taking values $-g$ and
  $h$. Here $\bar{b} = 1/g$.
\item We say $(S_t)$ is \emph{sub-Gaussian} when it is sub-$\psi_N$ for
  some suitable variance process $(V_t)$, where
  \begin{align}
    \psi_N(\lambda) \defineas \lambda^2 / 2
    \quad \text{for } 0 \leq \lambda < \infty = \lambda_{\max}.
  \end{align}
  Here $\bar{b} = \infty$.
\item We say $(S_t)$ is \emph{sub-Poisson} with scale parameter
    $c \in \R$ when it is sub-$\psi_{P,c}$ for some suitable variance process
  $(V_t)$, where
  \begin{align}
    \psi_{P,c}(\lambda) \defineas \frac{e^{c\lambda} - c\lambda - 1}{c^2}
    \quad \text{for } 0 \leq \lambda < \infty = \lambda_{\max}.
  \end{align}
  By taking the limit, we define $\psi_{P,0} = \psi_N$. Here
  $\bar{b} = \abs{c \bmin 0}^{-1}$.
\item We say $(S_t)$ is \emph{sub-gamma} with scale parameter
    $c \in \R$ when it is sub-$\psi_{G,c}$ for some suitable variance process
  $(V_t)$, where
  \begin{align}
    \psi_{G,c}(\lambda) \defineas \frac{\lambda^2}{2(1 - c\lambda)}
    \quad \text{for } 0 \leq \lambda < \frac{1}{c \bmax 0} = \lambda_{\max},
  \end{align}
  Here $\bar{b} = \abs{2c \bmin 0}^{-1}$.
\item We say $(S_t)$ is \emph{sub-exponential} with scale parameter
  $c \in \R$ when it is sub-$\psi_{E,c}$ for some suitable
  variance process $(V_t)$, where
  \begin{align}
    \psi_{E,c}(\lambda) \defineas \frac{-\log(1-c\lambda) - c\lambda}{c^2},
    \quad \text{for } 0 \leq \lambda < \frac{1}{c \bmax 0} = \lambda_{\max}.
  \end{align}
  By taking the limit, we define $\psi_{E,0} = \psi_N$. Here
  $\bar{b} = \abs{c \bmin 0}^{-1}$.
\end{enumerate}

We will typically write $\psi_B$, $\psi_P$, $\psi_G$, and $\psi_E$, omitting the
range or scale parameters from the notation when they are clear from the
context. We follow the definition of sub-gamma from
\citet{boucheron_concentration_2013}, despite the somewhat inconsistent
terminology: unlike the other four cases, $\psi_G$ is not the CGF of a
gamma-distributed random variable. It is convenient for a number of reasons: it
includes $\psi_N$ as a special case, it gives a useful upper bound for $\psi_P$
(see \cref{th:psi_relations} part 5, below), it falls naturally out of the use
of a Bernstein condition on higher moments to bound the CGF, and it is simple
enough to permit analytically tractable results for the slope and decay
transforms and the various bounds to follow. We remark also that our definition
of sub-exponential in terms of the CGF of the exponential distribution follows
that of \citet[Exercise 2.22]{boucheron_concentration_2013}, but differs from
another well-known definition which says that the CGF is bounded by
$\lambda^2 / 2$ for $\lambda$ in some neighborhood of zero. The two are
equivalent up to appropriate choice of constants, as detailed in
\cref{sec:equiv_sub_exp}.

The sub-gamma and sub-exponential functions $\psi_{G,c}$ and $\psi_{E,c}$
possess the following universality property, which we prove in
\cref{sec:proof_universal}.
\begin{proposition}\label{th:universal}
  For any twice-differentiable $\psi: [0, \lambda_{\max}) \to \R$ with
  $\psi(0) = \psi'(0_+) = 0$, there exist constants $a, c > 0$ such that
  $\psi(\lambda) \leq a \psi_{G,c}(\lambda)$ for all
  $\lambda \in [0, \lambda_{\max})$. Likewise, there exists constants
  $\tilde{a}, \tilde{c} > 0$ such that
  $\psi(\lambda) \leq \tilde{a} \psi_{E,\tilde{c}}(\lambda)$ for all $\lambda\in [0, \lambda_{\max})$.
\end{proposition}
In particular, this means that if $S_t = \sum_{i=1}^t X_i$ for any zero-mean,
i.i.d.\ sequence $(X_i)$ satisfying $\E e^{\lambda X_1} < \infty$ for some
$\lambda > 0$, then $(S_t)$ is sub-gamma and sub-exponential with appropriate
scale constants and variance process $V_t$ proportional to $t$. Furthermore, any
process that is sub-$\psi$ with a CGF-like $\psi$ function is also sub-gamma
and sub-exponential with appropriate scaling of the variance process by a
constant.

\begin{landscape}
\begin{table}
\renewcommand{\arraystretch}{2}
  \centering
  \begin{tabular}{rcccc}
    \toprule
    Name & $\psi(\lambda)$ &
      \makecell{Legendre-Fenchel\\transform $\psi^\star(u)$} &
      \makecell{Slope transform\\$\slope(u)$} &
      \makecell{Decay transform\\$\decay(u)$} \\
    \midrule
    \makecell[tr]{Bernoulli\\$\psi_B,\slope_B,\decay_B$} &
      $\frac{1}{gh} \log\pfrac{g e^{h\lambda} + h e^{-g \lambda}}{g + h}$ &
      $\frac{1}{gh} \KL{\frac{g(1+hu)}{g+h}}{\frac{g}{g+h}}$ &
      $\frac{h\log(1-gu)^{-1} - g\log(1+hu)}
            {gh(\log(1-gu)^{-1} + \log(1+hu))}$ &
      $\geq \frac{2ghu}{\varphi(g,h)}$ \\
    \makecell[tr]{Gaussian/normal\\$\psi_N,\slope_N,\decay_N$} &
      $\lambda^2 / 2$ &
      $u^2 / 2$ &
      $u / 2$ &
      $2u$ \\
    \makecell[tr]{Poisson\\$\psi_P,\slope_P,\decay_P$} &
      $\frac{e^{c\lambda} - c\lambda - 1}{c^2}$ &
      $\frac{(1 + cu) \log(1 + cu) - cu}{c^2}$ &
      $\frac{cu - \log(1+cu)}{c\log(1 + cu)}$ &
      $\geq \frac{2u}{1+2cu/3}$ \\
    \makecell[tr]{``Gamma''\\$\psi_G,\slope_G,\decay_G$} &
      $\frac{\lambda^2}{2(1 - c\lambda)}$ &
      $\frac{u^2}{1 + cu + \sqrt{1 + 2cu}}$ &
      $\frac{u}{1 + \sqrt{1 + 2cu}}$ &
      $\frac{2u}{1 + 2cu}$ \\
    \makecell[tr]{Exponential\\$\psi_E,\slope_E,\decay_E$} &
      $\frac{\log(1-c\lambda)^{-1} - c \lambda}{c^2}$ &
      $\frac{cu - \log(1 + cu)}{c^2}$ &
      $\frac{(1+cu) \log(1+cu) - cu}{c^2 u}$ &
      $\geq \frac{2u}{1+2cu}$ \\
    \bottomrule
  \end{tabular}
  \parbox{1.4\textwidth}{\caption{Summary of common $\psi$ functions and related transforms. KL denotes
    the Bernoulli Kullback-Leibler divergence,
    ${\KL{q}{p} = q \log\pfrac{q}{p} + (1 - q) \log\pfrac{1-q}{1-p}}$. For the
    gamma and exponential cases, the domain of $\psi$ is bounded by
    $\lambda_{\max} = 1 / (c \bmax 0)$; for the other three cases,
    $\lambda_{\max} = \infty$. For the Bernoulli, Poisson, and exponential
    cases, a closed-form expression for $D(u)$ is not available, but we give
    lower bounds based on \cref{th:psi_relations}; $\varphi(g,h)$ is defined in
    \eqref{eq:varphi_defn}. \label{tab:psi_transforms}}}
\end{table}
\end{landscape}

\subsection{Conditions for sub-$\psi$ processes}\label{sec:sub_psi_conds}

\begin{table}
\renewcommand{\arraystretch}{1.5}
\centering
\begin{tabular}{llll}
\toprule
& Condition & $\psi$ & $V_t$ \\
\midrule
\multicolumn{4}{l}{\emph{Discrete time, one-sided}} \\
\qquad Bernoulli II &
  $\Delta S_t \leq h , \E \Delta S_t^2 \leq gh $ & $\psi_B$
  & $ght$ \\
\qquad Bennett & $\Delta S_t \leq c $ & $\psi_P$ & $\eangle{S}_t$ \\
\qquad Bernstein
  & $\mu^k_t \leq \frac{k!}{2} c^{k-2} \mu^2_t$
  & $\psi_G$ & $\eangle{S}_t$ \\
\qquad $^\star$Heavy on left
  & $\E_{t-1} T_a(\Delta S_t) \leq 0,\, \forall a > 0$
  & $\psi_N$ & $[S]_t$  \\
\qquad Bounded below & $\Delta S_t \geq -c $ & $\psi_E$ & $[S]_t$ \\
\multicolumn{4}{l}{\emph{Discrete time, two-sided}} \\
\qquad Parametric & $\Delta S_t \asiid F$ & $\log \E e^{\lambda \Delta S_1}$
  & $t$ \\
\qquad Bernoulli I & $-g  \leq \Delta S_t \leq h $ & $\psi_B$ &
  $ght$ \\
\qquad Hoeffding-KS & $-g_t  \leq \Delta S_t \leq h_t $ & $\psi_N$
  & $\sum_{i=1}^t \varphi(g_i, h_i) $ \\
\qquad $\implies$ Hoeffding I & $-g_t  \leq \Delta S_t \leq h_t $ & $\psi_N$
  & $\sum_{i=1}^t \pfrac{g_i+h_i}{2}^2 $ \\
\qquad $^\star$Symmetric & $\Delta S_t \sim -\Delta S_t \mid \Fcal_{t-1}$
  & $\psi_N$ & $[S]_t$ \\
\qquad SN I & $\E_{t-1} \Delta S_t^2 < \infty$ & $\psi_N$
  & $([S]_t + 2\eangle{S}_t) / 3$ \\
\qquad SN II & $\E_{t-1} \Delta S_t^2 < \infty$ & $\psi_N$
  & $([S_+]_t + \eangle{S_-}_t) / 2$ \\
\qquad Cubic SN & $\E_{t-1} \abs{\Delta S_t}^3 < \infty$ & $\psi_G$
  & $[S]_t + \sum_{i=1}^t \abs{\mu}^3_i$ \\
\multicolumn{4}{l}{\emph{Continuous time, one-sided}} \\
\qquad Bennett & $\Delta S_t \leq c$ & $\psi_P$ & $\eangle{S}_t$ \\
\qquad Bernstein & $W_{m,t} \leq \frac{m!}{2} c^{m-2} V_t$ & $\psi_G$ & $V_t$ \\
\multicolumn{4}{l}{\emph{Continuous time, two-sided}} \\
\qquad L\'evy & $\E e^{\lambda S_1} < \infty$ & $\log \E e^{\lambda S_1}$ 
  & $t$ \\
\qquad Continuous paths & $\Delta S_t \equiv 0$ & $\psi_N$ & $\eangle{S}_t$ \\
\bottomrule
\end{tabular}
\caption{Summary of sufficient conditions for a real-valued, discrete- or
  continuous-time martingale $(S_t)$ to be sub-$\psi$ with the given variance
  process. We use the shorthand $\mu^k_t \defineas \E_{t-1} (\Delta S_t)^k$ and
  $\abs{\mu}^k_t \defineas \E_{t-1} \eabs{\Delta S_t}^k$. In starred cases
  ($^\star$), the first moment $\E_{i-1} \Delta S_i$ need not exist, so $(S_t)$
  need not be a martingale. See
  \cref{th:discrete_facts,th:novel_discrete_conds,th:continuous_facts} for
  details of each case.  ``$\implies$ Hoeffding I'' indicates that the variance
  process $(V_t)$ for Hoeffding-KS is smaller. ``SN'' is short for
  ``self-normalized''. \label{tab:sufficient_conditions_1}}
\end{table}

\begin{table}
\renewcommand{\arraystretch}{1.5}
\centering
\begin{tabular}{llll}
\toprule
& Condition & $\psi$ & $Z_t$ \\
\midrule
\multicolumn{4}{l}{\emph{Discrete time, one-sided}} \\
\qquad Bernoulli II &
  $\Delta Y_t \preceq h I_d, \E \Delta Y_t^2 \preceq gh I_d$ & $\psi_B$
  & $ght I_d$ \\
\qquad Bennett & $\Delta Y_t \preceq c I_d$ & $\psi_P$ & $\eangle{Y}_t$ \\
\qquad Bernstein
  & $\mu^k_t \preceq \frac{k!}{2} c^{k-2} \mu^2_t$
  & $\psi_G$ & $\eangle{Y}_t$ \\
\qquad Bounded below & $\Delta Y_t \succeq -c I_d$ & $\psi_E$ & $[Y]_t$ \\
\multicolumn{4}{l}{\emph{Discrete time, two-sided}} \\
\qquad Bernoulli I & $-g I_d \preceq \Delta Y_t \preceq h I_d$ & $\psi_B$ &
  $ght I_d$ \\
\qquad Hoeffding-KS & $-G_t I_d \preceq \Delta Y_t \preceq H_t I_d$ & $\psi_N$
  & $\sum_{i=1}^t \varphi(G_i, H_i) I_d$ \\
\qquad $\implies$ Hoeffding I & $-G_t I_d \preceq \Delta Y_t \preceq H_t I_d$
  & $\psi_N$ & $\sum_{i=1}^t \pfrac{G_i+H_i}{2}^2 I_d$ \\
\qquad Hoeffding II & $\Delta Y_t^2 \preceq A_t^2$ & $\psi_N$
  & $\sum_{i=1}^t A_i^2$ \\
\qquad $^\star$Symmetric & $\Delta Y_t \sim -\Delta Y_t \mid \Fcal_{t-1}$
  & $\psi_N$ & $[Y]_t$ \\
\qquad SN I & $\E_{t-1} \Delta Y_t^2 < \infty$ & $\psi_N$
  & $([Y]_t + 2\eangle{Y}_t) / 3$ \\
\qquad SN II & $\E_{t-1} \Delta Y_t^2 < \infty$ & $\psi_N$
  & $([Y_+]_t + \eangle{Y_-}_t) / 2$ \\
\qquad Cubic SN & $\E_{t-1} \abs{\Delta Y_t}^3 < \infty$ & $\psi_G$
   & $[Y]_t + \sum_{i=1}^t \abs{\mu}^3_i$ \\
\bottomrule
\end{tabular}
\caption{Summary from \cref{th:discrete_facts,th:novel_discrete_conds} of
  sufficient conditions for an $\Hcal^d$-valued, discrete-time martingale
  $(Y_t)$ to have a sub-$\psi$ maximum eigenvalue process $S_t = \gammamax(Y_t)$
  with variance process $V_t = \gammamax(Z_t)$. We use the shorthand
  $\mu^k_t \defineas \E_{t-1} (\Delta S_t)^k$ and
  $\abs{\mu}^k_t \defineas \E_{t-1} \eabs{\Delta S_t}^k$. In the
  symmetric$^\star$ case, $\E_{i-1} \Delta Y_i$ need not exist, so $(Y_t)$ need
  not be a martingale.  ``$\implies$ Hoeffding I'' indicates that $(V_t)$ for
  Hoeffding-KS is smaller. ``SN'' is short for
  ``self-normalized''. \label{tab:sufficient_conditions_2}}
\end{table}

In \cref{tab:sufficient_conditions_1,tab:sufficient_conditions_2}, we summarize
a variety of standard and novel conditions for a process $(S_t)$ to be
sub-$\psi$. \Cref{th:discrete_facts,th:novel_discrete_conds} contain
discrete-time results, while results for continuous time are in
\cref{th:continuous_facts}. We let $I_d$ denote the $d~\times~d$ identity
matrix. For a process $(Y_t)_{t \in \Tcal}$, $[Y]_t$ denotes the quadratic
variation and $\eangle{Y}_t$ the conditional quadratic variation; in discrete
time, $[Y]_t \defineas \sum_{i=1}^t \Delta Y_i^2$ and
$\eangle{Y}_t \defineas \sum_{i=1}^t \E_{i-1} \Delta Y_i^2$. We extend a
function $f: \R \to \R$ on the real line to an operator
$f : \Hcal^d \to \Hcal^d$ on the space of Hermitian matrices in the standard
way: if $A \in \Hcal^d$ has the spectral decomposition $U \Lambda U^\star$ where
$\Lambda$ is diagonal with elements $\lambda_1, \dots, \lambda_d$, then
$f(A) = U f(\Lambda) U^\star$ where $f(\Lambda)$ is diagonal with elements
$f(\lambda_1), \dots, f(\lambda_d)$. In particular, the absolute value function
extends to $\Hcal^d$ by taking absolute values of the eigenvalues, while
$[Y_+]_t \defineas \sum_{i=1}^t \max(0, \Delta Y_i)^2$ and
$\eangle{Y_-}_t \defineas \sum_{i=1}^t \E_{i-1} \min(0, \Delta Y_i)^2$ operate
by truncating the eigenvalues.

In the discrete-time case, we have the following known results.

\begin{fact}\label{th:discrete_facts}
  Let $(Y_t)_{t \in \N}$ be any $\Hcal^d$-valued martingale, and let
  $S_t \defineas \gammamax(Y_t)$ for $t \in \N$. In all cases we set $l_0 = d$.
\begin{enumerate}[label=(\alph*)]
\item (Scalar parametric) If $d=1$ and $S_t$ is a cumulative sum of i.i.d.,
  real-valued random variables, each of which is mean zero with known CGF
  $\psi(\lambda)$ that is finite on $\lambda \in [0, \lambda_{\max})$, then
  $(S_t)$ is sub-$\psi$ with variance process $V_t = t$.
\item (Bernoulli I) If $-g I_d \preceq \Delta Y_t \preceq h I_d$ a.s.\ for all
  $t \in \N$, then $(S_t)$ is sub-Bernoulli with variance process
  $V_t = ght$ and range parameters $g,h$ \citep{hoeffding_probability_1963,
    tropp_user-friendly_2012}.
\item (Bennett) If $\Delta Y_t \preceq c I_d$ a.s.\ for all $t \in \N$ for some
  $c > 0$, then $(S_t)$ is sub-Poisson with variance process
  $V_t = \gammamax(\eangle{Y}_t)$ and scale parameter $c$
  \citep{bennett_probability_1962,hoeffding_probability_1963,
    tropp_user-friendly_2012}.
\item (Bernstein) If
  $\E_{t-1} (\Delta Y_t)^k \preceq (k!/2) c^{k-2} \E_{t-1}(\Delta Y_t)^2$ for
  all $t \in \N$ and $k = 2, 3, \dots$, then $(S_t)$ is sub-gamma with variance
  process $V_t = \gammamax(\eangle{Y}_t)$ and scale parameter $c$
  \citep{bernstein_theory_1927, tropp_user-friendly_2012,
    boucheron_concentration_2013}.
\item (Heavy on left) Let $T_a(y) \defineas (y \bmin a) \bmax -a$ for $a > 0$
  denote the truncation of $y$. If $d = 1$ and
  \begin{align}
    \E_{t-1} T_a(\Delta Y_t) \leq 0 \quad \text{for all } a > 0, t \in \N,
    \label{eq:heavy_on_left}
  \end{align}
  then $(S_t)$ is sub-Gaussian with variance process $V_t = \gammamax([Y]_t)$. A
  random variable satisfying \eqref{eq:heavy_on_left} is called \emph{heavy on
    left}, and $(Y_t)$ need not be a martingale in this case
  \citep{bercu_exponential_2008, delyon_exponential_2015,
    bercu_concentration_2015}. For example, the centered versions of the
  exponential, gamma, Pareto, log-normal, Poisson ($\lambda \in \N$), Bernoulli
  ($p < 1/2$) and geometric ($0<p<1$) distributions are known to be heavy on
  left. When $-\Delta Y_t$ satisfies \eqref{eq:heavy_on_left} we say
  $\Delta Y_t$ is \emph{heavy on right}.
\end{enumerate}
\end{fact}

In addition to the above known results, we provide the following extensions of
known scalar results to matrices. 

\begin{lemma} \label{th:novel_discrete_conds}
Let $(Y_t)_{t \in \N}$ be any $\Hcal^d$-valued martingale, and let
  $S_t \defineas \gammamax(Y_t)$ for $t \in \N$. In all cases we set $l_0 = d$.
\begin{enumerate}[label=(\alph*)]
\item (Bernoulli II) If, for all $t \in \N$, $\Delta Y_t \preceq h I_d$ a.s.\
  and $\E \Delta Y_t^2 \preceq gh I_d$, then $(S_t)$ is sub-Bernoulli with
  variance process $V_t = ght$.
\item (Hoeffding-KS) If $-G_t I_d \preceq \Delta Y_t \preceq H_t I_d$ a.s.\ for
  all $t \in \N$ for some real-valued, predictable sequences $(G_t)$ and
  $(H_t)$, then $(S_t)$ is sub-Gaussian with variance process
  ${V_t = \sum_{i=1}^t \varphi(G_i, H_i)}$, where
  \begin{align}
    \varphi(g,h) \defineas \begin{cases}
      \frac{h^2 - g^2}{2 \log(h/g)}, & g < h \\
      gh, & g \geq h.
    \end{cases} \label{eq:varphi_defn}
  \end{align}
\item (Hoeffding I) If $-G_t I_d \preceq \Delta Y_t \preceq H_t I_d$ a.s.\ for
  all $t \in \N$ for some real-valued, predictable sequences $(G_t)$ and
  $(H_t)$, then $(S_t)$ is sub-Gaussian with variance process
  ${V_t = \sum_{i=1}^t (G_i + H_i)^2/4}$.
\item (Conditionally symmetric) If $\Delta Y_t \sim -\Delta Y_t ~|~ \Fcal_{t-1}$
  for all $t \in \N$, then $(S_t)$ is sub-Gaussian with variance process
  $V_t = \gammamax([Y]_t)$. Here, $\Delta Y_t$ need not be integrable, so
  $(Y_t)$ need not be a martingale.
\item (Bounded from below) If $\Delta Y_t \succeq -c I_d$ a.s.\ for all
  $t \in \N$ for some $c > 0$, then $(S_t)$ is sub-exponential with variance
  process $V_t = \gammamax([Y]_t)$ and scale parameter $c$.
\item (General self-normalized I) If $\E_{t-1} \Delta Y_t^2$ is finite for all
  $t \in \N$, then $(S_t)$ is sub-Gaussian with variance process
  $V_t = \gammamax([Y]_t + 2 \eangle{Y}_t) / 3$.
\item (General self-normalized II) If $\E_{t-1} \Delta Y_t^2$ is finite for all
  $t \in \N$, then $(S_t)$ is sub-Gaussian with variance process
  $V_t = \gammamax([Y_+]_t + \eangle{Y_-}_t) / 2$.
\item (Hoeffding II) If $\Delta Y_t^2 \preceq A_t^2$ a.s.\ for all $t \in \N$
  for some $\Hcal^d$-valued predictable sequence $(A_t)$, then $(S_t)$ is
  sub-Gaussian with variance process $V_t = \gamma_{\max}(\sum_{i=1}^t A_i^2)$.
\item (Cubic self-normalized) If $\E_{t-1} \eabs{\Delta Y_t}^3$ is finite for
  all $t \in \N$, then $(S_t)$ is sub-gamma with variance process
  $V_t = \gammamax\eparen{[Y]_t + \sum_{i=1}^t \E_{i-1} \eabs{\Delta Y_i}^3}$
  and scale parameter $c = 1/6$.
\end{enumerate}
\end{lemma}

The proof of the above lemma can be found in
\cref{sec:proof_novel_discrete_conds}. Case (a) is a straightforward extension
of Bennett's condition for upper-bounded random variables with bounded variance
to matrices with upper-bounded eigenvalues and bounded matrix variance
\citep[p. 42]{bennett_probability_1962}. Cases (b) and (c) are similar
extensions of Hoeffding's sub-Gaussian conditions for bounded random variables
to matrices with bounded eigenvalues (\citealp[Theorems 1 and
2]{hoeffding_probability_1963}; \citealp{kearns_large_1998}; \citealp[Theorem
2.49]{bercu_concentration_2015}). In the conditionally symmetric case (d), we
can achieve control without any moment or boundedness assumptions by defining
$V_t$ in terms of observed rather than expected squared deviations; this is
known for $d=1$ (\citealp{de_la_pena_general_1999}, Lemma 6.1;
\citealp{bercu_concentration_2015}), allowing exponential concentration for
distributions like Cauchy. In the lower-bounded increments case (e), we have a
self-normalized complement to the Bennett-style bound, a result known for $d=1$
\citep[Lemma 4.1]{fan_exponential_2015}. For the square-integrable martingale
cases (f, g), we achieve control for a broad class of processes by incorporating
the conditional variance and the observed squared deviations, as known for $d=1$
(\citealp{delyon_exponential_2009}, Theorem 4;
\citealp{bercu_concentration_2015}). The Hoeffding-like case (h) follows from
the self-normalized bounds, highlighting a connection implicit in the proof of
Corollary 4.2 of \citet{mackey_matrix_2014}. The third moment bound (i) is
similar to a fixed-sample bound given by \citet[Corollary
2.2]{fan_exponential_2015}.

In the continuous-time, scalar case we have the following sufficient conditions
for a local martingale $(S_t)$ to be sub-$\psi$. Here we always assume $(S_t)$
is c\`adl\`ag, $\Delta S_t \defineas S_t - S_{t-}$ denotes the jumps of $S$,
$[S]_t$ denotes the quadratic variation, and $\eangle{S}_t$ is the conditional
quadratic variation, the compensator of $[S]_t$.

\begin{fact}\label{th:continuous_facts}
Here $\Tcal = (0, \infty)$ and $d = 1$, and we set $l_0 = 1$.
\begin{enumerate}[label=(\alph*)]
\item (L\'evy process) If $(S_t)$ is a L\'evy process which is a martingale with
  the CGF $\psi(\lambda) = \log \E e^{\lambda S_1} < \infty$ for all
  $\lambda \in [0, \lambda_{\max})$, then $(S_t)$ is sub-$\psi$ with variance
  process $V_t = t$. See, e.g., \citet[Proposition
  10.2]{papapantoleon_introduction_2008}.
\item (Continuous Bennett) If $(S_t)$ is a local martingale with
  $\Delta S_t \leq c$ for all $t$ a.s., then $(S_t)$ is sub-Poisson with scale
  parameter $c$ and variance process $V_t = \eangle{S}_t$
  \citep[p.\ 157]{lepingle_sur_1978}.
\item (Continuous Bernstein) Suppose $(S_t)$ is a locally square integrable
  martingale: let $W_{2,t} = \eangle{S}_t$, and for $m = 3, 4, \dots$ let
  $W_{m,t}$ be the compensator of the process
  $\sum_{u \leq t} \abs{\Delta S_u}^m$. If, for some $c > 0$ and predictably
  measurable, c\`adl\`ag, nondecreasing process $(V_t)$, it holds that
  $W_{m,t} \leq \frac{m!}{2} c^{m-2} V_t$ for all $m \geq 2$, then $(S_t)$ is
  sub-gamma with scale parameter $c$ and variance process $V_t$
  \citep[implicit in the proof of Lemma 2.2]{van_de_geer_exponential_1995}.
\item (Continuous paths) If $(S_t)$ is a local martingale with a.s.\ continuous
  paths, then $(S_t)$ is sub-Gaussian with variance process
  $V_t = \eangle{S}_t$. This may be seen as a special case of (c), or a limiting
  case of (b).
\end{enumerate}
\end{fact}

\subsection{Implications between sub-$\psi$ conditions}

In many settings, a process $(S_t)$ may satisfy \cref{th:canonical_assumption}
with several different choices of $\psi$ and $(V_t)$. Choosing a smaller $\psi$
function will lead to tighter bounds in \cref{th:uniform_chernoff}, but in some
cases one may opt for a larger $\psi$ function to achieve analytical or
computational convenience. It is clear that making $\psi$ uniformly larger
retains the sub-$\psi$ property, since the exponential process
$\expebrace{\lambda S_t - \psi(\lambda) V_t}$ can only become smaller. It is
therefore useful to characterize relationships among the above sub-$\psi$
conditions, so that, after invoking one of the sufficient conditions given in
\cref{sec:sub_psi_conds}, one may invoke \cref{th:uniform_chernoff} with a
different, more convenient $\psi$ function.

\begin{figure}[h!]
  \centering
  \begin{tikzpicture}[node distance = 4.5cm, auto]
    \node[block2] (ber) {Sub-Bernoulli};
    \node[block2, below of=ber, node distance=1.4cm] (norm) {Sub-Gaussian};
    \node[block2, right of=ber] (poi) {Sub-Poisson};
    \node[block2, right of=poi] (gam) {Sub-gamma};
    \node[block2, below of=gam, node distance=1.4cm] (exp)
      {Sub-exponential};
    \path [line] (ber) -- (norm);
    \path [line] (ber.10) -- (poi.170);
    \path [line] (poi.190) -- node[below] {$c < 0$} (ber.350);
    \path [line] (norm.5) to[out=0, in=270] (poi.220);
    \path [line] (poi.270) to[out=270, in=0] node[below right] {$c < 0$}
      (norm.350);
    \path [line] (poi.10) -- (gam.170);
    \path [line] (gam.190) -- node[below] {$c < 0$} (poi.350);
    \path [line] (gam.300) -- (exp.60);
    \path [line] (exp.120) -- (gam.240);
  \end{tikzpicture}
  \caption{Each arrow indicates that any process satisfying the source
    sub-$\psi$ condition, subject to a restriction on the scale parameter $c$,
    also satisfies the destination sub-$\psi$ condition with appropriately
    scaled variance process. See \cref{tab:psi_relations} and
    \cref{th:psi_relations} for details. \label{fig:psi_relations}}
\end{figure}
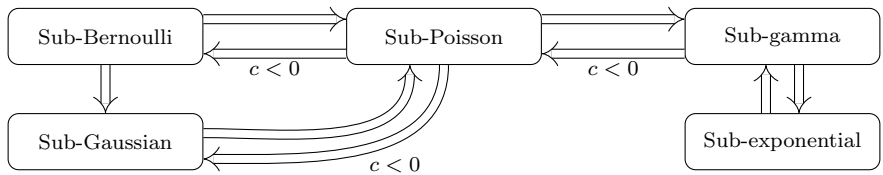

\begin{table}[h!]
\centering
\renewcommand{\arraystretch}{1.1}
\begin{tabular}{lllll}
\toprule
& $\psi_1$ & $\psi_2$ & $a$ & Restriction \\
\midrule
(1) & $\psi_{B,g,h}$ & $\psi_N$ & $\frac{\varphi(g, h)}{gh}$ & \\
(2) & $\psi_{B,g,h}$ & $\psi_N$ & $\frac{(g+h)^2}{4gh}$ & \\
(3) & $\psi_{B,g,h}$ & $\psi_{P,h-g}$ & 1 & \\
(4) & $\psi_{N}$ & $\psi_{P,0}$ & 1 & \\
(5) & $\psi_{P,c}$ & $\psi_{G,c/3}$ & 1 & \\
(6) & $\psi_{G,c}$ & $\psi_{E,3c/2}$ & 1 & \\
(7) & $\psi_{E,c}$ & $\psi_{G,c}$ & 1 & $c \geq 0$ \\
(8) & $\psi_{E,c}$ & $\psi_{G,c/2}$ & 1 & $c < 0$ \\
(9) & $\psi_{G,c}$ & $\psi_{P,2c}$ & 1 & $c < 0$ \\
(10) & $\psi_{P,c}$ & $\psi_N$ & 1 & $c < 0$ \\
(11) & $\psi_{P,c}$ & $\psi_{B,-c,h}$ & 1 & $c < 0$, any $h > 0$ \\
\bottomrule
\end{tabular}
\caption{If $(S_t)$ is sub-$\psi_1$ with variance process $(V_t)$,
  subject to the given restriction, then $(S_t)$ is also sub-$\psi_2$ with
  variance process $(aV_t)$. $\varphi(g,h)$ is defined in
  \eqref{eq:varphi_defn}. See \cref{th:psi_relations} for
  details. \label{tab:psi_relations}}
\end{table}

Note that $\psi_G$, $\psi_P$ and $\psi_E$ are nondecreasing in
  $c$ for all values of $\lambda \geq 0$, so that if a process is sub-$\psi$
  with scale $c$ for any of these $\psi$ functions, then it is sub-$\psi$ for
  any scale $c' > c$ as well. Similarly, $\psi_B$ is nonincreasing in $g$ and
  nondecreasing in $h$. \Cref{tab:psi_relations} and \cref{th:psi_relations}
  fully characterize all implications among sub-$\psi$ conditions, as
  illustrated in \cref{fig:psi_relations}. These follow from inequalities of the
  form $\psi_1 \leq a \psi_2$, some of which are based on standard arguments; see \cref{sec:proof_psi_relations}.

\begin{proposition}\label{th:psi_relations}
  For each row in \cref{tab:psi_relations}, if $(S_t)$ is sub-$\psi_1$ with
  variance process $(V_t)$, and the given restrictions are satisfied, then
  $(S_t)$ is also sub-$\psi_2$ with variance process $(aV_t)$. Furthermore, when
  we allow only scaling of $V_t$ by a constant, these capture all possible
  implications among the five sub-$\psi$ conditions defined above, and the given
  constants are the best possible (in the case of row (2), the constant
  $(g+h)^2 / 4gh$ is the best possible of the form $k/gh$ where $k$ depends only
  on the total range $g + h$).
\end{proposition}

\section{Applications of \cref{th:uniform_chernoff}}
\label{sec:special_cases}

Here, we illustrate how \cref{th:uniform_chernoff} recovers or
strengthens a wide variety of existing results. Most results in this section
follow immediately upon combining one of the sufficient conditions from
\cref{th:discrete_facts}, \cref{th:novel_discrete_conds}, or
\cref{th:continuous_facts} with \cref{th:uniform_chernoff}, and we omit proof
details in many cases. As a rough plan, we first discuss classical
Cram\'er-Chernoff and Freedman-style bounds and then Blackwell's line crossing
inequalities. After discussing de la Pe\~na-style self-normalized bounds and
Pinelis' Banach-space inequalities, we end by exhibiting some continuous time
results and mention connections to the sequential probability ratio test.

\subsection{Fixed-time Cram\'er-Chernoff and Freedman-style uniform
  bounds}
\label{sec:chernoff}

\begin{figure}
\includegraphics{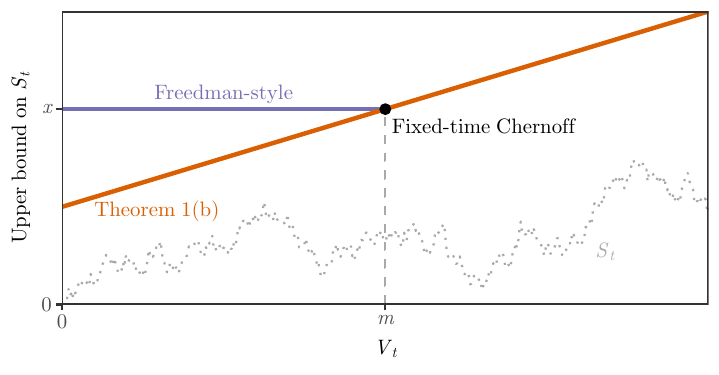}
\caption{Comparison of $(i)$ fixed-time Cram\'er-Chernoff bound
  \eqref{eq:basic_chernoff}, which bounds the deviations of $S_m$ at a fixed
  time $m$; $(ii)$ ``Freedman-style'' constant uniform bound \eqref{eq:freedman},
  which bounds the deviations of $S_t$ for all $t$ such that $V_t \leq m$, with
  a constant boundary equal in value to the fixed-time Cram\'er-Chernoff bound;
  and $(iii)$ linear uniform bound from \cref{th:uniform_chernoff}(b), which
  bounds the deviations of $S_t$ for all $t \in \N$, with a boundary growing
  linearly in $V_t$. Each bound gives the same tail probability and thus implies
  the preceding one. \label{fig:three_bounds}}
\end{figure}

In the discrete-time, scalar setting, a simple sufficient condition for a
process $(S_t)$ to be 1-sub-$\psi$ with variance process $(V_t)$ is that
\begin{align}
\E_{t-1} \expebrace{\lambda \Delta S_t - \psi(\lambda) \Delta V_t} \leq 1,
  \quad \forall t,
\end{align}
which is the standard assumption for a martingale-method Cram\'er-Chernoff
inequality, typically with $(V_t)$ predictable
\citep{mcdiarmid_concentration_1998, chung_concentration_2006,
  boucheron_concentration_2013}. When $(V_t)$ is deterministic, the fixed-time
Cram\'er-Chernoff method gives, for fixed $x$ and $m$,
\begin{align}
\P(S_m \geq x) \leq \expebrace{-V_m \psi^\star\pfrac{x}{V_m}},
\label{eq:basic_chernoff}
\end{align}
so \cref{th:uniform_chernoff}(b) is a uniform \emph{extension} of the
Cram\'er-Chernoff inequality, losing nothing at the fixed time $m$ [B; C or
D]. For random $(V_t)$, a stopping time argument due to
\citet{freedman_tail_1975} extends this to the uniform bound
\begin{align}
\P(\exists t \in \Tcal: S_t \geq x \text{ and } V_t \leq m)
  \leq \expebrace{-m \psi^\star\pfrac{x}{m}}. \label{eq:freedman}
\end{align}
When $(V_t)$ is deterministic, analogous uniform bounds can be obtained from
Doob's maximal inequality for submartingales, as in
\citet[eq.\ 2.17]{hoeffding_probability_1963}. \Cref{th:uniform_chernoff}
strengthens this ``Freedman-style'' inequality [B; C or D], since it yields
tighter bounds for all times $t$ such that $V_t < m$, and also extends the
inequality to hold for all times $t$ with $V_t > m$, as illustrated by
\cref{fig:three_bounds}.

\citet{tropp_freedmans_2011, tropp_user-friendly_2012} extends the scalar
Cram\'er-Chernoff approach to random matrices via control of the matrix
moment-generating function, giving matrix analogues of Hoeffding's, Bennett's,
Bernstein's and Freedman's inequalities. Following this approach,
\cref{th:uniform_chernoff} gives corresponding strengthened versions of these
inequalities for matrix-valued processes [B].

We summarize explicit results below for three well-known special cases reviewed
in \cref{th:intro_example}(a): Hoeffding's sub-Gaussian inequality for
observations bounded from above and below, with variance process depending only
on the radius of the interval of boundedness \citep{hoeffding_probability_1963};
Bennett's sub-Poisson inequality for observations bounded from above, with
variance process depending on the true variance of the observations
\citep{bennett_probability_1962}; and Bernstein's sub-gamma inequality for
observations satisfying a bound on growth of higher moments, also with a
variance process depending on the true variance
\citep{bernstein_theory_1927}. In each case below, we recover the standard,
fixed-sample result at $V_t = m$. Recall the definitions of
$\slope_P,\psi^\star_P,\slope_G,\psi^\star_G$ from \cref{tab:psi_transforms}.

\begin{corollary}\label{th:chernoff_cases}
\begin{enumerate}[label=(\alph*)]
\item Suppose $(Y_t)_{t \in \N}$ is an $\Hcal^d$-valued martingale whose
  increments satisfy $\Delta Y_t^2 \preceq A_t^2$ a.s.\ for all $t$ for some
  $\Hcal^d$-valued, predictable sequence $(A_t)$. Let
  $S_t \defineas \gamma_{\max}(Y_t)$, and let either
  \begin{align}
    V_t \defineas
      \frac{1}{2}\gamma_{\max}\eparen{\eangle{Y}_t + \sum_{i=1}^t A_i^2}
    \quad \text{or} \quad
    V_t \defineas \gamma_{\max}\eparen{\sum_{i=1}^t A_i^2}.
  \end{align}
  Then for any $x, m > 0$, we have
  \begin{align}
    \P\eparen{
      \exists t \in \N : S_t \geq x + \frac{x}{2m}(V_t - m)
    } \leq d \expebrace{-\frac{x^2}{2m}}.
  \end{align}
  This strengthens Hoeffding's inequality \citep{hoeffding_probability_1963}
  [A,B,D] and its matrix analogues in \citet[Theorem
  7.1]{tropp_user-friendly_2012} [B,E] and \citet[Corollary
  4.2]{mackey_matrix_2014} [A,B].
\item Suppose $(Y_t)_{t \in \N}$ is an $\Hcal^d$-valued martingale satisfying
  $\gamma_{\max}(\Delta Y_t) \leq c$ a.s. for all $t$. Let
  $S_t \defineas \gamma_{\max}(Y_t)$ and
  $V_t \defineas \gamma_{\max}(\eangle{Y}_t)$. Then for any $x, m > 0$, we have
  \begin{multline}
    \P\eparen{
      \exists t \in \N : S_t \geq
      x + \slope_P\pfrac{x}{m} \cdot (V_t - m)
    } \leq d \expebrace{-m \psi^\star_P\pfrac{x}{m}}
      \\ \leq d \expebrace{-\frac{x^2}{2(m + cx/3)}}.
  \end{multline}
  This strengthens Bennett's and Freedman's inequalities
  \citep{bennett_probability_1962, freedman_tail_1975} [B; C or D] for
  scalars and the corresponding matrix bounds from
  \citet{tropp_freedmans_2011, tropp_user-friendly_2012} [B].
\item Suppose $(S_t)$ is $l_0$-sub-gamma with variance process $(V_t)$ and scale
  parameter $c$. Then for any $x,m > 0$, we have
  \begin{multline}
    \P\eparen{
      \exists t \in \Tcal : S_t \geq
      x + \slope_G\pfrac{x}{m} \cdot (V_t - m)
    } \leq l_0 \expebrace{-m \psi^\star_G\pfrac{x}{m}}
      \\ \leq l_0 \expebrace{-\frac{x^2}{2(m + cx)}}.
  \end{multline}
  This strengthens Bernstein's inequality \citep{bernstein_theory_1927} [B;
  C or D], along with the matrix Bernstein inequality
  \citep{tropp_user-friendly_2012} [B].
\end{enumerate}
\end{corollary}

Case (a) is a consequence of \cref{th:novel_discrete_conds}(g); see also
\cref{th:delyon_uniform}, which uses
$V_t = \frac{1}{2} \gamma_{\max}([Y_+]_t + \eangle{Y_-}_t)$. The first setting
of $V_t$ in case (a) follows from the bound
$[Y_+]_t \preceq \sum_{i=1}^t A_i^2$, and further upper bounding
$\eangle{Y_-}_t \preceq \sum_{i=1}^t A_i^2$ yields the second setting of
$V_t$. As is well known, the Hoeffding-style bound in part (a) and the
Bennett-style bound in part (b) are not directly comparable: $V_t$ may be
smaller in part (b), but $\psi^\star_P \leq \psi^\star_N$, so neither subsumes
the other. We remark that
$\psi^\star_P(u) \geq \frac{u}{2c} \arcsinh\pfrac{cu}{2}$, so the Bennett-style
inequality in part (b) is an improvement on the inequality of
\citet{prokhorov_extremal_1959} for sums of independent random variables, as
noted by \citet{hoeffding_probability_1963}, as well as its extension to
martingales in \citet{de_la_pena_general_1999}.

As an example of the Hermitian dilation technique for extending bounds on
Hermitian matrices to bounds for rectangular matrices, we give a bound for
rectangular matrix Gaussian and Rademacher series, following
\citet{tropp_user-friendly_2012}; here $\opnorm{A}$ denotes the largest singular
value of $A$. The proof is in \cref{sec:proof_rectangular_series}.
\begin{corollary}\label{th:rectangular_series}
  Consider a sequence $(B_t)_{t \in \N}$ of fixed matrices with dimension
  $d_1 \times d_2$, and let $(\epsilon_t)_{t \in \N}$ be a sequence of
  independent standard normal or Rademacher variables. Let
  $S_t \defineas \opnorm{\sum_{i=1}^t \epsilon_i B_i}$ and
  \begin{align}
    V_t \defineas \max\ebrace{\opnorm{\sum_{i=1}^t B_i B_i^\star},
        \opnorm{\sum_{i=1}^t B_i^\star B_i}}.
  \end{align}
  Then for any $x, m > 0$, we have
\begin{align}
  \P\eparen{
    \exists t \in \N : S_t \geq x + \frac{x}{2m}(V_t - m)
  } \leq (d_1 + d_2) \expebrace{-\frac{x^2}{2m}}.
\end{align}
This strengthens Corollary 4.2 of \citet{tropp_user-friendly_2012} [B].
\end{corollary}

\subsection{Line-crossing inequalities}\label{sec:line_crossing}

Before giving specific results in this section, we start with simplified
versions of \cref{th:uniform_chernoff}(d) which are useful for recovering
existing results. The probability bound in \eqref{eq:late_crossing_approx} is
merely an analytically simplified upper bound on that from
\cref{th:uniform_chernoff}(d). We prove the following in
\cref{sec:proof_late_crossing_simplified}.

\begin{corollary}\label{th:late_crossing_simplified}
  If $(S_t)$ is $l_0$-sub-$\psi$ with variance process $(V_t)$ and $\psi$ is
  CGF-like, then for any $m \geq 0$, $x > 0$ and $b \in (0, \bar{b})$, we have
\begin{multline}
\P\eparen{\exists t \in \Tcal: V_t \geq m \text{ and } S_t \geq x + b (V_t - m)}
  \\ \leq l_0 \expebrace{-m \psi^\star(b) - (x - bm) \psistarprime(b)}.
    \label{eq:late_crossing_approx}
\end{multline}
In particular, for $m > 0$, we have
\begin{align}
  \P\eparen{\exists t \in \Tcal: V_t \geq m \text{ and } S_t \geq b V_t}
  \leq l_0 \expebrace{-m \psi^\star(b)}. \label{eq:line_thru_origin}
\end{align}
\end{corollary}

In fitting with the approach of this paper, \cref{th:uniform_chernoff}(d) and
\cref{th:late_crossing_simplified} bound the upcrossing probability on
$\brace{V_t \geq m}$ using the results of \cref{th:uniform_chernoff}(a,b) and a
geometric argument. It may seem naive and wasteful to bound a line-crossing
probability on $\brace{V_t \geq m}$ using a bound which applies for
$\brace{V_t > 0}$. The literature includes a handful of results bounding
line-crossing probabilities on $\brace{V_t \geq m}$ which appear to give bounds
tighter than what \cref{th:uniform_chernoff} offers, by making more direct use
of the intrinsic-time condition \citep{blackwell_large_1997,
  khan_$l_p$-version_2009}. Below we demonstrate that this is not true: we give
several special cases of \cref{th:uniform_chernoff}(d) and
\cref{th:late_crossing_simplified} which improve upon existing results.

\begin{corollary}\label{th:gamma_line_crossing}
  Suppose $(S_t)$ is $l_0$-sub-gamma with variance process $(V_t)$ and scale
  parameter $c$.
\begin{enumerate}[label=(\alph*)]
\item For any $a, b > 0$, we have
  \begin{align}
    \P\eparen{\exists t \in \Tcal : S_t \geq a + b V_t}
      \leq l_0 \expebrace{-\frac{2ab}{1 + 2cb}}.
  \end{align}
  When $\Tcal = \N$, $c=0$ and $d=1$ this strengthens Theorem 1 of
  \citet{blackwell_large_1997} [A; C or D], which is written for
  discrete-time scalar processes with bounded increments.
\item For any $m, b > 0$, we have
  \begin{multline}
    \P\eparen{\exists t \in \Tcal : V_t \geq m \text{ and } S_t \geq b V_t}
    \leq l_0 \expebrace{-m \psi^\star_G(b)}
    \\ \leq l_0 \expebrace{-\frac{b^2 m}{2(1 + cb)}}.
  \end{multline}
  When $\Tcal = \N$, $c=0$ and $d=1$ this strengthens the second bound in
  Theorem 2 of \citet{blackwell_large_1997} [A; C or D], which is written for
  discrete-time scalar processes with bounded increments.
\end{enumerate}
\end{corollary}

In discrete time, as presented in \cref{th:discrete_facts}, for a process with
bounded increments we may construct both sub-Bernoulli and sub-Gaussian
bounds. The sub-Bernoulli case, in combination with \eqref{eq:line_thru_origin},
yields the following:
\begin{corollary}\label{th:bernoulli_line_origin}
  Suppose $(Y_t)_{t \in \N}$ is an $\Hcal^d$-valued martingale satisfying
  $\opnorm{\Delta Y_t} \leq 1$ a.s.\ for all $t \in \N$. Then for any
  $b \in [0, 1]$ and $m \geq 1$, we have
\begin{align}
  \P\eparen{
    \exists t \in \N: t \geq m \text{ and } \gamma_{\max}(Y_t) \geq b t
  } \leq \ebracket{(1+b)^{(1+b)} (1-b)^{(1-b)}}^{-m/2}.
\end{align}
This strengthens the first bound in Theorem 2 of \citet{blackwell_large_1997}
[D].
\end{corollary}

Theorems 4.1-4.3 of \citet{khan_$l_p$-version_2009} are closest in form to our
main results and represent key precedents to our framework. The simplified bound
\eqref{eq:late_crossing_approx} recovers Khan's Theorem 4.3 [C or D], while
\cref{th:uniform_chernoff}(d) improves the exponent [E]. Our
\cref{th:uniform_chernoff}(b) gives a strengthened version of Khan's
``Freedman-style'' Theorem 4.2 [B; C or D]. Khan's Theorem 4.1 is not strictly
comparable to our work since it involves an initial condition on \emph{nominal}
time, $t \geq t_0$, rather than on intrinsic time, $V_t \geq m$, but when $V_t$
is deterministic, then our \cref{th:uniform_chernoff}(d) is tighter [B; C or D;
E].

\subsection{Self-normalized uniform bounds}

Collectively, \citet{de_la_pena_general_1999, de_la_pena_moment_2000,
  de_la_pena_self-normalized_2004, de_la_pena_pseudo-maximization_2007,
  de_la_pena_theory_2009}; and \citet{de_la_pena_self-normalized_2009} give a
wide variety of sufficient conditions for the exponential process
$\expebrace{\lambda S_t - \psi(\lambda) V_t}$ to be a supermartingale in both
discrete- and continuous-time settings. They formulate their bounds for ratios
involving $S_t$ in the numerator and $V_t$ in the denominator, as in
\cref{th:uniform_chernoff}(c), and often specify initial-time conditions, as in
\cref{th:uniform_chernoff}(d). In this section we draw some comparisons between
\cref{th:uniform_chernoff} and their results. As a first example, consider the
boundary of \cref{th:uniform_chernoff}(c) for the ratio $S_t / V_t$, strictly
decreasing towards the asymptotic level $\slope(x)$. In particular, at time
$V_t = m$ the boundary equals $x$, so \cref{th:uniform_chernoff}(c) strengthens
various theorems of \citet{de_la_pena_general_1999} and
\citet{de_la_pena_pseudo-maximization_2007} which use a constant boundary after
time $V_t = m$ [B; C or D]; for example, Theorem 1.2B, eq.\ 1.5 of
\citet{de_la_pena_general_1999} states that
\begin{align}
  \P\eparen{
    \exists t \geq 1:
    V_t \geq m \text{ and } \frac{S_t}{V_t} \geq x
  } \leq \expebrace{-m \psi_G^\star(x)} \label{eq:de_la_pena}
\end{align}
for scalar processes $(S_t)$ which are 1-sub-gamma with variance process
$(V_t)$. As before, we give explicit results for special cases.

\begin{figure}
\includegraphics{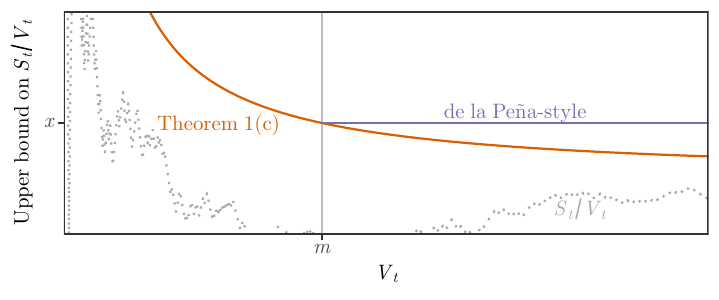}
\caption{Comparison of our decreasing boundary from
  \cref{th:uniform_chernoff}(c), as in inequality~\eqref{eq:dlp_strengthened},
  to a ``de la Pe\~na-style'' constant uniform bound as in
  inequality~\eqref{eq:de_la_pena}, which bounds the deviations of $S_t / V_t$
  for all $t$ such that $V_t \geq m$ with a constant
  boundary. \label{fig:ratio_constant}}
\end{figure}

\begin{corollary}\label{th:variance_normalized}
  Suppose $(S_t)$ is $l_0$-sub-gamma with variance process $(V_t)$ and scale
  parameter $c$. Then for any $x, m > 0$, we have
  \begin{multline}
    \P\eparen{
      \exists t \in \Tcal : \frac{S_t}{V_t}
      \geq \slope_G(x) \eparen{1 + \frac{m\sqrt{1 + 2cx}}{V_t}}
    } \leq l_0 \expebrace{-m \psi^\star_G(x)}
      \\ \leq l_0 \expebrace{-\frac{m x^2}{2(1 + cx)}}.
    \label{eq:dlp_strengthened}
  \end{multline}
  This strengthens eq.\ 1.5 from Theorem 1.2B of \citet{de_la_pena_general_1999}
  [B; C or D]. In the sub-Gaussian case (obtained at $c = 0$), the above bound
  simplifies to
  \begin{align}
    \P\eparen{
      \exists t \in \Tcal :
      \frac{S_t}{V_t + m} \geq x
    } \leq l_0 \expebrace{-2 m x^2}.
  \end{align}
  This strengthens Theorem 2.1 of \citet{de_la_pena_pseudo-maximization_2007}
  and Theorem 6.1 of \citet{de_la_pena_general_1999} [B, C~or~D].
\end{corollary}
Recall that $\slope_G(x) = x / (1 + \sqrt{1 + 2cx})$, so for the boundary in
\eqref{eq:dlp_strengthened}, we have
$\slope_G(x) (1 + m\sqrt{1 + 2cx} / V_t) \leq x$ for all $V_t \geq m$ with
equality at $V_t = m$. \Cref{th:variance_normalized}(a), therefore, gives the same
probability bound as \eqref{eq:de_la_pena} for a larger crossing
event. \Cref{fig:ratio_constant} visualizes this relationship.

More generally, when we normalize by $\alpha + \beta V_t$ and include an initial
time condition $V_t \geq m$, \cref{th:uniform_chernoff}(d) and
\cref{th:late_crossing_simplified} become the following:

\begin{corollary}\label{th:ratio_subg}
  If $(S_t)$ is $l_0$-sub-$\psi$ with variance process $(V_t)$, where $\psi$ is
  CGF-like and $\bar{b} = \infty$, then for any $\beta, x > 0$ and
  $\alpha, m \geq 0$ with at least one of $\alpha, m > 0$, we have
\begin{align}
  &\P\eparen{\exists t \in \Tcal: V_t \geq m
    \text{ and } \frac{S_t}{\alpha + \beta V_t} \geq x} \\
  &\qquad \leq \begin{cases}
    l_0 \expebrace{-\alpha x \decay(\beta x)},
      & \beta x \leq \slope\pfrac{x (\alpha + \beta m)}{m} \\
    l_0 \expebrace{-m \psi^\star\pfrac{x(\alpha + \beta m)}{m}},
      & \beta x \geq \slope\pfrac{x (\alpha + \beta m)}{m}
  \end{cases} \\
  &\qquad \leq l_0
    \expebrace{-m \psi^\star(\beta x) - \alpha x \psistarprime(\beta x)}.
\end{align}
In the case $(S_t)$ is sub-Gaussian, for any $\beta, x > 0$ and
$\alpha, m \geq 0$ with at least one of $\alpha, m > 0$, we have
\begin{multline}
  \P\eparen{\exists t \in \Tcal: V_t \geq m
    \text{ and } \frac{S_t}{\alpha + \beta V_t} \geq x}
  \\ \leq \expebrace{
    -x^2 \eparen{
      2 \alpha \beta
      + \frac{(\beta m - \alpha)^2 \indicator{\alpha \leq \beta m}}{2m}
    }
  },
\end{multline}
taking $0/0 = 0$ on the right-hand side when $m=0$. With
\cref{th:novel_discrete_conds}(d), this improves eq.\ 6.4 from Theorem 6.2 of
\citet{de_la_pena_general_1999} [C or D; E].
\end{corollary}

A defining feature of self-normalized bounds is that they involve a variance
process $(V_t)$ constructed with the squared observations themselves rather than
just conditional variances or constants. Such normalization can be found in
common statistical procedures such as the $t$-test. Furthermore, it allows for
Gaussian-like concentration while reducing or eliminating moment conditions.
\Cref{th:novel_discrete_conds} gives several extensions of well-known conditions
for scalar sub-Gaussian concentration of self-normalized processes. As one
particular special case, \cref{th:novel_discrete_conds}(f) and (g) yield general
self-normalized uniform bounds for any discrete-time, square-integrable,
$\Hcal^d$-valued martingale, building upon breakthrough results obtained for
scalar processes by Bercu, Touati and Delyon:

\begin{corollary}\label{th:delyon_uniform}
  Suppose $(Y_t)_{t \in \N}$ is an $\Hcal^d$-valued martingale with
  $\E Y_t^2 < \infty$ for all $t \in \N$. Let
  ${S_t \defineas \gamma_{\max}(Y_t)}$ and either
  \begin{align}
    V_t \defineas \frac{1}{2} \gamma_{\max}([Y_+]_t + \eangle{Y_-}_t)
    \quad \text{or} \quad
    V_t \defineas \frac{1}{3} \gamma_{\max}([Y]_t + 2 \eangle{Y}_t).
  \end{align}
  Then for any $x, m > 0$, we have
\begin{align}
\P\eparen{\exists t \in \N: \frac{S_t}{V_t + m} \geq x}
  \leq d \expebrace{-2mx^2}.
\end{align}
This strengthens eq.\ 20 from Theorem 4 of \citet{delyon_exponential_2009}
[B,D], Theorem 2.1 of \citet{bercu_exponential_2008} [B,D,E], and an implicit
self-normalized bound of \citet[Corollary 4.2]{mackey_matrix_2014} [B].
\end{corollary}

\Cref{th:delyon_uniform} is remarkable for the fact that it gives Gaussian-like
concentration with only the existence of second moments for the increments. If
the increments have conditionally symmetric distributions, one may instead apply
\cref{th:novel_discrete_conds}(d) to achieve Gaussian-like concentration without
existence of any moments, as discovered by \citet{de_la_pena_general_1999} and
illustrated in the following example.
\begin{example}[Cauchy increments]\label{ex:cauchy}
  Let $(\Delta S_t)_{t \in \N}$ be i.i.d.\ standard Cauchy random
  variables (symmetric about zero).
  \cref{th:novel_discrete_conds}(d) shows that $(S_t)$ is sub-Gaussian with
  variance process $V_t = [S]_t$. \cref{th:variance_normalized} yields,
  for any $x, m > 0$,
  \begin{align}
    \P\eparen{
      \exists t \in \N: \frac{S_t}{[S]_t + m}
      \geq x
    } \leq \expebrace{-2 m x^2}.
  \end{align}
\end{example}
For another example, \cref{th:novel_discrete_conds}(i) gives a self-normalized
bound involving third rather than second moments:

\begin{corollary}\label{th:fan_uniform}
  Suppose $(Y_t)_{t \in \N}$ is an $\Hcal^d$-valued martingale with
  $\E \eabs{Y_t}^3$ finite for all $t \in \N$. Let
  ${S_t \defineas \gamma_{\max}(Y_t)}$ and
  ${V_t \defineas \gamma_{\max}([Y]_t + \sum_{i=1}^t \E_{i-1} (\Delta
    Y_i)_-^3)}$. Then for $\slope_G$ and $\psi^\star_G$ using $c = 1/6$, we have for any $x, m > 0$,
\begin{align}
\P\eparen{\exists t \in \N: S_t \geq x + \slope_G\pfrac{x}{m} \cdot (V_t - m)}
  &\leq d \expebrace{-m \psi^\star_G\pfrac{x}{m}} \\
  &\leq d \expebrace{-\frac{x^2}{2(m + x/6)}}. \label{eq:third_moment_bound}
\end{align}
This is a uniform alternative
to Corollary 2.2 of \citet{fan_exponential_2015} [B,D].
\end{corollary}

Note the exponent in \eqref{eq:third_moment_bound} is different from that in
\citet{fan_exponential_2015}; neither strictly dominates the other. Also
note that, unlike the classical Bernstein bound, neither of
\cref{th:delyon_uniform,th:fan_uniform} assume existence of moments of all
orders.

\subsection{Martingales in smooth Banach spaces}\label{sec:banach}

The applications presented thus far allow us to uniformly bound the operator
norm deviations of a sequence of random Hermitian matrices. A different approach
is due to \citet{pinelis_approach_1992, pinelis_optimum_1994}, who gave an
innovative approach to exponential tail bounds in abstract Banach spaces. We
describe how this approach can be incorporated into our framework. For this
section, let $(Y_t)_{t \in \N}$ be a martingale with respect to $(\Fcal_t)$
taking values in a separable Banach space $(\Xcal, \norm{\cdot})$. We can use
Pinelis's device to uniformly bound the process $(\Psi(Y_t))$ for any function
$\Psi : \Xcal \to \R$ which satisfies the following smoothness property:
\begin{definition}[\citealp{pinelis_optimum_1994}]
  A function $\Psi : \Xcal \to \R$ is called \emph{$(2,D)$-smooth} for some
  $D > 0$ if, for all $x, v \in \Xcal$, we have
\begin{subequations}
\begin{align}
\Psi(0) &= 0 \label{eq:d_smooth_1} \\
\abs{\Psi(x + v) - \Psi(x)} &\leq \norm{v} \label{eq:d_smooth_2} \\
\Psi^2(x + v) - 2 \Psi^2(x) + \Psi^2(x - v)
  &\leq 2D^2 \norm{v}^2. \label{eq:d_smooth_3}
\end{align}
\end{subequations}
\end{definition}
A Banach space is called $(2,D)$-smooth if its norm is $(2,D)$-smooth; in such a
space we may take $\Psi(\cdot) = \norm{\cdot}$ to uniformly bound the deviations
of a martingale. In this case, observe that property \eqref{eq:d_smooth_1} is
part of the definition of a norm, property \eqref{eq:d_smooth_2} is the triangle
inequality, and property \eqref{eq:d_smooth_3} can be seen to hold with $D = 1$
for the norm induced by the inner product in any Hilbert space, regardless of
the (possibly infinite) dimensionality of the space. Note also that setting
$x = 0$ shows that $D \geq 1$ whenever $\Psi(\cdot) = \norm{\cdot}$. Finally,
observe that if we write $f(x) = \Psi^2(x)$, then we may equivalently replace
condition \eqref{eq:d_smooth_3} by
$f(tx + (1-t)y) \geq tf(x) + (1-t)f(y) - D^2 t (1 - t) \norm{x - y}^2$, a
perhaps more familiar definition of smoothness.

\begin{corollary}\label{th:pinelis}
  Consider a martingale $(Y_t)_{t \in \N}$ taking values in a separable Banach
  space $(\Xcal, \norm{\cdot})$. Let the function $\Psi : \Xcal \to \R$ be
  $(2,D)$-smooth and define $D_\star \defineas 1 \bmax D$.
\begin{enumerate}[label=(\alph*)]
\item Suppose $\norm{\Delta Y_t} \leq c_t$ a.s.\ for all $t \in \N$ for some
  constants $(c_t)_{t \in \N}$, and let $V_t \defineas \sum_{i=1}^t c_i^2$. Then
  for any $x, m > 0$, we have
  \begin{align}
    \P\eparen{\exists t \in \N:
              \Psi(Y_t) \geq x + \frac{x}{2m}(V_t - m)}
      \leq 2 \expebrace{-\frac{x^2}{2 D_\star^2 m}}.
      \label{eq:pinelis_hoeffding}
  \end{align}
  This strengthens Theorem 3.5 from \citet{pinelis_optimum_1994} [B].
\item Suppose $\norm{\Delta Y_t} \leq c$ a.s.\ for all $t \in \N$ for some
  constant $c$, and let
  $V_t \defineas \sum_{i=1}^t \E_{i-1} \norm{\Delta Y_i}^2$. Then for any
  $x, m > 0$, we have
  \begin{multline}
    \P\eparen{
      \exists t \in \N :
      \Psi(Y_t) \geq x + D_\star^2 \slope_P\pfrac{x}{D_\star^2 m} \cdot (V_t - m)
    } \\
      \leq 2 \expebrace{-D_\star^2 m \psi^\star_P\pfrac{x}{D_\star^2 m}}
      \leq 2 \expebrace{-\frac{x^2}{2(D_\star^2 m + c x / 3)}}.
        \label{eq:pinelis_bennett}
  \end{multline}
  This strengthens Theorem 3.4 from \citet{pinelis_optimum_1994} [B].
\end{enumerate}
\end{corollary}

We prove this result in \cref{sec:proof_pinelis}. As before, the Hoeffding-style
bound in part (a) and the Bennett-style bound in part (b) are not directly
comparable: $V_t$ may be smaller in part (b), but the exponent is also smaller.

We briefly highlight some of the strengths and limitations of this
approach. Since the Euclidean $l_2$-norm is induced by the standard inner
product in $\R^d$, \cref{th:pinelis} gives a dimension-free uniform bound on the
$l_2$-norm deviations of a vector-valued martingale in $\R^d$ which exactly
matches the form for scalars. Compare this to bounds based on the operator norm
of a Hermitian dilation: the bound of \citet{tropp_user-friendly_2012}
includes dimension dependence [B,E] while the bound of \citet[Corollary
4.1]{minsker_extensions_2017} incurs an extra constant factor of 14 [B,E]. Our
bounds extend to martingales taking values in sequence space
$\ebrace{(a_i)_{i \in \N} : \sum_i \abs{a_i}^2 < \infty}$ or function space
$L^2[0,1]$, and we may instead use the $l_p$ norm, $p \geq 2$, in which case
$D = \sqrt{p - 1}$. These cases follow from \citet[Proposition
2.1]{pinelis_optimum_1994}.

Similarly, \cref{th:pinelis} gives dimension-free uniform bounds for the
Frobenius-norm deviations of a matrix-valued martingale. This extends to
martingales taking values in a space of Hilbert-Schmidt operators on a separable
Hilbert space, with deviations bounded in the Hilbert-Schmidt norm; compare
\citet[\S 3.2]{minsker_extensions_2017}, which gives operator-norm bounds. The
method of \cref{th:pinelis} does not extend directly to operator-norm bounds
because the operator norm is not $(2,D)$-smooth for any $D$: for a simple
illustration in $\Hcal^2$, consider $x = a I_2$ and $v = \diag\brace{b, -b}$, so
that $\opnorm{x + v}^2 + \opnorm{x - v}^2 - 2 \opnorm{x}^2 = 2b^2 + 4ab$ and
condition \eqref{eq:d_smooth_3} cannot be satisfied. However, \cref{th:pinelis}
does apply to the matrix Schatten $p$-norm for $p < \infty$, using
$D = \sqrt{p-1}$, and this holds for rectangular matrices as well
\citep{ball_sharp_1994}.

\subsection{Continuous-time processes}

While
\cref{th:chernoff_cases,th:gamma_line_crossing,th:variance_normalized,th:ratio_subg}
already generalize results known in discrete time to new results for
continuous-time martingales [C], here we summarize a few more useful bounds
explicitly for continuous-time processes which follow from
\cref{th:uniform_chernoff} and the conditions of \cref{th:continuous_facts},
making use of the novel strategies devised by \citet{shorack_empirical_1986} and
\citet{van_de_geer_exponential_1995}. These results use the conditional
quadratic variation $\eangle{S}_t$. We remind the reader that
$[S]_t = \eangle{S}_t = t$ for Brownian motion, and the first equality holds
more generally for martingales with continuous paths, while for a Poisson
process with rate one, $\eangle{S}_t = t$ but $[S]_t = S_t$.

\begin{corollary}\label{th:continuous_bounds}
Let $(S_t)_{t \in (0,\infty)}$ be a real-valued process.
\begin{enumerate}[label=(\alph*)]
\item If $(S_t)$ is a locally square-integrable martingale with a.s.\ continuous
  paths, then for any $a, b > 0$, we have
  \begin{align}
    \P\eparen{\exists t \in (0, \infty): S_t \geq a + b \eangle{S}_t}
    \leq e^{-2ab}.
  \end{align}
  If $\eangle{S}_t \uparrow \infty$ as $t \uparrow \infty$, then the probability
  upper bound holds with equality. This recovers as a special case the standard
  line-crossing probability for Brownian motion (e.g.,
  \citealp{durrett_probability:_2017}, Exercise 7.5.2).
\item If $(S_t)$ is a local martingale with $\Delta S_t \leq c$ for all $t$,
  then for any $x, m > 0$, we have
  \begin{multline}
    \P\eparen{
      \exists t \in (0, \infty) : S_t \geq
      x + \slope_P\pfrac{x}{m} \cdot (\eangle{S}_t - m)
    } \\ \leq \expebrace{-m \psi^\star_P\pfrac{x}{m}}
      \leq \expebrace{-\frac{x^2}{2(m + cx/3)}}.
      \label{eq:continuous_bennett_bound}
  \end{multline}
  This strengthens Appendix B, Inequality 1 of \citet{shorack_empirical_1986}
  [B].
\item If $(S_t)$ is any locally square-integrable martingale satisfying the
  Bernstein condition of \cref{th:continuous_facts}(c) for some predictable
  process $(V_t)$, then for any $x, m > 0$, we have
  \begin{multline}
    \P\eparen{
      \exists t \in (0, \infty) : S_t \geq
      x + \slope_G\pfrac{x}{m} \cdot (V_t - m)
    } \\ \leq \expebrace{-m \psi^\star_G\pfrac{x}{m}}
      \leq \expebrace{-\frac{x^2}{2(m + cx)}}.
  \end{multline}
  This strengthens Lemma 2.2 of \citet{van_de_geer_exponential_1995} [B,E].
\end{enumerate}
\end{corollary}

Clearly, \cref{th:continuous_bounds}(b) applies to centered Poisson processes
with ${c = 1}$. Of course, one can also apply \cref{th:continuous_facts}(a) for
general L\'evy processes, obtaining the same bound
\eqref{eq:continuous_bennett_bound}. The point of \cref{th:continuous_bounds}(b)
is that any local martingale with bounded jumps obeys this inequality, and so
concentrates like a centered Poisson process in this sense. \Citet[\S
4]{barlow_inequalities_1986} describe further exponential supermartingales
obtained for continuous-time processes using the quadratic variation, and derive
``Freedman-style'' self-normalized bounds; incorporating these cases into our
framework would be interesting future work.

\subsection{Exponential families and the sequential probability ratio test}
\label{sec:expo_family}

It is well known that the likelihood ratio $f_{1,t}(X_1^t) / f_{0,t}(X_1^t)$ is
a martingale under the null hypothesis that $X_1^t \sim f_{0,t}$. Then Ville's
inequality gives a sequential test with valid type I error, equivalent to an
open-ended sequential probability ratio test (SPRT,
\citealp{wald_sequential_1945}), in which we stop when the likelihood ratio
exceeds an upper threshold, but not when it drops below any lower threshold. In
the one-parameter exponential family case, we obtain a simple analytical result
which is equivalent to \cref{th:uniform_chernoff}, as we detail below.

Suppose $(X_t)_{t \in \N}$ are i.i.d.\ from a one-parameter exponential family
with natural parameter $\theta$ and log-partition function $A$, so that $X_t$
has density $f_\theta(x) = h(x) \expebrace{\theta T(x) - A(\theta)}$. Let
$S_t = \sum_{i=1}^t T(X_i)$. An open-ended SPRT testing $H_0: \theta = \theta_0$
against $H_1: \theta = \theta_0 + \lambda$ stops to reject $H_0$ as soon as the
likelihood ratio
$L_t = \expebrace{\lambda S_t - [A(\theta_0 + \lambda) - A(\theta_0)] t}$
exceeds the threshold $\alpha^{-1} > 1$.
\begin{corollary}\label{th:sprt_summary}
This one-sided SPRT has type I error rate no greater than $\alpha$:
$\P_{\theta_0}(\exists t \in \N: L_t \geq \alpha^{-1}) \leq \alpha$.
\end{corollary}
This standard fact follows easily from \cref{th:uniform_chernoff} because
$L_t \geq A$ if and only if
$S_t \geq (\log A) / \lambda + \psi(\lambda) t / \lambda$, where
$\psi(\lambda) = A(\theta_0 + \lambda) - A(\theta_0)$, the CGF of $T(X_i)$ at
$\theta = \theta_0$. Hence the rejection boundary for the SPRT is equivalent to
the linear boundary of \cref{th:uniform_chernoff}. In light of this, we may
interpret the above sub-Gaussian, sub-Poisson, sub-exponential and sub-Bernoulli
bounds as open-ended SPRTs for i.i.d.\ observations from these exponential
families. The fact that such tests are also valid for testing various
nonparametric classes of distributions, as outlined in
\cref{sec:sufficient_conditions}, illustrates how our framework provides
nonparametric generalizations of the SPRT. For example, if one wants to test the
mean of a bounded distribution, our framework suggests that one apply an SPRT
for Bernoulli or Poisson observations, for example. It has long been known that
the normal SPRT bound can be applied to sequential problems involving any
i.i.d.\ sequence of sub-Gaussian observations \citep{darling_iterated_1967,
  robbins_statistical_1970}. Our work expands the breadth of nonparametric
sequential problems amenable to such methods and deepens the connection between
exponential concentration inequalities and sequential testing procedures.

\section{Discussion and extensions}\label{sec:discussion}

This section is divided into three parts. We first discuss the sharpness of the
derived bounds. Then, building further on the geometric intuition of the paper,
we point out an interesting geometric relationship between fixed-sample
exponential bounds and our uniform bounds. We end by discussing directions for
future work.

\subsection{When is \cref{th:uniform_chernoff} sharp?}

In the discrete-time, sub-Gaussian case $\psi = \psi_N$ and $l_0 = 1$,
\cref{th:uniform_chernoff}(a) is sharp: for any $a,b > 0$,
\begin{align}
  \sup_{(S_t,V_t) \in \mathbb{S}^1_{\psi_N}}
    \P\condparen{\exists t \in \N: S_t \geq a + b V_t}{\Fcal_0}
    = e^{-2ab}.
\end{align}
In fact, this can be achieved by rescaling any sum of i.i.d.\ observations with
finite variance, which we prove in \cref{sec:proof_iid_invariance} as a
corollary of Theorem 2 of \citet{robbins_boundary_1970}:
\begin{corollary}\label{th:iid_invariance}
  Suppose $(X_t)_{t \in \N}$ are i.i.d.\ mean zero with variance
  $\sigma^2 < \infty$. Let $S_t = \sum_{i=1}^t X_i$. Let
  $S^{(m)}_t \defineas S_t / \sqrt{m}$ and $V^{(m)}_t \defineas t \sigma^2 /
  m$. Then for any $a, b > 0$,
\begin{align}
  \lim_{m \to \infty}
    \P\eparen{\exists t \in \N: S^{(m)}_t \geq a + b V^{(m)}_t} = e^{-2ab}.
  \label{eq:invariance}
\end{align}
\end{corollary}

The following more general sandwich relation, which we prove in
\cref{sec:proof_sandwich}, quantifies the looseness in
\cref{th:uniform_chernoff}(a) and gives a sufficient condition for the
probability bound to be exact. This condition involves the ``overshoot'' of the
process $S_t$ over the line $a + b V_t$, a quantity which has been studied
extensively in the context of sequential testing
\citep{siegmund_sequential_1985}. The upper bound in equation
\eqref{eq:sandwich} below is a restatement of \cref{th:uniform_chernoff}(a);
only the lower bound is new.

\begin{proposition}\label{th:sandwich}
  Consider real-valued processes $(S_t)$, $(V_t)$ and a CGF-like function
  $\psi$. Fix $a > 0, b \in (0, \bar{b})$ and suppose
\begin{enumerate}
\item $M_t \defineas \expebrace{\decay(b) S_t - \psi(\decay(b)) V_t}$ is a
  martingale with $M_0 \equiv 1$ (rather than just upper bounded by a
  supermartingale, as \cref{th:canonical_assumption} requires),
\item $S_t - b V_t \to -\infty$ as $t \uparrow \infty$ a.s., and
\item For some $\epsilon \geq 0$, $S_\tau \leq a + b V_\tau + \epsilon$ a.s.\ on
  $\brace{\tau < \infty}$, where
  $\tau \defineas \inf\brace{t \in \Tcal : S_t \geq a + b V_t}$.
\end{enumerate}
Then we have
\begin{align}
   e^{-\epsilon \decay(b)}
    \leq \frac{\P\eparen{\exists t \in \Tcal: S_t \geq a + b V_t}}
              {\expebrace{-a\decay(b)}}
    \leq 1. \label{eq:sandwich}
\end{align}
\end{proposition}

In particular, if the conditions of \cref{th:sandwich} hold with $\epsilon = 0$,
then the probability bounds in \cref{th:uniform_chernoff} parts (a), (b) and (c)
hold with equality. In the continuous-time case with $(S_t)$ a continuous
martingale, these conditions often hold with $\psi = \psi_N$ and $V_t =
[S]_t$. We give details for the following result in
\cref{sec:proof_continuous_martingale}; see \citet[Theorem
III.44]{protter_stochastic_2005} for more on Kazamaki's criterion:
\begin{corollary}\label{th:continuous_martingale}
  Suppose $(S_t)_{t \in (0,\infty)}$ is a continuous martingale with $S_0 = 0$
  and $[S]_t \uparrow \infty$ a.s.\ satisfying Kazamaki's criterion:
  $\sup_T \E e^{S_T / 2} < \infty$, where the supremum is taken over all bounded
  stopping times $T$. Then
  $\P(\exists t \in (0, \infty): S_t \geq a + b V_t) = e^{-2ab}$.
\end{corollary}

In the discrete-time case with i.i.d.\ observations bounded above by $\epsilon$
a.s.\ and having CGF $\psi$, the conditions of \cref{th:sandwich} hold, setting
$V_t = t$. Hence the probability bound in \cref{th:uniform_chernoff}(a) can be
made arbitrarily close to exact by taking $b$ sufficiently small relative to
$\epsilon$, and similarly for parts (b) and (c). So \cref{th:uniform_chernoff}
is sharp in the sense that for any such process, the probability bound is
arbitrarily close to exact for some choice of $(a,b)$ or $(x,m)$. To see the
connection with \cref{th:iid_invariance}, rewrite \eqref{eq:invariance} to keep
the processes $S_t$ and $V_t = t \sigma^2$ fixed and take limits with respect to
$a,b$:
\begin{align}
  \lim_{m \to \infty}
    \P\eparen{\exists t \in \N:
      S_t \geq a\sqrt{m} + \frac{b}{\sqrt{m}} \cdot t \sigma^2} = e^{-2ab}.
\end{align}

\subsection{Geometric relationship between \cref{th:uniform_chernoff} and
  Cram\'er-Chernoff}
\label{sec:geometric}

\begin{figure}[h!]
\centering
\includegraphics{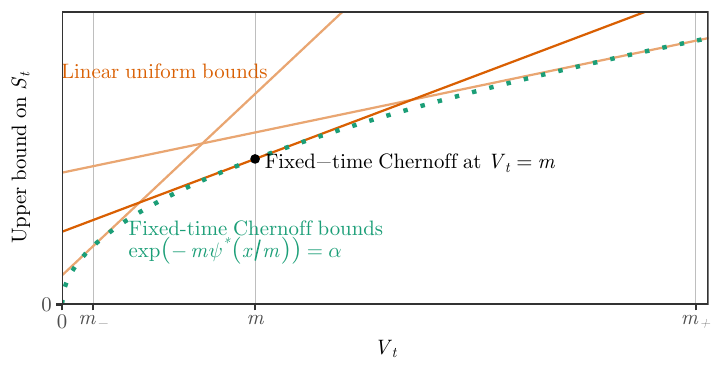}
\caption{Geometric illustration of \cref{th:uniform_chernoff}(b) and its
  relation to fixed-time Cram\'er-Chernoff bounds. \Cref{th:uniform_chernoff}(b)
  chooses the linear boundary which is optimal for $V_t = m$, but other linear
  boundaries with the same crossing probability are illustrated, each of which
  achieves the optimal fixed-time bound at some other time $V_t = m_{\pm}$. Each
  uniform Chernoff bound is tangent to the curve of fixed-time bounds, and
  indeed the curve of fixed-time bounds may be defined as the pointwise infimum
  of such linear uniform bounds. \label{fig:tangent}}
\end{figure}

Whenever a process $(S_t)$ is sub-$\psi$ with $V_t = t$, a fixed-time
Cram\'er-Chernoff upper bound of the form \eqref{eq:basic_chernoff} holds: for
any fixed $t \in \N$, we have
$\P(S_t \geq x) \leq \expebrace{-t \psi^\star(x/t)}$. Let $f_\alpha(t)$ denote
the curve of such fixed-time bounds constructed for a fixed crossing probability
$\alpha$ at each time $t$:
\begin{align}
f_\alpha(t) \defineas t {\psi^\star}^{-1}\pfrac{\log \alpha^{-1}}{t},
\end{align}
where
${\psi^\star}^{-1}(\lambda) = \inf\brace{u \geq 0 : \psi^\star(u) >
  \lambda}$. For example, in the sub-Gaussian case
$\psi(\lambda) = \psi_N(\lambda) = \lambda^2/2$, we have the standard formula
$f_\alpha(t) = \sqrt{2 t \log \alpha^{-1}}$.
\begin{proposition}\label{th:tangent}
  Any line $a + bt$ which is tangent to $f_\alpha(t)$ satisfies
  $\P(\exists t \in \Tcal: S_t \geq a + b t) \leq \alpha$.
\end{proposition}

In words, the above proposition states that the set of linear boundaries from
\cref{th:uniform_chernoff} is exactly the set of tangent lines to $f_\alpha$, or
conversely, $f_\alpha$ is defined as the pointwise infimum of this set of linear
boundaries, as illustrated in \cref{fig:tangent}. We give the proof in
\cref{sec:proof_tangent}. This observation provides some intuition for the
appearance of the Legendre-Fenchel transform in the standard Cram\'er-Chernoff
formula \eqref{eq:basic_chernoff}.

\subsection{Future work}

\textbf{Characterizing families of sub-$\psi$ processes.} Our
\cref{th:uniform_chernoff} bounds the maximal line-crossing probability over
each family $\subpsiclass$, and \cref{sec:sufficient_conditions} collects
sufficient conditions for membership is many such families. It would be
interesting to better delineate such families, for example by characterizing
necessary conditions for inclusion. When $\psi$ is CGF-like and $(V_t)$ is
predictable, it is necessary for the increments $\Delta S_t$ to have finite
conditional CGFs a.s. When $S_t$ is a cumulative sum of i.i.d., real-valued
random variables and $V_t \propto t$, the existence of the CGF is sufficient as
well (\cref{th:discrete_facts}(a)). When the increments are not i.i.d., however,
existence of conditional CGFs is no longer sufficient. When $(V_t)$ is not
predictable, as with self-normalized bounds, it is no longer necessary for
increments to have finite CGFs (e.g., \cref{ex:cauchy}).

\textbf{Determining optimal $l_0$ values.} Smaller values of $l_0$ are
preferable since they lead to tighter bounds in \cref{th:uniform_chernoff}. Most
of the results in this paper take either $l_0 = 1$ for scalar observations or
$l_0 = d$ for $d \times d$ matrix observations. Taking $\lambda \downarrow 0$ in
\cref{th:canonical_assumption} shows we cannot have $l_0 < 1$. Furthermore,
asymptotic results about maxima of independent Gaussians show that $l_0 = d$ is
an asymptotic lower bound as $d \uparrow \infty$ for operator-norm inequalities
over any class that includes matrices with independent Gaussians on the diagonal
(\citealp{galambos_asymptotic_1978}; \citealp{boucheron_concentration_2013},
Exercise 2.17). It would be useful to derive more results about optimal values
of $l_0$ in various settings.

\textbf{Generalizing assumptions.} \Cref{th:canonical_assumption} can be further
generalized, allowing it to subsume more known inequalities and yield sharper
results for certain cases. However, the corresponding general theorem and
specific results are less user-friendly. We have chosen our
\cref{th:canonical_assumption} and \cref{th:uniform_chernoff} to balance
generality and tractability, but in \cref{sec:more_general} we present one
possible generalization of our assumption and a corresponding general theorem
and specific bound.

\textbf{Polynomial line-crossing inequalities.} We have focused on exponential
inequalities, but polynomial concentration also plays an important role in the
literature. A theory of polynomial line-crossing analogous to that presented
here may begin with the Dubins-Savage inequality (see \cref{sec:dubins_savage})
and its $l_p$ extension in \citet{khan_$l_p$-version_2009}.

\textbf{Banach spaces.} The Banach space bounds in \cref{sec:banach} give
dimension-free $l_p$ bounds for $2 \leq p < \infty$, but do not give $l_\infty$
bounds. In particular, this does not yield operator-norm bounds for
infinite-dimensional Hilbert-Schmidt operators, as in
\citet{minsker_extensions_2017}. Extending Minsker's ``effective rank''
approach to the uniform bounds of this paper would be an interesting future
extension.

\section{Proofs}\label{sec:proofs}

\subsection{Proof of \cref{th:ville}}
\label{sec:proof_ville}

Define the stopping time $\tau \defineas \inf\brace{t \in \Tcal: L_t \geq a}$,
where $\inf \emptyset = \infty$. For any fixed $m \in \Tcal$, Markov's
inequality implies
\begin{align}
  \P\condparen{\tau \leq m}{\Fcal_0}
  = \P\condparen{L_{\tau \bmin m} \geq a}{\Fcal_0}
  \leq \frac{\E\condparen{L_{\tau \bmin m}}{\Fcal_0}}{a}
  \leq \frac{L_0}{a},
\end{align}
where we have used Doob's optional stopping theorem for bounded stopping times
in the final step (e.g., \citealp{durrett_probability:_2017}, Exercise 4.4.2; or
\citealp{protter_stochastic_2005}, Theorem I.17). Taking $m \to \infty$ and
using the bounded convergence theorem yields
$\P\condparen{\tau < \infty}{\Fcal_0} \leq L_0 / a$, which is the desired
conclusion.

\subsection{Proof of \cref{th:universal}}
\label{sec:proof_universal}

Applying Taylor's theorem to $\psi$ at the origin, we have
$\psi(\lambda) = \ebracket{\frac{\psi''(0_+)}{2} + h(\lambda)} \lambda^2$ where
$h(\lambda) \to 0$ as $\lambda \downarrow 0$. Choose $\lambda_0 > 0$ small
enough so that $\psi(\lambda) \leq \psi''(0_+) \lambda^2$ for all
$0 \leq \lambda < \lambda_0$. Then, setting $a \defineas 2 \psi''(0_+)$,
$c \defineas 1/\lambda_0$ and using that fact that $\psi_{G,c} \geq \psi_N$ for
$c \geq 0$, we have
$\psi(\lambda) \leq a \psi_N(\lambda) \leq a \psi_{G,c}(\lambda)$ for all
$0 \leq \lambda < 1/c$. The same argument holds with $\psi_E$ in place of
$\psi_G$. \qed

\subsection{Proof of \cref{th:psi_relations}}\label{sec:proof_psi_relations}

In each case, we show an inequality between two $\psi$ functions. The conclusion
then follows from the fact that is $\psi_1 \leq \psi_2$, then
$\expebrace{\lambda S_t - \psi_2(\lambda) V_t} \leq \expebrace{\lambda S_t -
  \psi_1(\lambda) V_t}$, showing that the key condition of
\cref{th:canonical_assumption} continues to hold with $\psi_2$ in place of
$\psi_1$.

\textbf{Part (1)}: the proof of Theorem 1 in \citet{hoeffding_probability_1963}
shows that, for all $\mu \in (0, 1)$ and all $t \in [0, 1-\mu)$,
\begin{multline}
  (\mu + t) \log\pfrac{\mu + t}{\mu} +
    (1 - \mu - t) \log\pfrac{1 - \mu - t}{1 - \mu}
  \\ \geq t^2 \begin{cases}
    \frac{1}{1 - 2\mu} \log\pfrac{1 - \mu}{\mu}, & 0 < \mu < \frac{1}{2}, \\
    \frac{1}{2 \mu (1 - \mu)}, & \frac{1}{2} \leq \mu < 1,
  \end{cases}
\end{multline}
with equality at $t = 1 - 2\mu$. Substituting $\mu = g / (g + h)$ and
$t = u / (g + h)$ for $u \in [0, h)$, some algebra shows that the left-hand side
is equal to $gh\psi_B^\star(u/gh)$ and the right-hand side is equal to
$\psi_N^\star(u) / \varphi(g, h)$, so that, for all $g, h > 0$ and
$u \in [0, h)$, $\psi_B^\star(u/gh) \geq \psi_N^\star(u) / [gh\varphi(g, h)]$,
with equality at $u = h - g$. The order-reversing and scaling properties of the
Legendre-Fenchel transform now imply
$\psi_B^{\star \star}(\lambda) \leq \psi_N^{\star \star}(\varphi(g,h) \lambda) /
[gh \varphi(g,h)]$ for all $\lambda \geq 0$. Finally, since $\psi_B$ and
$\psi_N$ are convex and continuous, each is equal to its biconjugate
$\psi^{\star\star}$ by the Fenchel-Moreau theorem, so that
$\psi_B(\lambda) \leq \frac{\varphi(g, h)}{gh} \psi_N(\lambda)$.

\textbf{Part (2)}: This follows directly from equation (4.15) in
\citet{hoeffding_probability_1963} which, in our notation, says that
$\psi_B(\lambda) \leq \frac{(g+h)^2}{4gh} \psi_N(\lambda)$ for all
$\lambda \in \R$.

\textbf{Part (3)}: In our notation, Lemma 2.32 of
\citet{bercu_concentration_2015} shows that
$(g \psi_{B,g,1})^\star(u) \geq (g \psi_{P,1-g})^\star(u)$ for all $u \in [0,1]$
and $g > 0$. The order-reversing and scaling properties of the Legendre-Fenchel
transform imply
$\psi_{B,g,1}^{\star \star}(\lambda) \leq \psi_{P,1-g}^{\star \star}(\lambda)$
for all $\lambda \geq 0$. Since $\psi_{B,g,1}$ and $\psi_{P,1-g}$ are convex and
continuous, each is equal to its biconjugate $\psi^{\star \star}$ by the
Fenchel-Moreau theorem, so that
$\psi_{B,g,1}(\lambda) \leq \psi_{P,1-g}(\lambda)$. The result now follows from
algebraic identities involving $\psi_B$ and $\psi_P$: for any $g,h > 0$,
\begin{align}
  \psi_{B,g,h}(\lambda) = \frac{1}{h^2} \psi_{B,g/h,1}(h\lambda)
  \leq \frac{1}{h^2} \psi_{P,(h-g)/h}(h\lambda)
  = \psi_{P,h-g}(\lambda).
\end{align}

\textbf{Part (4)} is immediate from the definition $\psi_P = \psi_N$ when
$c = 0$.

\textbf{Part (5)}: since $\psi_{P,c_P}''(\lambda) = e^{c_P \lambda}$ and
$\psi_{G,c_G}''(\lambda) = (1 - c_G \lambda)^{-3}$,
\begin{align}
  \frac{\psi_{P,c}''(\lambda)}{\psi_{G,c/3}''(\lambda)}
  = (1 - c \lambda/3)^3 e^{c \lambda}
  = f(1 - c \lambda/3),
  \quad \text{where } f(y) = y^3 e^{3 (1 - y)}.
\end{align}
We have $f(1) = 1$ and $f'(y) = 3y^2 e^{3(1 - y)} (1 - y)$, so that
$f'(y) \leq 0$ for $y \geq 1$ and $f'(y) \geq 0$ for $y \leq 1$. Hence
$f(y) \leq f(1) = 1$ for all $y$, i.e.,
$\psi_{P,c}''(\lambda) \leq \psi_{G,c/3}''(\lambda)$ for all $\lambda$. Since
$\psi_{P,c}(0) = \psi_{G,c/3}(0) = 0$ and
$\psi_{P,c}'(0) = \psi_{G,c/3}'(0) = 0$, we conclude
$\psi_{P,c}(\lambda) \leq \psi_{G,c/3}(\lambda)$ for all $\lambda$.

\textbf{Parts (6,7,8)}: some algebra shows that
\begin{align}
  \psi_{G,c_G}'(\lambda) - \psi_{E,c_E}'(\lambda)
  = \frac{\lambda^2 [3c_G - 2c_E + c_G(c_E - 2c_G) \lambda]}
         {2(1-c_G \lambda)^2(1-c_E \lambda)}.
  \label{eq:diff_of_derivs}
\end{align}
Since $\psi_{G,c_G}(0) = \psi_{E,c_E}(0) = 0$, we have
$\psi_{G,c_G}(\lambda) \geq (\leq)\, \psi_{E,c_E}(\lambda)$ for all $\lambda$ if
$\psi_{G,c_G}' \geq (\leq)\, \psi_{E,c_E}'$ for all $\lambda$, and
\eqref{eq:diff_of_derivs} shows the latter is true if and only if
$f(\lambda) \defineas 3c_G - 2c_E + c_G(c_E - 2c_G) \lambda \geq (\leq)\, 0$ for
all $\lambda$. Note we need only check the domain
$0 \leq \lambda < c_E^{-1} \bmin (2c_G)^{-1}$ on which both functions are
defined.
\begin{itemize}
\item For part (6), if $c_E = 3c_G/2$, then
  $f(\lambda) = -c_G^2 \lambda / 2 \leq 0$, so that
  $\psi_{G,c} \leq \psi_{E,3c/2}$ for $c \in \R$.
\item For part (7), if $c_G = c_E \geq 0$ then we have
  $f(\lambda) = c(1 - c\lambda) \geq 0$ for $0 \leq \lambda < c^{-1}$, so that
  $\psi_{E,c} \leq \psi_{G,c}$ for $c \geq 0$.
\item For part (8), if $c_G = c_E/2 < 0$, then $f(\lambda) = -c_E/2 > 0$, so
  that $\psi_{E,c} \leq \psi_{G,c/2}$ for $c < 0$.
\end{itemize}

\textbf{Part (9)}: from $\psi_{P,2c}'(\lambda) = \frac{e^{2c\lambda} - 1}{2c}$
and $\psi_{G,c}'(\lambda) = \frac{\lambda(2-c\lambda)}{2(1-c\lambda)^2}$, we
have
\begin{align}
  \psi_{P,2c}'(\lambda) - \psi_{G,c}'(\lambda)
  = \frac{1 - f(1 + \abs{c}\lambda)}{2\abs{c}(1 - c\lambda)^2},
  \quad \text{where } f(y) = y^2 e^{2(1-y)}.
\end{align}
We have $f(1) = 1$ and $f'(y) = 2y e^{2(1-y)} (1 - y) \leq 0$ for all
$y \geq 1$, so that $f(y) \leq 1$ for all $y \geq 1$. Hence
$\psi_{P,2c}'(\lambda) \geq \psi_{G,c}'(\lambda)$ for all $\lambda \geq
0$. Together with $\psi_{P,2c}(0) = \psi_{G,c}(0) = 0$, we conclude
$\psi_{P,2c}(\lambda) \geq \psi_{G,c}(\lambda)$ for all $\lambda \geq 0$.

\textbf{Part (10)} follows from the fact that $\psi_{P,c} \uparrow \psi_N$ as
$c \uparrow 0$.

\textbf{Part (11)}: for any $g,h > 0$, we have
\begin{align}
  \psi_{B,g,h}'(\lambda)
    = \frac{e^{h\lambda} - e^{-g\lambda}}{ge^{h\lambda} + he^{-g\lambda}},
  \label{eq:psi_B_deriv}
\end{align}
so
$\lim_{h \downarrow 0} \psi_{B,g,h}'(\lambda) = (1 - e^{-g\lambda})/g =
\psi_{P,-g}'(\lambda)$. Since $\psi_{B,g,h}(0) = \psi_{P,c} = 0$ for all $g,h>0$
and all $c \in \R$, we see that
$\lim_{h \downarrow 0} \psi_{B,g,h}(\lambda) = \psi_{P,-g}(\lambda)$ for all
$\lambda \geq 0$. Furthermore, differentiating \eqref{eq:psi_B_deriv} with
respect to $h$ reveals
\begin{align}
  \frac{d}{dh} \psi_{B,g,h}'(\lambda)
  = \frac{e^{(h-g)\lambda} (g+h)^2 \psi_{P,-(g+h)}(\lambda)}
         {(ge^{h\lambda} + he^{-g\lambda})^2}
  \geq 0,
\end{align}
which implies $\psi_{B,g,h}(\lambda)$ is nondecreasing with $h$ for all
$\lambda \geq 0$. We conclude
$\psi_{B,g,h}(\lambda) \downarrow \psi_{P,-g}(\lambda)$ as $h \downarrow 0$,
hence $\psi_{P,c} \leq \psi_{B,-c,h}$ for all $h > 0$ whenever $c < 0$.

To see that no other implications are possible, observe that, as
$\lambda \to \infty$, $\psi_B(\lambda) = \Ocal(\lambda)$,
$\psi_N(\lambda) = \Ocal(\lambda^2)$, and when $c > 0$,
$\psi_P(\lambda) = \Ocal(e^{c\lambda})$, while $\psi_G(\lambda)$ and
$\psi_E(\lambda)$ diverge at a finite value of $\lambda$. So we cannot use
$a \psi_B$ to upper bound any of the other $\psi$ functions for any constant
$a$. Likewise, we cannot use $a \psi_N$ to upper bound $\psi_P$,
$\psi_G$ or $\psi_E$, and we cannot use $a \psi_P$ to upper bound $\psi_G$
or $\psi_E$.

Now if $S_t$ is a sum of i.i.d. $\Normal(0,1)$ random variables, then $(S_t)$ is
sub-Gaussian with variance process $V_t = t$, and the exponential process
$\expebrace{\lambda S_t - \lambda^2 t / 2}$ is a martingale. Under any scaling
of $V_t$ by a constant $a > 0$, $(S_t)$ cannot be sub-Bernoulli, because
$\E \expebrace{\lambda \Delta S_t - a \psi_B(\lambda)} = \expebrace{\lambda^2 /
  2 - a \psi_B(\lambda)} > 1$ for sufficiently large $\lambda$, so the
exponential process $\expebrace{\lambda S_t - \psi_B(\lambda) t}$ will be
expectation-increasing. Analogous arguments shows that other reverse
implications are not possible.

To see that the above constants are the best possible when we allow only scaling
of $V_t$ by a constant, consider the third-order expansions of each $\psi$
function about $\lambda = 0$:
\begin{align}
  \psi_B(\lambda) &=
    \ebracket{\frac{\lambda^2}{2} + \frac{(h-g) \lambda^3}{6}} + o(\lambda^3)
    \\
  \psi_N(\lambda) &= \frac{\lambda^2}{2} \\
  \psi_P(\lambda) &= \frac{\lambda^2}{2} + \frac{c\lambda^3}{6} + o(\lambda^3)\\
  \psi_E(\lambda) &= \frac{\lambda^2}{2} + \frac{c\lambda^3}{3} + o(\lambda^3)\\
  \psi_G(\lambda) &= \frac{\lambda^2}{2} + \frac{c\lambda^3}{2} + o(\lambda^3).
\end{align}
It is clear from these expansions that parts (3), (4), (5), (6), and (11) have
the best possible constants. Part (7) is unimprovable because $\psi_E$ diverges
at $\lambda = 1 / c$, and using any scale parameter in $\psi_G$ smaller than $c$
would make $\psi_G$ finite at $\lambda = 1 / c$. For part (8), recall that when
$c < 0$, $\bar{b} = \abs{c}^{-1}$ for $\psi_E$, while $\bar{b} = \abs{2c}^{-1}$
for $\psi_G$. Hence, if $c' < c/2 < 0$, then
$\lim_{\lambda \to \infty} \psi_{G,c'}'(\lambda) = \abs{2c'}^{-1} < \abs{c}^{-1}
= \lim_{\lambda \to \infty} \psi_{E,c}'(\lambda)$, so that
$\psi_{G,c'}(\lambda)$ must be smaller than $\psi_{E,c}(\lambda)$ for
sufficiently large $\lambda$. Part (9) is unimprovable by an analogous argument.

For part (1), when $g \geq h$, we know that the constant of one in front of
$\psi_N(\lambda)$ is the best possible from the expansions above. When $g < h$,
some algebra shows that the inequality
$\psi_{B,g,h}(\lambda) \leq \frac{\varphi(g,h)}{gh} \psi_N(\lambda)$ holds with
equality at $\lambda = (h-g) / \varphi(g,h)$, so the constant cannot be
improved. For part (2), it is easy to see that
$\varphi(g,h) = \pfrac{g+h}{2}^2 = g^2$ when $g = h$, so the constant
$\frac{(g+h)}{4gh}$ is the best possible of the form $k/gh$ where $k$ is a
function of $g+h$ alone. \qed

A brief remark on the rationale behind part (2). In the ``Bernoulli I''
(\cref{th:discrete_facts}(b)) and ``Bernoulli II''
(\cref{th:novel_discrete_conds}(a)) conditions, $V_t = ght$, so applying
\cref{th:psi_relations}, part (2) leads to $V_t = \pfrac{g+h}{2}^2 t$, a
function of the total range $g+h$ alone. This is useful in the common case that
observations are known to be bounded in a range $[a,b]$, and an inequality is
desired which depends only on the range $b - a$ and not on the location of the
means within $[a,b]$.

\subsection{An intermediate condition for sub-$\psi$ processes}

In discrete time, the following result capture a useful general condition on a
matrix-valued process $(Y_t)$ that is sufficient to show that the
maximum-eigenvalue process $S_t = \gamma_{\max}(Y_t)$ is sub-$\psi$.

\begin{lemma}\label{th:sub_psi_lemma}
  Let $\psi$ be a real-valued function with domain $[0, \lambda_{\max})$. Let
  $(Y_t)_{t \in \N}$ be an adapted, $\Hcal^d$-valued process. Let
  $(W_t)_{t \in \N}$ be predictable, $\Hcal^d$-valued, and nondecreasing in the
  semidefinite order, with $W_0$ = 0. Let $(U_t)_{t \in \N}$ be defined by
  $U_0 = 0$ and $\Delta U_t = u_t(\Delta Y_t)$ for some
  $u_t : \R \to \R_{\geq 0}$, for each $t$. If, for all $t \in \N$ and
  $\lambda \in [0, \lambda_{\max})$, we have
  \begin{align}
    \log \E_{t-1} \expebrace{\lambda \Delta Y_t - \psi(\lambda) \Delta U_t}
    \preceq \psi(\lambda) \Delta W_t,
    \label{eq:sub_psi}
  \end{align}
  then ${S_t = \gammamax(Y_t)}$ is $d$-sub-$\psi$ with variance process
  $V_t = \gammamax(U_t + W_t)$.
\end{lemma}

For a familiar example, suppose $d = 1$ and $(Y_t)$ has independent
increments. Let $W_t = t$, $U_t \equiv 0$ and $\psi(\lambda) = \lambda^2 /
2$. Then \eqref{eq:sub_psi} reduces to the usual definition of a 1-sub-Gaussian
random variable \citep{boucheron_concentration_2013}. For a self-normalized
example, let $(\Delta Y_t)$ be i.i.d.\ from any distribution symmetric about
zero. Then, again letting $\psi(\lambda) = \lambda^2 / 2$, an argument due to
\citet{de_la_pena_general_1999} shows that \eqref{eq:sub_psi} holds with
$W_t \equiv 0$ and $U_t = \sum_{i=1}^t \Delta Y_i^2$. See
\cref{th:novel_discrete_conds}(d) for a general statement of this condition.

The value $l_0 = d$, the ambient dimension, leads to a pre-factor of $d$ in
all of our operator-norm matrix bounds. In cases when
$\sup_{t \in \Tcal} \rank(U_t + W_t) \leq r < d$ a.s., the pre-factor $d$ in our
bounds may be replaced by $r$ via an argument originally due to
\citet{oliveira_sums_2010}. See \cref{sec:rank_prefactor} for details.

\begin{proof}[Proof of \cref{th:sub_psi_lemma}]
  The key result here is Lieb's concavity theorem:

  \begin{fact}[\citealp{lieb_convex_1973,tropp_user-friendly_2012}]
    \label{th:lieb}
    For any fixed $H \in \Hcal^d$, the function
    $A \mapsto \tr\expebrace{H + \log(A)}$ is concave on the positive-definite
    cone.
  \end{fact}

  Fixing ${\lambda \in [0, \lambda_{\max})}$, Lieb's theorem and Jensen's
  inequality together imply
  \begin{multline}
  \E_{t-1} \tr \expebrace{\lambda Y_t - \psi(\lambda) \cdot (U_t + W_t)}
    \\ \leq \tr \expebrace{
      \lambda Y_{t-1} - \psi(\lambda) \cdot (U_{t-1} + W_t)
      + \log \E_{t-1} e^{\lambda \Delta Y_t - \psi(\lambda) \cdot \Delta U_t}
    }.
  \end{multline}
  Now we apply inequality \eqref{eq:sub_psi} to the expectation and use the
  monotonicity of the trace exponential to obtain
  \begin{align}
  \E_{t-1} \tr \expebrace{\lambda Y_t - \psi(\lambda) \cdot (U_t + W_t)}
    &\leq \tr \expebrace{
            \lambda Y_{t-1} - \psi(\lambda) \cdot (U_{t-1} + W_{t-1})
          }.
  \end{align}
  This shows that the process
  $L_t \defineas \tr\expebrace{\lambda Y_t - \psi(\lambda) \cdot (U_t + W_t)}$
  is a supermartingale, with $L_0 = d$. Next we show that the key condition of
  \cref{th:canonical_assumption} holds,
  $L_t \geq \expebrace{\lambda \gamma_{\max}(Y_t) - \psi(\lambda)
    \gamma_{\max}(U_t + W_t)}$ a.s.\ for all $t$. We repeat a short argument
  from \citet{tropp_user-friendly_2012}. First, by the monotonicity of the trace
  exponential,
  \begin{align}
  &\tr\expebrace{\lambda Y_t - \psi(\lambda) \cdot (U_t + W_t)} \\
    &\qquad \geq \tr\expebrace{\lambda Y_t
      - \psi(\lambda) \gamma_{\max}(U_t + W_t) I_d} \\
    &\qquad \geq \gamma_{\max}\eparen{
      \expebrace{\lambda Y_t - \psi(\lambda) \gamma_{\max}(U_t + W_t) I_d}}
      \defineright B.
  \end{align}
  using the fact that the trace of a positive semidefinite matrix is at least as
  large as its maximum eigenvalue. Then the spectral mapping property gives
  \begin{align}
  B &= \expebrace{\gamma_{\max}\eparen{
      \lambda Y_t - \psi(\lambda) \gamma_{\max}(U_t + W_t) I_d}}.
  \end{align}
  Finally, we use the fact that
  $\gamma_{\max}(A - c I_d) = \gamma_{\max}(A) - c$ for any $A \in \Hcal^d$ and
  $c \in \R$ to see that
  ${B = \expebrace{\lambda \gamma_{\max}(Y_t) - \psi(\lambda) \gamma_{\max}(U_t
      + W_t)}}$, completing the argument.
\end{proof}

\subsection{Proof of \cref{th:novel_discrete_conds}}
\label{sec:proof_novel_discrete_conds}

We rely on the following transfer rule for the semidefinite ordering.
\begin{fact}[\citealp{tropp_user-friendly_2012}, eq.\ 2.2]\label{th:transfer}
If $f(a) \leq g(a)$ for all $a \in S$, then $f(A) \preceq g(A)$ when the
eigenvalues of $A$ lie in $S$.
\end{fact}

We make frequent use of the martingale property $\E_{t-1} \Delta Y_t = 0$, and
prove in most cases that
\begin{align}
  \E_{t-1} \expebrace{\lambda \Delta Y_t - \psi(\lambda) \Delta U_t} \preceq
    \expebrace{\psi(\lambda) \Delta W_t}
  \label{eq:stronger_sub_psi}
\end{align}
for some $(U_t)$ and $(W_t)$, then invoke \cref{th:sub_psi_lemma}. This a
stronger condition than property \eqref{eq:sub_psi}; the latter is implied by
taking logarithms on both sides, recalling the monotonicity of the matrix
logarithm.

\textbf{Part (a)}: we adapt the argument of
\citet[p. 42]{bennett_probability_1962}. Fix $\lambda \geq 0$ and choose real
numbers $u,v,w$ so that $e^{\lambda x} \leq ux^2 + vx + w$ for all $x \leq h$,
with equality at $x = h$ and $x = -g$. Using the assumption
$\Delta Y_t \preceq h I_d$, the transfer rule implies
\begin{align}
  \E_{t-1} e^{\lambda \Delta Y_t}
    \preceq u \E_{t-1} \Delta Y_t^2 + v \E_{t-1} \Delta Y_t + w I_d
    \preceq (u gh + w) I_d,
  \label{eq:bennett_generic_upper_bound}
\end{align}
where the second inequality uses the assumption
$\E_{t-1} \Delta Y_t^2 \preceq gh I_d$ and the martingale property. Now
consider the random matrix
\begin{align}
  Z = \begin{cases}
    -g I_d, & \text{with probability } \frac{h}{h + g}, \\
    h I_d, & \text{with probability } \frac{g}{h + g}.
  \end{cases}
\end{align}
Evidently $\E Z = 0$ and $\E Z^2 = gh I_d$, so $Z$ also satisfies the
aforementioned assumptions. Note that for any function $f : \R \to \R$,
\begin{align}
  \E f(Z) = \ebracket{
    \frac{h}{h + g} \cdot f(-g)
    + \frac{g}{h + g} \cdot f(h)
  } I_d.
\end{align}
By our choice of $u,v,w$, we see that
$\E e^{\lambda Z} = \E (u Z^2 + v Z + w I_d) = (u gh + w) I_d$, so by
direct calculation,
\begin{align}
  (u gh + w) I_d = \E e^{\lambda Z}
  = \ebracket{
    \frac{h}{h + g} \cdot e^{-\lambda g}
    + \frac{g}{h + g} \cdot e^{\lambda h}
  } I_d
  = e^{g h \psi_B(\lambda)} I_d.
  \label{eq:bennett_maximal_mgf}
\end{align}
Combining \eqref{eq:bennett_maximal_mgf} with
\eqref{eq:bennett_generic_upper_bound} shows that \eqref{eq:stronger_sub_psi}
holds with $U_t \equiv 0$ and $W_t = ght I_d$, as desired.

\textbf{Part (b)}: As in Lemma 1 of \citet{hoeffding_probability_1963}, we
use the fact that
$e^{\lambda x} \leq \frac{g + x}{g + h} e^{h \lambda} + \frac{h - x}{g + h}
e^{-g \lambda}$ for all $x \in [-g, h]$, along with the transfer rule, to
conclude that, for each $t$,
\begin{align}
\E_{t-1} e^{\lambda \Delta Y_t}
  \preceq \eparen{\frac{G_t}{G_t+G_t} e^{H_t \lambda}
    + \frac{H_t}{G_t+H_t} e^{-G_t \lambda}} I_d
  = \expebrace{G_t H_t \psi_{B,G_t,H_t}(\lambda)} I_d.
\end{align}
Now the proof of \cref{th:psi_relations} part (1) shows that
$\psi_{B,g,h}(\lambda) \leq \varphi(g,h) \psi_N(\lambda) / gh$, so we have
\begin{align}
\E_{t-1} e^{\lambda \Delta Y_t} \preceq
  \expebrace{\psi_N(\lambda) \varphi(G_t,H_t) I_d},
\end{align}
which shows that \eqref{eq:stronger_sub_psi} holds with $U_t \equiv 0$ and
$\Delta W_t = \varphi(G_t, H_t) I_d$, as desired.

\textbf{Part (c)}: the argument is identical to that for part (a), except for
the use of $\psi_{B,g,h}(\lambda) \leq\frac{(g + h)^2}{4gh} \psi_N(\lambda)$
from the proof of \cref{th:psi_relations} part (2).

\textbf{Part (d)}: From the standard inequality $\cosh x \leq e^{x^2/2}$ we see
that $f(x) \defineas e^{-x^2/2} \cosh x \leq 1$ for all $x$. Introducing an
independent Rademacher random variable $\varepsilon$, we have for any $t$,

\begin{align}
\E_{t-1} \expebrace{\lambda \Delta Y_t - \frac{\lambda^2 \Delta Y_t^2}{2}}
  &= \E_{t-1} \expebrace{\lambda \varepsilon \Delta Y_t
    - \frac{\lambda^2 \Delta Y_t^2}{2}} \\
  &= \E_{t-1} \E\condbracket{
       \expebrace{\lambda \varepsilon \Delta Y_t
         - \frac{\lambda^2 \Delta Y_t^2}{2}}
     }{\Fcal_{t-1}, \Delta Y_t} \\
  &= \E_{t-1} f(\lambda \Delta Y_t) \\
  &\preceq I_d,
\end{align}
applying the transfer rule in the last step. This shows that
\eqref{eq:stronger_sub_psi} holds with $U_t = [Y]_t$ and $W_t \equiv 0$.

\textbf{Part (e)}: Lemma 4.1 of \citet{fan_exponential_2015} shows that
\begin{align}
  \expebrace{\lambda x - [\log(1-\lambda)^{-1} - \lambda] x^2}
    \leq 1 + \lambda x,
    \text{ for all } x \geq -1 \text{ and } 0 \leq \lambda < 1.
\end{align}
Applying the transfer rule and taking expectations, we have for any $t$,
\begin{align}
\E_{t-1} \expebrace{
  \lambda \cdot \frac{\Delta Y_t}{c}
  - [\log(1-\lambda)^{-1} - \lambda] \cdot \frac{\Delta Y_t^2}{c^2}
} \preceq I_d.
\end{align}
Replace $\lambda$ with $c \lambda$ and identify $\psi_E$ to complete the
argument that \eqref{eq:stronger_sub_psi} holds with $U_t = [Y]_t$ and
$W_t \equiv 0$.

\textbf{Part (f)}: Proposition 12 of \citet{delyon_exponential_2009} shows that
$e^{x - x^2 / 6} \leq 1 + x + x^2 / 3$ for all $x \in \R$. This implies, by the
transfer rule,
\begin{align}
\E_{t-1} \expebrace{\lambda \Delta Y_t - \frac{\lambda^2}{6} \Delta Y_t^2}
  &\preceq I_d + \frac{\lambda^2}{3} \E_{t-1} \Delta Y_t^2 \\
  &\preceq \expebrace{\frac{\lambda^2}{3} \E_{t-1} \Delta Y_t^2}.
\end{align}
This shows that \eqref{eq:stronger_sub_psi} holds with $U_t = [Y]_t/3$ and
$W_t = 2 \eangle{Y}_t / 3$.

\textbf{Part (g)}: Proposition 12 of \citet{delyon_exponential_2009}, together
with the fact that $e^{-x} + x - 1 \leq x^2 / 2$ for $x \geq 0$, shows that
$e^{x - x_+^2 / 2} \leq 1 + x + x_-^2 / 2$. Again the transfer rule implies
\begin{align}
\E_{t-1} \expebrace{\lambda \Delta Y_t - \frac{\lambda^2}{2} (\Delta Y_t)_+^2}
  &\preceq I_d + \frac{\lambda^2}{2} \E_{t-1} (\Delta Y_t)_-^2 \\
  &\preceq \expebrace{\frac{\lambda^2}{2} \E_{t-1} (\Delta Y_t)_-^2}.
\end{align}
This shows that \eqref{eq:stronger_sub_psi} holds with $U_t = [Y_+]_t/2$ and
$W_t = \eangle{Y_-}_t / 2$.

\textbf{Part (h)}: we appeal to part (d) to see that $S_t$ is $d$-sub-Gaussian
with variance process
$V_t = \gamma_{\max}(\frac{1}{3} [Y]_t + \frac{2}{3} \eangle{Y}_t)$. Now the
condition $\Delta Y_t^2 \preceq A_t^2$ ensures that
$\frac{1}{3} [Y]_t + \frac{2}{3} \eangle{Y}_t \preceq \sum_{i=1}^t A_i^2$, hence
$V_t \leq \gamma_{\max}(\sum_{i=1}^t A_i^2)$. Substituting this larger variance
process only makes the exponential process in \cref{th:canonical_assumption}
smaller, so the assumption remains satisfied.

\textbf{Part (i)}: the proof of Corollary 2.2 in \citet{fan_exponential_2015} is
based on the inequality $e^{x - x^2/2} \leq 1 + x + x_-^3 / 3$ for all
$x \in \R$. The transfer rule implies
\begin{align}
\E_{t-1} \expebrace{\lambda \Delta Y_t - \frac{\lambda^2}{2} \Delta Y_t^2}
  \preceq I_d + \frac{\lambda^3}{3} \E_{t-1} (\Delta Y_t)_-^3
  \preceq \expebrace{\frac{\lambda^3}{3} \E_{t-1} (\Delta Y_t)_-^3}.
\end{align}
Setting $c = 1/6$ in $\psi_G$, we have for all $x \in [0,6)$ the obvious
inequality $x^2/2 \leq \psi_G(x)$ and we claim $x^3/3 \leq \psi_G(x)$ as well;
indeed,
\begin{align}
\frac{x^3/3}{x^2/2(1-x/6)} = \frac{x(6-x)}{9},
\end{align}
which reaches a maximum value of one at $x = 3$. The transfer rule now implies
\begin{align}
\E_{t-1} \expebrace{\lambda \Delta Y_t - \psi_G(\lambda) \Delta Y_t^2}
  &\preceq \E_{t-1} \expebrace{\lambda \Delta Y_t
    - \frac{\lambda^2}{2} \Delta Y_t^2} \\
  &\preceq \expebrace{\frac{\lambda^3}{3} \E_{t-1} (\Delta Y_t)_-^3} \\
  &\preceq \expebrace{\psi_G(\lambda) \E_{t-1} (\Delta Y_t)_-^3},
\end{align}
which shows that \eqref{eq:stronger_sub_psi} holds with $U_t = [Y]_t$ and
$V_t = \sum_{i=1}^t \E_{i-1} \abs{\Delta Y_i}^3$.

\subsection{Proof of \cref{th:rectangular_series}}
\label{sec:proof_rectangular_series}

Define the $\Hcal^{d_1 + d_2}$-valued process $(Y_t)$ using the dilation of
$B_t$:
\begin{align}
\Delta Y_t \defineas \epsilon_t \pmat{0 & B_t \\ B_t^\star & 0}.
\end{align}
Since the dilation operation is linear and preserves spectral information, we
have $S_t = \gamma_{\max}(Y_t) = \opnorm{\sum_{i=1}^t \epsilon_i B_i}$
\citep[Eq.~2.12]{tropp_user-friendly_2012}. Furthermore, since each $B_i$ is
fixed and $\epsilon_i$ is 1-sub-Gaussian (in the usual sense for scalar random
variables), $(Y_t)$ satisfies the conditions of \cref{th:sub_psi_lemma} with
$\psi = \psi_N$, $U_t \equiv 0$, and
\begin{align}
W_t = \sum_{i=1}^t \pmat{B_i B_i^\star & 0 \\ 0 & B_i^\star B_i},
\end{align}
by \citet[Lemma 4.3]{tropp_user-friendly_2012}. Hence $(S_t)$ is
$(d_1+d_2)$-sub-Gaussian with variance process
\begin{align}
  V_t = \opnorm{W_t} = \max\ebrace{
    \enorm{\sum_{i=1}^t B_i B_i^\star}_{\text{op}},
    \enorm{\sum_{i=1}^t B_i^\star B_i}_{\text{op}}}.
\end{align}
The result now follows from \cref{th:uniform_chernoff}(b).

\subsection{Proof of \cref{th:late_crossing_simplified}}
\label{sec:proof_late_crossing_simplified}

First, observe $\slope^{-1}(u) = \psi'(\decay(u))$ for any $u \in (0,
\bar{b})$. Indeed, from the definition of $\slope(\cdot)$ and
\cref{th:slope_facts}(v) we see that if $u = \slope(v)$ then
$\decay(u) = \psistarprime(v) = {\psi'}^{-1}(v)$, so that
$v = \psi'(\decay(u))$. This identity will be used below.

Now let $h(b) \defineas m \psi^\star(b) + (x - bm) \psistarprime(b)$. We will
show the following:
\begin{enumerate}[label=(\Roman*)]
\item If $m = 0$ or $b \leq \slope\pfrac{x}{m}$, then
  $h(b) \leq (x - (b \bmin \bar{b})m) \decay(b)$.
\item If $m > 0$ then
  $h(b) \leq m \psi^\star\eparen{\frac{x}{m}} = h\pfrac{x}{m}$.
\end{enumerate}
Together with \cref{th:uniform_chernoff}(d) these prove that
\eqref{eq:late_crossing_approx} holds, and \eqref{eq:line_thru_origin} follows
upon setting $x = bm$.

First suppose $m = 0$, so it suffices to show $\psistarprime(b) \leq \decay(b)$
to prove (I) in this case. But \cref{th:slope_facts}(vi) implies
$u \leq \slope^{-1}(u)$ for any $u \in [0, \bar{b})$, and together with the
convexity of $\psi^\star$, we have
$\psistarprime(b) \leq \psistarprime(\slope^{-1}(b))$. Then the identities
$\slope^{-1}(u) = \psi'(\decay(u))$ and $\psistarprime = {\psi'}^{-1}$ imply
$\psistarprime(\slope^{-1}(b)) = \decay(b)$.

Now suppose $m > 0$. It is easy to see that
$h'(b) = (x - bm) {\psi^\star}''(b)$. The convexity of $\psi^\star$ now implies
$h$ is nondecreasing for $b \leq x/m$ and nonincreasing for $b \geq x/m$. Hence
$h(b)$ is maximized at $b = x/m$, which proves (II). To prove (I) in this case,
we claim that $h(\slope^{-1}(b)) = (x - bm) \decay(b)$. Then the condition
$b \leq \slope(x/m)$ and \cref{th:slope_facts}(vi) imply
$b < \slope^{-1}(b) \leq x/m$, so that $h(b) \leq h(\slope^{-1}(b))$ since $h$
is nondecreasing on this region, which is (I).

To prove the claim, substitute the identity $\slope^{-1}(u) = \psi'(\decay(u))$
into the definition of $h(\cdot)$, yielding
\begin{align}
h(\slope^{-1}(b)) = h(\psi'(\decay(b)))
  &= m \psi^\star(\psi'(\decay(b))) + [x - m \psi'(\decay(b))] \decay(b).
\end{align}
Now use the identity
$\psi^\star(u) = u \psistarprime(u) - \psi(\psistarprime(u))$ to obtain
\begin{align}
h(\slope^{-1}(b)) &= x \decay(b) - m \psi(\decay(b)) \\
  &= x \decay(b) - m b \decay(b),
\end{align}
where the final equality uses \cref{th:slope_facts}(v), proving the claim.

The second statement \eqref{eq:line_thru_origin} follows directly from
\cref{th:uniform_chernoff}(d) with $x = mb$. When $b \leq \bar{b}$,
\cref{th:slope_facts}(vi) implies $s(x/m) \leq x/m = b$, so the second case in
\eqref{eq:late_crossing_bound} applies. When $b > \bar{b}$, we have
$x > m \bar{b}$, so the first case in \eqref{eq:late_crossing_bound}
applies. Noting that $D(b) = \infty = \psi^\star(b)$ in this case using
\cref{th:slope_facts}(i), we see that \eqref{eq:line_thru_origin} remains
valid. \qed

\subsection{Proof of \cref{th:pinelis}}\label{sec:proof_pinelis}

We invoke arguments from \citet{pinelis_optimum_1994} and
\citet{pinelis_approach_1992} to show that \cref{th:canonical_assumption} is
satisfied.

For \textbf{part (a)}, the proofs of Theorem 3 in \citet{pinelis_optimum_1994}
and Theorem 3 in \citet{pinelis_approach_1992} show that, for each $t \in \N$,
\begin{align}
\E_{t-1} \cosh\big( \lambda \Psi(Y_t) \big)
  &\leq e^{\lambda^2 D_\star^2 c_t^2 / 2}
    \cosh\big( \lambda \Psi(Y_{t-1}) \big).
\end{align}
Hence
$L_t \defineas \cosh\big( \lambda \Psi(Y_t) \big) e^{-\lambda^2 D_\star^2
  \sum_{i=1}^t c_i^2 / 2}$ is a supermartingale, and the inequality
$\cosh x > e^x / 2$ implies that \cref{th:canonical_assumption} is satisfied for
$S_t = \Psi(Y_t)$, $V_t = D_\star^2 \sum_{i=1}^t c_i^2$ and $\psi = \psi_N$ with
$\lambda_{\max} = \infty$ and $l_0 = 2$. The conclusion
\eqref{eq:pinelis_hoeffding} follows from replacing $m$ with $D_\star^2 m$ to
make $D_\star^2$ explicit in the bound.

For \textbf{part (b)}, the proof of Theorem 3 in \citet{pinelis_optimum_1994}
shows that
\begin{align}
\E_{t-1} \cosh\big( \lambda \Psi(Y_t) \big)
  &\leq \expebrace{
    D_\star^2 \E_{t-1}\ebracket{
      e^{\lambda \norm{\Delta Y_t}} - \lambda \norm{\Delta Y_t} - 1}
  } \cosh\big( \lambda \Psi(Y_{t-1}) \big) \\
  &\leq \expebrace{
    D_\star^2 \pfrac{e^{c \lambda} - c \lambda - 1}{c^2}
    \E_{t-1} \norm{\Delta Y_t}^2
  } \cosh\big( \lambda \Psi(Y_{t-1}) \big).
\end{align}
using the fact that $(e^{c \lambda} - c \lambda - 1)/c^2$ is
nondecreasing. Hence the process
$L_t \defineas \cosh\big( \lambda \Psi(Y_t) \big) e^{-\psi_P(\lambda) D_\star^2
  \sum_{i=1}^t \E_{i-1} \norm{\Delta Y_i}^2}$ is a supermartingale, and we see that
\cref{th:canonical_assumption} is satisfied for $S_t = \Psi(Y_t)$,
$V_t = D_\star^2 \sum_{i=1}^t \E_{i-1} \norm{\Delta Y_i}^2$ and $\psi = \psi_P$ with
$\lambda_{\max} = \infty$ and $l_0 = 2$. The conclusion
\eqref{eq:pinelis_bennett} follows as in part (a). \qed

\subsection{Proof of \cref{th:iid_invariance}}\label{sec:proof_iid_invariance}

We invoke Theorem 2 of \citet{robbins_boundary_1970} for the sum $S_n / \sigma$
with $g(t) = a/\sigma + b \sigma t$, noting that
\begin{align}
\lim_{m \to \infty} \P\eparen{
  \exists t \in \N: \frac{S_n}{\sqrt{m}} \geq a + \frac{b t \sigma^2}{m}
}
&= \lim_{m \to \infty} \P\eparen{
    \exists t \in \N: \frac{S_n}{\sigma} \geq \sqrt{m} g\pfrac{t}{m}}.
\end{align}
It is easy to verify the conditions of parts (i) and (ii) of Robbins and
Siegmund's theorem, yielding the conclusion
\begin{align}
\lim_{m \to \infty} \P\eparen{
  \exists t \in \N: \frac{S_n}{\sigma} \geq \sqrt{m} g\pfrac{t}{m}}
&= \P\eparen{\exists t \in (0, \infty): B_t \geq g(t)},
\end{align}
where $(B_t)$ is standard Brownian motion. The latter probability is equal to
$e^{-2ab}$ by the standard line-crossing formula for Brownian motion (e.g.,
\citealp{durrett_probability:_2017}, Exercise 7.5.2). \qed

\subsection{Proof of \cref{th:sandwich}}\label{sec:proof_sandwich}

From the definition of $\decay(\cdot)$, we see that
$M_t = \expebrace{\decay(b) \cdot (S_t - b V_t)}$. Since $\tau$ is a stopping
time, $(M_{t \bmin \tau})$ is a martingale, so $1 = \E M_{t \bmin \tau}$ for
each $t \in \N$. The third condition of the proposition ensures that
$M_{t \bmin \tau} \leq e^{\decay(b) \cdot (a + \epsilon)}$ for all $t$ a.s., so
by dominated convergence we have $\E M_{t \bmin \tau} \to \E M_\tau = 1$, where
$M_\tau$ is defined as the a.s.\ limit of $(M_{t \bmin \tau})$, whose existence
is guaranteed since the stopped process is a nonnegative martingale. The second
condition of the proposition implies $M_t \convas 0$, hence
\begin{align}
1 = \E M_\tau
  &= \E M_\tau \indicatorp{\tau < \infty} +
     \E M_\infty \indicatorp{\tau = \infty} \\
  &\leq \expebrace{\decay(b) \cdot (a + \epsilon)} \P(\tau < \infty),
\end{align}
which gives the desired lower bound on $\P(\tau < \infty)$.

\subsection{Proof of \cref{th:continuous_martingale}}
\label{sec:proof_continuous_martingale}

The conclusion follows immediately from \cref{th:sandwich} with $\epsilon = 0$
once we show that the conditions of the proposition are satisfied for $(S_t)$
with $V_t = [S]_t$ and $\psi = \psi_N$.

In this case, since $(S_t)$ has continuous paths a.s, $(M_t)$ is the stochastic
exponential of the process $(\decay(b) S_t)$ \citep[Ch.\ II, Theorem
37]{protter_stochastic_2005}. Kazamaki's criterion is sufficient to ensure
$(M_t)$ is a martingale \citep[Ch.\ III, Theorem
44]{protter_stochastic_2005} and $M_0 = 1$ since $S_0 = 0$. This shows that
condition (1) of \cref{th:sandwich} holds. Condition (3) follows directly from
the continuity of paths of $(S_t)$.

It remains to show that condition (2) holds. For this we express $(S_t)$ as a
time change of Brownian motion \citep[Ch.\ II, Theorem
42]{protter_stochastic_2005}: $S_t = B_{[S]_t}$ where $(B_t)$ is a standard
Brownian motion (with respect to a different filtration). From the law of the
iterated logarithm we know that $B_t / t \convas 0$ as $t \to \infty$, hence
$S_t - b [S]_t = [S]_t (B_{[S_t]} / [S]_t - b) \to -\infty$ since
$[S]_t \uparrow \infty$. \qed

\subsection{Proof of \cref{th:tangent}}\label{sec:proof_tangent}

Lemma 2.4 of \citet{boucheron_concentration_2013} shows that
\begin{align}
f_\alpha(t) = \inf_{\lambda} \ebracket{
  \frac{\log \alpha^{-1}}{\lambda} + \frac{\psi(\lambda)}{\lambda} \cdot t
}, \label{eq:f_alpha_inf}
\end{align}
so that $f_\alpha(t)$ is a pointwise infimum of lines indexed by $\lambda$ with
intercepts $a_\lambda = (\log \alpha^{-1}) / \lambda$ and slopes
${b_\lambda = \psi(\lambda) / \lambda}$. Hence $\decay(b_\lambda) = \lambda$,
and by \cref{th:uniform_chernoff} the crossing probability of each such line is
${e^{-a_\lambda \decay(b_\lambda)} = \alpha}$. Note we have also shown that
$f_\alpha$ is concave. The optimizer $\lambda_\star(t)$ in
\eqref{eq:f_alpha_inf} is the solution in $\lambda$ of
$\lambda \psi'(\lambda) - \psi(\lambda) = (\log \alpha^{-1}) / t$. The left-hand
side of this equation has positive derivative in $\lambda$ by the convexity of
$\psi$, so the map $t \mapsto \lambda_\star(t)$ is injective. Hence the optimum
line $a_{\lambda_\star(m)} + b_{\lambda_\star(m)} t$ is tangent to the curve
$f_\alpha(t)$ at $t = m$.

\section{Acknowledgments}

We thank Victor de la Pe\~{n}a and Joel Tropp for their detailed perspectives on
our work. We also thank Eli Ben-Michael, Lester Mackey, Xinjia Chen, Itai
Kreisler, and referees for helpful comments. Howard thanks Office of Naval
Research (ONR) Grant N00014-15-1-2367. Sekhon thanks ONR grants N00014-17-1-2176
and N00014-15-1-2367.

\setlength{\bibsep}{0pt plus 0.3ex}

\bibliographystyle{agsm}
\bibliography{best_arm}

\appendix

\section{Sharpened pre-factors based on rank}\label{sec:rank_prefactor}

This argument is adapted from \citet{wainwright_high-dimensional_2017}, though
the idea originates in \citet{oliveira_sums_2010}. Suppose the conditions of
\cref{th:sub_psi_lemma} hold and
\begin{align}
  \sup_{t \in \Tcal} \rank(U_t + W_t) \leq r < d, \text{ a.s.}
\end{align}
Since $\Delta U_t \succeq 0$ and $\Delta W_t \succeq 0$ for all $t$, we know
that $\range(U_t + W_t) \subseteq S$ for all $t$ a.s., where $S$ is an
$r$-dimensional subspace. Inequality \eqref{eq:sub_psi} implies that
$\range(Y_t) \subseteq S$ for all $t$ a.s.\ as well. Let $M$ be a $d \times r$
matrix whose columns form an orthonormal basis for this subspace. Then the
$r$-dimensional process $\widetilde{Y}_t \defineas M^\star Y_t M$ has the same
spectrum as $Y_t$ for all $t$ a.s., so we may apply our bounds to
$(\widetilde{Y}_t)$, with $(\widetilde{U}_t)$ and $(\widetilde{W}_t)$ defined
analogously, to obtain bounds with $l_0 = r$. \qed

\section{Relation to the Dubins-Savage inequality}\label{sec:dubins_savage}

The Dubins-Savage inequality \citep{dubins_tchebycheff-like_1965} says that for
any martingale $S_t$ in discrete time with $S_0 = 0$, setting
$V_t = \sum_{i=1}^t \Var_{i-1} (S_t - S_{t-1})$, we have
\begin{align}
  \P\eparen{\exists t \in \N: S_t \geq a + b V_t} \leq \frac{1}{1 + ab}.
  \label{eq:dubins_savage}
\end{align}
The Dubins-Savage inequality may be proved by means similar to ours, invoking
Ville's inequality for a suitable supermartingale. The relationship of our
bounds to the Dubins-Savage inequality is analogous to that between fixed-time
Cram\'er-Chernoff bounds and Chebyshev's inequality. More precisely, the
Dubins-Savage inequality is analogous to Uspensky's one-sided version of
Chebyshev's inequality \citep{uspensky_introduction_1937,
  bennett_probability_1962}:
\begin{align}
\P(X - \E X \geq x) \leq \frac{\Var X}{\Var X + x^2}. \label{eq:uspensky}
\end{align}
Similar to our \cref{th:uniform_chernoff}(b), we may optimize the RHS of
\eqref{eq:dubins_savage} over all lines passing through a point $(m, x)$ to
obtain the equivalent bound
\begin{align}
  \P\eparen{\exists t \in \N: S_t \geq x + \frac{x}{2m}\eparen{V_t - m}}
  \leq \frac{m}{m + x^2 / 4},
\end{align}
recovering Uspensky's inequality \eqref{eq:uspensky} with $x/2$ in place of
$x$. The Dubins-Savage inequality does not recover Uspensky's inequality at the
fixed time $m$---something is necessarily lost in going from a fixed time to a
uniform bound. Compare our \cref{th:uniform_chernoff}(b), which exactly recovers
the fixed-time Cram\'er-Chernoff bound \eqref{eq:basic_chernoff}. For these
exponential bounds, we lose nothing in going from a fixed time to a uniform
bound.

\section{Graphical comparison of $\psi$ functions}

\begin{figure}[h!]
  \centering
  \includegraphics{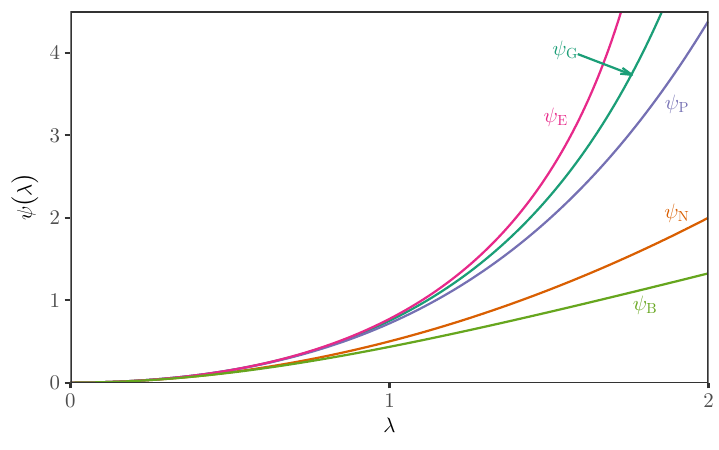}
  \caption{Comparison of $\psi$ functions given in \cref{tab:psi_transforms}. We
    have set $g = h = 1$ in $\psi_B$, $c = 1$ in $\psi_P$, $c=1/3$ in $\psi_G$,
    and $c=1/2$ in $\psi_E$. These are all values that might be used in bounding
    a process with $[-1,1]$-valued increments using the same variance process;
    see \cref{fig:psi_relations} and \cref{th:psi_relations}. In general, bounds
    based on different $\psi$ functions may have different assumptions and
    variance processes, so may not be comparable based on $\psi$ functions
    alone. However, with identical variance processes, a smaller $\psi$ function
    yields a tighter bound. Note all functions behave like
    $\psi_N(\lambda) = \lambda^2 / 2$ near the
    origin. \label{fig:psi_functions}}
\end{figure}

\Cref{fig:psi_functions} illustrates together the five standard $\psi$ functions
discussed in \cref{sec:sufficient_conditions}, to help the reader gain
intuition. With the given parameter settings, the inequalities apparent in the
figure do hold for all $\lambda \geq 0$:
$\psi_B(\lambda) \leq \psi_N(\lambda) \leq \psi_P(\lambda) \leq \psi_G(\lambda)
\leq \psi_E(\lambda)$. See the proof of \cref{th:psi_relations} in
\cref{sec:proof_psi_relations}.

\section{A more general boundary-crossing result}\label{sec:more_general}

The following assumption weakens \cref{th:canonical_assumption}, replacing the
product $\psi(\lambda) \Delta V_t$ with a function $f(\lambda, \Delta V_t)$.

\begin{assumption}\label{th:generalized_assumption}
  Let $(S_t)_{t \in \N \union \brace{0}}$ and
  $(V_t)_{t \in \N \union \brace{0}}$ be two real-valued processes adapted to an
  underlying filtration $(\Fcal_t)_{t \in \N \union \brace{0}}$ with
  $S_0 = V_0 = 0$ a.s.\ and $V_t \geq 0$ a.s.\ for all $t \in \N$. Let
  $f: [0, \lambda_{\max}) \times (0, \infty) \to \R$ be concave in its second
  argument for each value of the first, and let $l_0 \in [1, \infty)$. We
  assume, for each $\lambda \in [0, \lambda_{\max})$, there exists a
  supermartingale $(L_t(\lambda))_{t \in \N \union \brace{0}}$ with respect to
  $(\Fcal_t)$ such that $L_0 \leq l_0$ a.s.\ and
  $\expebrace{\lambda S_t - \sum_{i=1}^t f(\lambda, \Delta V_i)} \leq
  L_t(\lambda)$ a.s.\ for all $t \in \N$.
\end{assumption}

Clearly, when $f(\lambda, v) \equiv \psi(\lambda) \cdot v$ for some $\psi$,
\cref{th:canonical_assumption} holds and \cref{th:uniform_chernoff}
applies. Under the weaker \cref{th:generalized_assumption} we have the following
results:
\begin{theorem}\label{th:generalized_bound}
  If \cref{th:generalized_assumption} holds for some real-valued processes
  $(S_t)$ and $(V_t)$, then for any $\lambda \in [0, \lambda_{\max})$ and
  $a > 0$, we have
\begin{align}
  \P\eparen{
    \exists t \in \N: S_t \geq a + \frac{t f(\lambda, V_t / t)}{\lambda}
  } \leq l_0 e^{-a\lambda}.
\end{align}
Furthermore, if $f_v(\cdot) \defineas f(\cdot, v)$ is CGF-like for each $v > 0$,
then for any $n \in \N$, $m > 0$ and
$0 \leq x < n \sup_\lambda f_{m/n}'(\lambda)$, we have
\begin{align}
  \P\eparen{
    \exists t \leq n : S_t \geq x + \frac{n}{\lambda_\star} \ebracket{
      f\eparen{\lambda_\star, \frac{V_t}{n}}
      - f\eparen{\lambda_\star, \frac{m}{n}}
    }
  } &\leq l_0 \expebrace{-n f_{m/n}^\star\pfrac{x}{n}} \\
  \P\eparen{
    \exists t \in \N :
    S_t \geq x + \frac{tf(\lambda_\star, m/t) - nf(\lambda_\star, m/n)}
                      {\lambda_\star}
  } &\leq l_0 \expebrace{-n f_{m/n}^\star\pfrac{x}{n}}
\end{align}
where $\lambda_\star \defineas (f_{m/n}^\star)'(x/n)$.
\end{theorem}

The proof follows the same principles as that of \cref{th:uniform_chernoff} and
is omitted for brevity. One application of this result is to martingales with
bounded increments, making use of $\psi_B$:

\begin{corollary}
  Let $(Y_t)_{t \in \N}$ be an $\Hcal^d$-valued martingale and let
  $S_t \defineas \gamma_{\max}(Y_t)$. Suppose $\gamma_{\max}(\Delta Y_t) \leq c$
  for all $t$ for some $c > 0$, and let
  $V_t \defineas \gamma_{\max}(\eangle{Y}_t)$. Then for any $x, m > 0, n \in \N$
  we have
\begin{multline}
\P\eparen{
  \exists t \leq n: S_t \geq x + n\ebracket{g\pfrac{V_t}{n} - g\pfrac{m}{n}}
} \\
\leq \ebracket{\pfrac{m}{x+m}^{x+m} \pfrac{n}{n-x}^{n-x}}^{n/(n+m)},
\end{multline}
and
\begin{multline}
\P\eparen{
  \exists t \in \N:
  S_t \geq x + tg\pfrac{m}{t} - ng\pfrac{m}{n}
} \\
\leq \ebracket{\pfrac{m}{x+m}^{x+m} \pfrac{n}{n-x}^{n-x}}^{n/(n+m)},
\end{multline}
where
\begin{align}
g(v) \defineas \frac{m+cn}{n(v+c) \log \xi} \ebracket{
  v \xi^{\frac{cn}{m+cn}} + c \xi^{-\frac{vn}{m+cn}}
} \quad \text{and} \quad \xi \defineas \frac{1+x/m}{1-x/cn}.
\end{align}
This generalizes Theorem 2.1 of \citet{fan_hoeffdings_2012} [B, D].
\end{corollary}

One can further generalize \cref{th:generalized_assumption} by replacing
$\sum_i f(\lambda, \Delta V_i)$ with $\sum_i f_i(\lambda, \Delta V_i)$,
permitting $f_i$ to vary with time, but the added generality further weakens the
conclusions we can draw.

\section{Equivalent sub-exponential conditions}\label{sec:equiv_sub_exp}

Here we show that our sub-exponential condition \eqref{eq:our_sub_exp} is
equivalent to another common definition \eqref{eq:other_sub_exp}
\citep{wainwright_high-dimensional_2017}. We rephrase both conditions for the
right tail of a mean-zero random variable $X$.

\begin{proposition}
  For a zero-mean random variable $X$, the following are equivalent:
  \begin{enumerate}
  \item There exist $\sigma^2 > 0$ and $c > 0$ such that
    \begin{align}
      \log \E e^{\lambda X}
        \leq \frac{\ebracket{-\log(1 - c\lambda) - c\lambda} \sigma^2}{c^2}
      \qquad \text{ for all } \lambda \in [0, 1/c).
      \label{eq:our_sub_exp}
    \end{align}
  \item There exist $\nu > 0$ and $\alpha > 0$ such that
    \begin{align}
      \log \E e^{\lambda X} \leq \frac{\lambda^2 \nu}{2}
      \qquad \text{ for all } \lambda \in [0, 1/\alpha).
      \label{eq:other_sub_exp}
    \end{align}
  \end{enumerate}
\end{proposition}
\begin{proof}
  Suppose the first condition holds. A Taylor expansion of
  $[-\log(1-c\lambda) - c\lambda] / c^2$ about $\lambda = 0$ yields
  \begin{align}
    \frac{\ebracket{-\log(1 - c\lambda) - c\lambda} \sigma^2}{c^2}
      = \frac{\lambda^2\sigma^2}{2}
        + \lambda^2 \sigma^2 \sum_{k=1}^\infty \frac{(c\lambda)^k}{2+k}
      = \frac{\lambda^2 \sigma^2}{2} + o(\lambda^2).
  \end{align}
  So choosing $\nu > \sigma^2$, we can find $\alpha$ sufficiently large to
  ensure that
  \begin{align}
    \frac{\ebracket{-\log(1 - c\lambda) - c\lambda} \sigma^2}{c^2}
      \leq \frac{\lambda^2 \nu}{2}
    \qquad \text{ for all } \lambda \in [0, 1/\alpha),
  \end{align}
  implying the second condition holds.

  Now suppose the second condition holds. Then since $\lambda \geq 0$, the above
  series expansion shows that the first condition holds with $\sigma^2 = \nu$
  and $c = \alpha$.
\end{proof}

Note that if the first condition \eqref{eq:our_sub_exp} applies to both $X$ and
$-X$, then $X$ satisfies the usual, two-tailed sub-exponential condition,
$\log \E e^{\lambda X} \leq \lambda^2 \nu / 2$ for all
$\abs{\lambda} < 1/\alpha$.

\end{document}